\documentclass[12pt]{amsart}

\usepackage{setspace}
\usepackage{amssymb,amsfonts, amsmath, amsthm}
\usepackage{mathtools}
\usepackage{array}
\usepackage[all,arc]{xy}
\usepackage{enumerate}
\usepackage{mathrsfs}
\usepackage{tikz}
\usepackage{tikz-cd}
\usepackage[margin=1in]{geometry}
\usetikzlibrary{decorations.markings}
\usepackage{enumitem}
\usepackage{ytableau,varwidth}
\usepackage{MnSymbol}
\usepackage[enableskew,vcentermath]{youngtab} 
\usepackage[backend=biber]{biblatex}
\usepackage{caption}
\usepackage{subcaption}
\usepackage{multirow}
\usepackage{hyperref}
\usepackage{xurl}
\hypersetup{breaklinks=true}

\usetikzlibrary{calc}%
\usetikzlibrary{shapes}%
\usetikzlibrary{patterns}%
\usetikzlibrary{positioning}%
\usetikzlibrary{knots}
\usetikzlibrary{decorations.markings}
\usetikzlibrary{hobby}
\usepackage{epstopdf}%
\usetikzlibrary { decorations.pathmorphing, decorations.pathreplacing, decorations.shapes }
\pdfpageattr{/Group <</S /Transparency /I true /CS /DeviceRGB>>}

\usetikzlibrary{external}
\newtheorem{thm}{Theorem}[section]
\newtheorem{cor}[thm]{Corollary}
\newtheorem{prop}[thm]{Proposition}
\newtheorem{lem}[thm]{Lemma}

\theoremstyle{definition}
\newtheorem{defn}[thm]{Definition}

\newtheorem{exmp}[thm]{Example}

\theoremstyle{remark}
\newtheorem{rem}[thm]{Remark}

\newcolumntype{M}[1]{>{\centering\arraybackslash}m{#1}}

\newcommand{\F}{\mathbb{F}}

\usepackage[bottom]{footmisc}

\newcommand{\plab}{\textrm{plab}}

\makeatletter
\let\c@equation\c@thm
\makeatother
\numberwithin{equation}{section}
\title{The HOMFLY Polynomial of a Forest Quiver}

\author{Amanda Schwartz}
\address{Department of Mathematics, University of Michigan, 530 Church Street, Ann Arbor, MI 48109-1043, USA}
\email{\href{mailto:amschw@umich.edu}{amschw@umich.edu}}
\thanks{The author was supported by grants DMS-1840234 and DMS-1953852 from the National Science Foundation.}

\begin{document}

\begin{abstract}
We define the HOMFLY polynomial of a forest quiver $Q$ using a recursive definition on the underlying graph of the quiver. We then show that this polynomial is equal to the HOMFLY polynomial of any plabic link which comes from a connected plabic graph whose quiver is $Q$. We also prove a closed-form expression for the HOMFLY polynomial of a forest quiver $Q$ in terms of the independent sets of $Q$.
\end{abstract}
\maketitle

\setcounter{tocdepth}{1}
\tableofcontents

\section{Introduction}
There have been many developments in recent years relating cluster algebras and knot theory. In this work, we will define an invariant of forest quivers and relate it to certain link invariants. This gives one way to connect the study of forest quivers considered up to mutation equivalence and the study of associated links considered up to isotopy. Postnikov's plabic graphs, introduced in \cite{postnikov} to study a stratification of the totally nonnegative Grassmannian into positroid cells, will serve as an intermediate object in establishing this connection.

These plabic graphs and their generalizations have proved useful in establishing several connections between cluster algebras and knot theory. One can associate a plabic link to any plabic graph, as introduced in \cite{Fomin_2022, Shende_2019}. In  \cite{GLplabiclinks} Galashin and Lam studied connections between invariants of these plabic links and invariants of quivers associated to the plabic graphs. In \cite{galashin2024braidvarietyclusterstructures}, the authors defined 3D plabic graphs, generalizations of Postnikov's two-dimensional plabic graphs, and used them to construct a cluster structure on type $A$ braid varieties. The existence of cluster structures on links or related objects such as braid varieties has been the subject of much recent work; see \cite{ casals2024clusterstructuresbraidvarieties, casals2024positroidlinksbraidvarieties, galashin2023braidvarietyclusterstructures, Shende_2019} for several additional examples.

When a plabic graph is reduced, the associated plabic link is a positroid link and is isotopic to several other links one can obtain from objects in bijection with positroids; see \cite{casals2024positroidlinksbraidvarieties, GLplabiclinks} for more details. In this setting, polynomial invariants associated to these links can provide information about the objects associated to the plabic graph $G$. For example, if $G$ is a reduced plabic graph, then there is a rational function $R(Q;q)$ of the quiver $Q$ of $G$ which, after a normalization, yields the point count of the open positroid variety $\Pi^\circ(G)$ over a finite field $\F_q$; see \cite{GLplabiclinks, lam2021cohomologyclustervarietiesii}. There are also many connections relating positroid varieties and the Khovanov-Rozansky homology of their associated positroid links to Catalan combinatorics \cite{galashin2021positroidcatalannumbers, galashin2023positroidsknotsqtcatalannumbers, Mellit_2022, hogancamp2019toruslinkhomology}.

 The link invariants which we will study in this paper are the Alexander polynomial, which is the oldest knot polynomial, and a stronger invariant called the HOMFLY polynomial which specializes to the Alexander polynomial. The Alexander polynomial was discovered in 1928 by J.W. Alexander \cite{Alexander} while the HOMFLY polynomial was introduced in \cite{HOMFLY} and also studied independently in \cite{PT} in the 1980s. Several interesting connections between these link invariants and cluster algebras have been studied. In \cite{baziermatte2024knottheoryclusteralgebras} Bazier-Matte and Schiffler described a way to associate a cluster algebra to any link diagram and related the Alexander polynomial of the link to the $F$-polynomial of modules associated to the cluster algebra. Lee and Schiffler related the Jones polynomial, a different specialization of the HOMFLY polynomial, of a 2-bridge link to a specialization of a cluster variable in \cite{lee2017clusteralgebrasjonespolynomials}. 
 
 From the cluster algebra perspective, there has also been interest in finding and studying polynomial mutation invariants of quivers. For example, in \cite{fomin2024cyclicallyorderedquivers} Fomin and Neville studied invariants of cyclically ordered quivers, including an Alexander polynomial which they defined as the determinant of a matrix associated to the quiver. In this paper, we will focus on forest quivers, i.e. quivers whose underlying graphs are forests. This is a subset of the class of acyclic quivers which includes all type $A_n$, $D_n$, $E_6$, $E_7$, and $E_8$ quivers. We show that given any forest quiver, there is a reduced plabic graph with that forest quiver. Therefore, the set of links which we are studying includes certain positroid links.

We will begin by defining the HOMFLY polynomial of a forest quiver using a recursive definition on the underlying graph of the quiver. This can then be specialized to the Alexander polynomial of a forest quiver, just as the Alexander polynomial of a link is a specialization of the link's HOMFLY polynomial. Since the Alexander and HOMFLY polynomials depend only on the underlying graph of a quiver, we may instead refer to the Alexander or HOMFLY polynomial of a (undirected) forest. Some examples of trees and their Alexander and HOMFLY polynomials are shown in Table \ref{tab:polynomial examples}. 

In Theorem \ref{thm:welldef}, we show that for any plabic link which arises from a connected plabic graph whose quiver $Q_G$ is a forest, the HOMFLY polynomial of the link is the same as the HOMFLY polynomial of $Q_G$. Therefore, while plabic graphs serve as an intermediate object between quivers and plabic links, this result allows one to go directly from a forest quiver to an associated link invariant. While we initially define the HOMFLY polynomial of a forest quiver recursively, we also prove a closed-form formula for the polynomial in Theorem \ref{thm:homflyformula}.

\begin{table}
        \centering
\renewcommand{\arraystretch}{3}
\small
    \begin{tabular}
{M{0.15\textwidth}|M{0.3\textwidth}|M{0.45\textwidth}}
 \hline
 Tree & Alexander Polynomial & HOMFLY Polynomial \\ [0.5ex] 
 \hline\hline
  $A_4$ \hspace{0.15cm} \begin{tikzpicture}[scale = 0.5]
     \node [circle, draw=black, fill=black, scale =0.3] (1) at (0,0.5) {};
     \node [circle, draw=black, fill=black, scale =0.3] (2) at (0.5, 0.5) {};
     \node [circle, draw=black, fill=black, scale =0.3] (3) at (1, 0.5) {};
     \node [circle, draw=black, fill=black, scale =0.3] (4) at (1.5, 0.5) {};
     \draw [thick] (1) -- (2) -- (3) -- (4);
 \end{tikzpicture} & $t^{-2}\left(t^4 - t^3 + t^2 - t + 1\right)$ & $\frac{z^4 + 4z^2 + 3}{a^4} - \frac{z^2 + 2}{a^6}$ \\ 
 \hline
 $D_5$ \hspace{0.15cm} \begin{tikzpicture}[scale = 0.5]
    \node [circle, draw=black, fill=black, scale =0.3] (0) at (0,0.25) {};
     \node [circle, draw=black, fill=black, scale =0.3] (1) at (0,0.75) {};
     \node [circle, draw=black, fill=black, scale =0.3] (2) at (0.5, 0.5) {};
     \node [circle, draw=black, fill=black, scale =0.3] (3) at (1, 0.5) {};
     \node [circle, draw=black, fill=black, scale =0.3] (4) at (1.5, 0.5) {};
     \draw [thick] (1) -- (2) -- (0);
     \draw [thick] (2) -- (3)--(4);
 \end{tikzpicture} & $t^{-5/2}(t^5 - t^4 + t - 1) $ & $\frac{z^5 + 5z^3 + 6z + 2z^{-1}}{a^5} - \frac{z^3 + 4z + 3z^{-1}}{a^7} + \frac{z^{-1}}{{a^9}}$  \\
 \hline
 $E_6$ \hspace{0.15cm} \begin{tikzpicture}[scale = 0.5]
     \node [circle, draw=black, fill=black, scale =0.3] (1) at (0,0.5) {};
     \node [circle, draw=black, fill=black, scale =0.3] (2) at (0.5, 0.5) {};
     \node [circle, draw=black, fill=black, scale =0.3] (3) at (1, 0.5) {};
     \node [circle, draw=black, fill=black, scale =0.3] (4) at (1.5, 0.5) {};
     \node [circle, draw=black, fill=black, scale =0.3] (5) at (2, 0.5) {};
     \node [circle, draw=black, fill=black, scale =0.3] (6) at (1, 1) {};
     \draw [thick] (1)--(2)--(3)--(4)--(5);
     \draw [thick] (3)--(6);
 \end{tikzpicture} & $t^{-3}\left(t^6-t^5+t^3-t+1 \right)$ & $\frac{z^6 + 6z^4 + 10z^2 + 5}{a^6} - \frac{z^4 + 5z^2 + 5}{a^8} + \frac{1}{a^{10}}$  \\
 \hline
  $E_7$ \hspace{0.15cm} \begin{tikzpicture}[scale = 0.5]
     \node [circle, draw=black, fill=black, scale =0.3] (1) at (0,0.5) {};
     \node [circle, draw=black, fill=black, scale =0.3] (2) at (0.5, 0.5) {};
     \node [circle, draw=black, fill=black, scale =0.3] (3) at (1, 0.5) {};
     \node [circle, draw=black, fill=black, scale =0.3] (4) at (1.5, 0.5) {};
     \node [circle, draw=black, fill=black, scale =0.3] (5) at (2, 0.5) {};
     \node [circle, draw=black, fill=black, scale =0.3] (7) at (2.5, 0.5) {};
     \node [circle, draw=black, fill=black, scale =0.3] (6) at (1, 1) {};
     \draw [thick] (1)--(2)--(3)--(4)--(5)--(7);
     \draw [thick] (3)--(6); \end{tikzpicture} & $t^{-7/2}\left(t^7-t^6+t^4-t^3+t-1\right) $ & $\frac{z^7 + 7z^5 + 15z^3 + 11z + 2z^{-1}}{a^7}- \frac{z^5 + 6z^3 + 9z + 3z^{-1}}{a^9} + \frac{z + z^{-1}}{a^{11}}$  \\
 \hline
 $E_8$ \hspace{0.15cm} \begin{tikzpicture}[scale = 0.5]
     \node [circle, draw=black, fill=black, scale =0.3] (1) at (0,0.5) {};
     \node [circle, draw=black, fill=black, scale =0.3] (2) at (0.5, 0.5) {};
     \node [circle, draw=black, fill=black, scale =0.3] (3) at (1, 0.5) {};
     \node [circle, draw=black, fill=black, scale =0.3] (4) at (1.5, 0.5) {};
     \node [circle, draw=black, fill=black, scale =0.3] (5) at (2, 0.5) {};
     \node [circle, draw=black, fill=black, scale =0.3] (7) at (2.5, 0.5) {};
     \node [circle, draw=black, fill=black, scale =0.3] (8) at (3, 0.5) {};
     \node [circle, draw=black, fill=black, scale =0.3] (6) at (1, 1) {};
     \draw [thick] (1)--(2)--(3)--(4)--(5)--(7) -- (8);
     \draw [thick] (3)--(6); \end{tikzpicture} & $t^{-4}\left(t^8-t^7+t^5-t^4+t^3-t+1 \right)$  & $\frac{z^8 + 8z^6 + 21z^4 + 21z^2 + 7}{a^8} - \frac{z^6 + 7z^4 + 14z^2 + 8}{a^{10}} + \frac{z^2 + 2}{a^{12}}$  \\
     \hline
      \begin{tikzpicture}[scale = 0.5]
     \node [circle, draw=black, fill=black, scale =0.3] (1) at (0,0.5) {};
     \node [circle, draw=black, fill=black, scale =0.3] (2) at (0.5, 0.5) {};
     \node [circle, draw=black, fill=black, scale =0.3] (3) at (1, 0.5) {};
     \node [circle, draw=black, fill=black, scale =0.3] (4) at (1.5, 0.5) {};
     \node [circle, draw=black, fill=black, scale =0.3] (5) at (2, 0.5) {};
     \node [circle, draw=black, fill=black, scale =0.3] (8) at (2.5, 0.75) {};
     \node [circle, draw=black, fill=black, scale =0.3] (11) at (2.5, 0.25) {};
     \node [circle, draw=black, fill=black, scale =0.3] (6) at (1, 1) {};
     \node [circle, draw=black, fill=black, scale =0.3] (10) at (1, 0) {};
     \draw [thick] (1)--(2)--(3)--(4)--(5) -- (8);
     \draw [thick] (6)--(3) -- (10); 
     \draw [thick] (11)--(5); 
     \end{tikzpicture} & $\begin{aligned}&t^{-9/2}(t^9 - t^8 - 3t^7 + 7t^6 - 8t^5 \\
     &\quad+ 8t^4 - 7t^3 + 3t^2 + t - 1) \end{aligned}$ & $\frac{z^9 + 9z^7 + 28z^5 + 39z^3 + 28z + 11z^{-1} + 2z^{-3}}{a^9} - \frac{z^7 + 11z^5 + 36z^3 + 47z + 28z^{-1} + 7z^{-3}}{a^{11}} + \frac{5z^3 + 20z + 23z^{-1} + 9z^{-3}}{a^{13}} - \frac{z + 6z^{-1} + 5z^{-3}}{a^{15}} + \frac{z^{-3}}{a^{17}}$ \\ 
     [1ex] 
 \hline
\end{tabular}
    \caption{Examples of trees with their Alexander and HOMFLY polynomials.}
    \label{tab:polynomial examples}
\end{table}

With regard to future work, there are several ways in which our results could possibly be extended. For example, one could attempt to generalize the formula in Theorem \ref{thm:homflyformula} to a wider class of quivers. It would also be interesting to study link homologies like Khovanov-Rozansky homology and knot Floer homology - which categorify the HOMFLY and Alexander polynomials, respectively - of links arising from plabic graphs whose quivers are forests.
\subsection*{Acknowledgements} I thank my advisor, Thomas Lam, for his guidance and comments. I also thank Eugene Gorsky for helpful discussions. 

\section{Preliminaries}
\subsection{Plabic Graphs and Plabic Links}

A \emph{plabic graph} $G$ is a planar, bicolored graph which is embedded in the disk and whose vertices are colored black and white. We assume that $G$ has $N$ vertices on the boundary of the disk which are labeled $1,2,\dots, N$ in a clockwise order, are colored black, and each have degree $1$. A face in $G$ is said to be a \emph{boundary} (resp. \emph{interior}) face if it is (resp. is not) adjacent to the boundary of the disk. A \emph{strand} in $G$ is a path which follows the edges in $G$, obeying the \emph{rules of the road}. That is, the path turns maximally right at each black vertex and maximally left at each white vertex. The \emph{strand permutation} $\pi_G$ of $G$ is a permutation on $N$ obtained by setting $\pi_G(i)=j$ if the strand starting at boundary vertex $i$ ends at boundary vertex $j$. A plabic graph without internal leaves is said to be \emph{reduced} if it has the minimal number of faces among all such plabic graphs with the strand permutation $\pi_G$. See Figure \ref{fig:plabic graph examples} for an example of two plabic graphs, one reduced and one not reduced, which have the same strand permutation. 
\begin{figure}
    \centering
    \begin{tikzpicture}[scale=0.8]
        \draw[dashed] (0,0) circle (2);
        \node [circle, fill=white, draw=black, scale=0.5] (w1) at (-1, 0.75) {};
        \node [circle, fill=black, draw=black, scale=0.5] (b1) at (0, 0.75) {};
        \node [circle, fill=white, draw=black, scale=0.5] (w2) at (1, 0.75) {};
        \node [circle, fill=black, draw=black, scale=0.5] (b2) at (1, -0.75) {};
        \node [circle, fill=white, draw=black, scale=0.5] (w3) at (0, -0.75) {};
        \node [circle, fill=black, draw=black, scale=0.5] (b3) at (-1, -0.75) {};
        \node [circle, fill=black, draw=black, scale=0.2] (1) at (90:2) {};
        \node [above] at (90:2) {$1$};
        \node [circle, fill=black, draw=black, scale=0.2] (2) at (30:2) {};
        \node [right] at (30:2) {$2$};
        \node [circle, fill=black, draw=black, scale=0.2] (3) at (330:2) {};
        \node [right] at (330:2) {$3$};
        \node [circle, fill=black, draw=black, scale=0.2] (4) at (210:2) {};
        \node [left] at (210:2) {$4$};
        \node [circle, fill=black, draw=black, scale=0.2] (5) at (150:2) {};
        \node [left] at (150:2) {$5$};
        \draw[thick] (w1) -- (b1) -- (w2) -- (b2) -- (w3) -- (b3) -- (w1);
        \draw[thick] (b1) -- (w3);
        \draw[thick] (5) -- (w1);
        \draw[thick] (1) -- (b1);
        \draw[thick] (2) -- (w2);
        \draw[thick] (3) -- (b2);
        \draw[thick] (4) -- (b3);
    \end{tikzpicture}\hspace{1in}
    \begin{tikzpicture}[scale=0.8]
        \draw[dashed] (0,0) circle (2);
        \node [circle, fill=white, draw=black, scale=0.5] (w1) at (-1, 0.75) {};
        \node [circle, fill=black, draw=black, scale=0.5] (b1) at (0, 0.75) {};
        \node [circle, fill=white, draw=black, scale=0.5] (w2) at (1, 0.75) {};
        \node [circle, fill=black, draw=black, scale=0.5] (b2) at (1, -0.75) {};
        \node [circle, fill=white, draw=black, scale=0.5] (w3) at (0, -0.75) {};
        \node [circle, fill=black, draw=black, scale=0.5] (b3) at (-1, -0.75) {};
        \node [circle, fill=black, draw=black, scale=0.5] (B1) at (1, 0.25) {};
        \node [circle, fill=black, draw=black, scale=0.5] (B2) at (1, -0.25) {};
        \node [circle, fill=black, draw=black, scale=0.2] (1) at (90:2) {};
        \node [above] at (90:2) {$1$};
        \node [circle, fill=black, draw=black, scale=0.2] (2) at (30:2) {};
        \node [right] at (30:2) {$2$};
        \node [circle, fill=black, draw=black, scale=0.2] (3) at (330:2) {};
        \node [right] at (330:2) {$3$};
        \node [circle, fill=black, draw=black, scale=0.2] (4) at (210:2) {};
        \node [left] at (210:2) {$4$};
        \node [circle, fill=black, draw=black, scale=0.2] (5) at (150:2) {};
        \node [left] at (150:2) {$5$};
        \draw[thick] (w1) -- (b1) -- (w2) -- (B1);
        \draw[thick] (B2) -- (b2) -- (w3) -- (b3) -- (w1);
        \draw[thick] (b1) -- (w3);
        \draw[thick] (5) -- (w1);
        \draw[thick] (1) -- (b1);
        \draw[thick] (2) -- (w2);
        \draw[thick] (3) -- (b2);
        \draw[thick] (4) -- (b3);
        \draw[thick] (B1) to [out=300, in = 60] (B2) to [out = 120, in = 240] (B1);
    \end{tikzpicture}
    \caption{Two plabic graphs, one reduced (left) and one not reduced (right), which both have strand permutation $\pi = (1\; 4\; 2\; 5\; 3)$.}
    \label{fig:plabic graph examples}
\end{figure}
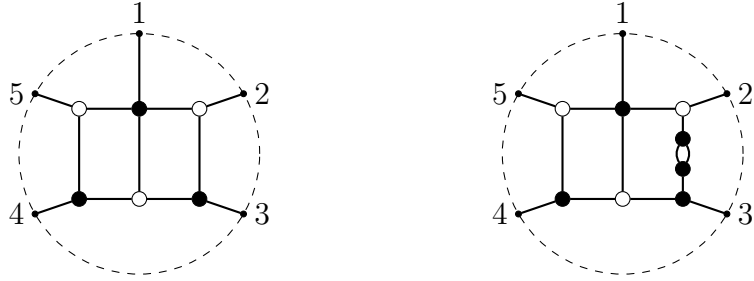

One can associate a plabic link $L_G^\plab$ to a plabic graph $G$ as follows. For more details, including an alternative description in terms of divides, see \cite{Fomin_2022, GLplabiclinks}. Draw all the strands of $G$. When two strands $S_1$ and $S_2$ cross at a point $p$, consider the arguments $\theta_1$ and $\theta_2$ of their tangent vectors at $p$, respectively, when considered in the complex plane. We assume that $0< \theta_1\neq \theta_2 <2\pi$. If $\theta_1$ is greater than (resp. less than) $\theta_2$ then $S_1$ goes under (resp. over) $S_2$. At any point where the tangent vector to a strand has argument 0, the strand must be adjusted as follows. If there is a point $p$ on a strand $S$ where the argument changes from being just below $2\pi$ to being just above $0$ as one travels through $p$ along $S$, then we break $S$ at $p$, sending it to the boundary just before reaching $p$ and then back to continue along its original path after $p$. Along the way to the boundary, it passes under all other strands it crosses, and on the way back to the strand $S$, it passes over all other strands. If the argument instead changes at $p$ from being just above 0 to being just below $2\pi$, the procedure is the same except the strand crosses above all other strands when heading to the boundary and under all others when returning from the boundary. This is demonstrated in Figure \ref{fig:arg0}.
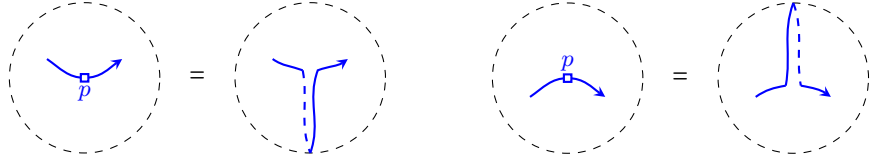
\begin{figure}
    \centering
    \begin{tikzpicture}
        \draw[dashed] (0,0) circle (1);
        \node at (1.5,0) {$=$};
        \draw[dashed] (3,0) circle (1);
        \draw[thick, blue, -stealth] (-0.5, 0.25) to [out=325, in = 180] (0,0) to [out=0, in =215] (0.5, 0.25);
        \draw[thick, fill=white, draw=blue] (-0.05,-0.05) rectangle ++(0.1, 0.1);
        \node[blue] at (0, -0.2) {\scriptsize$p$};
        \draw[thick, blue] (2.5, 0.25) to [out=325, in = 160] (2.9, 0.1);
        \draw[thick, blue, dashed] (2.9, 0.1) to [out=280, in = 110] (3, -1); 
        \draw[thick, blue, -stealth] (3, -1) to [out=70, in = 250] (3.1, 0.1) to [out=20, in=215] (3.5, 0.25); 
    \end{tikzpicture}\hspace{0.5in}
    \begin{tikzpicture}
        \draw[dashed] (0,0) circle (1);
        \node at (1.5,0) {$=$};
        \draw[dashed] (3,0) circle (1);
        \draw[thick, blue, -stealth] (-0.5, -0.25) to [out=35, in = 180] (0,0) to [out=0, in =145] (0.5, -0.25);
        \draw[thick, fill=white, draw=blue] (-0.05,-0.05) rectangle ++(0.1, 0.1);
        \node[blue] at (0, 0.2) {\scriptsize$p$};
        \draw[thick, blue] (2.5, -0.25) to [out=35, in = 190] (2.9,-0.1);
        \draw[thick, blue] (2.9,-0.1) to [out=80, in = 250] (3,1);
        \draw[thick, blue, dashed] (3,1) to [out=290, in = 100] (3.1, -0.1);
        \draw[thick, blue, -stealth] (3.1, -0.1) to [out=350, in =145] (3.5, -0.25);
    \end{tikzpicture}
    \caption{Modifications at a point $p$ on a strand where the tangent vector has argument 0.}
    \label{fig:arg0}
\end{figure}Note that we sometimes mark such points with rectangles. Taking the union of all these strands after these adjustments and connecting strands which start and end at the same boundary vertex gives a link diagram for $L_G^\plab$. For an example, see Figure \ref{fig:link example} for the plabic link of the leftmost graph in Figure \ref{fig:plabic graph examples}.
\begin{figure}
    \centering
    \begin{tikzpicture}[scale=1.2]
        \draw[dashed] (0,0) circle (2);
        \node [circle, fill=white, draw=black, scale=0.5] (w1) at (-1, 0.75) {};
        \node [circle, fill=black, draw=black, scale=0.5] (b1) at (0, 0.75) {};
        \node [circle, fill=white, draw=black, scale=0.5] (w2) at (1, 0.75) {};
        \node [circle, fill=black, draw=black, scale=0.5] (b2) at (1, -0.75) {};
        \node [circle, fill=white, draw=black, scale=0.5] (w3) at (0, -0.75) {};
        \node [circle, fill=black, draw=black, scale=0.5] (b3) at (-1, -0.75) {};
        \node [circle, fill=black, draw=black, scale=0.2] (1) at (90:2) {};
        \node [above] at (90:2) {$1$};
        \node [circle, fill=black, draw=black, scale=0.2] (2) at (30:2) {};
        \node [right] at (30:2) {$2$};
        \node [circle, fill=black, draw=black, scale=0.2] (3) at (330:2) {};
        \node [right] at (330:2) {$3$};
        \node [circle, fill=black, draw=black, scale=0.2] (4) at (210:2) {};
        \node [left] at (210:2) {$4$};
        \node [circle, fill=black, draw=black, scale=0.2] (5) at (150:2) {};
        \node [left] at (150:2) {$5$};
        \draw[thick] (w1) -- (b1) -- (w2) -- (b2) -- (w3) -- (b3) -- (w1);
        \draw[thick] (b1) -- (w3);
        \draw[thick] (5) -- (w1);
        \draw[thick] (1) -- (b1);
        \draw[thick] (2) -- (w2);
        \draw[thick] (3) -- (b2);
        \draw[thick] (4) -- (b3);
    \begin{knot}[clip width=4, consider self intersections, end tolerance=0.03pt, flip crossing=3, flip crossing=5, flip crossing=6]
    \strand[blue, thick] (0.5, 0.8) to [out=20, in = 110] (30:1.9) to [out=290, in =80] (1.2, 0.5) to [out=260, in=20] (0.6, -0.9) to [out=200, in = 0] (0, -1) to [out=180, in=340] (-0.6, -0.9) to [out= 160, in = 280] (-1.2, 0.5) to [out=100, in=250] (150:1.9) to [out=70, in = 160] (-0.5,0.8) to [out=340, in=100] (-0.1, -0.3) to [out=280, in=160] (0.9, -1) to [out=340, in = 250] (330:1.9) to [out=70, in = 280] (0.9, -0.2) to [out=100, in = 340] (0.4, 0.9) to [out=160, in =0] (90:1.9) to [out=180, in = 20] (-0.4, 0.9) to [out=200, in =80] (-0.9,-0.2) to [out=260, in=110] (210:1.9) to [out=290, in=200] (-0.9, -1) to [out=20, in=260] (0.1, -0.3) to [out=80, in=200] (0.5, 0.8);
    \end{knot}
     \draw[blue, fill=white] (-1.45,-1.2) rectangle ++(0.15cm, 0.15cm);
     \draw[blue, fill=white] (1.45,-1.2) rectangle ++(-0.15cm, 0.15cm);
     \draw[blue, fill=white] (1.45,1.05) rectangle ++(-0.15cm, 0.15cm);
     \draw[blue, fill=white] (-1.45,1.05) rectangle ++(0.15cm, 0.15cm);
     \draw[blue, thick, -stealth] (0, -1) -- (-0.1, -1);
     \draw[blue, thick, stealth-] (-0.25, 1.05) -- (-0.23, 1.15);
     \draw[blue, thick, -stealth]  (0.23, 1.15) -- (0.22, 1.25);
     \draw[blue, thick, stealth-] (-0.145, 0.05) -- (-0.152, 0.15);
     \draw[blue, thick, -stealth] (0.145, 0.15) -- (0.152, 0.25);
     \node [circle, fill=black, draw=black, scale=0.5] (b2) at (1, -0.75) {};
        \node [circle, fill=white, draw=black, scale=0.5] (w3) at (0, -0.75) {};
        \node [circle, fill=black, draw=black, scale=0.5] (b3) at (-1, -0.75) {};
    \end{tikzpicture}
    \caption{A plabic graph $G$ and its plabic link $L_G^\plab$.}
    \label{fig:link example}
\end{figure}
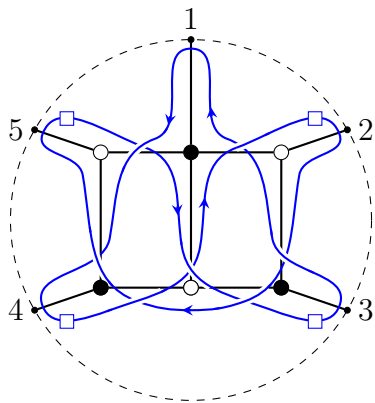

We will allow for certain local moves on plabic graphs, pictured in Figure \ref{fig:local moves}.
\begin{figure}
    \centering
    \begin{tikzpicture}[scale=0.8]
        \node at (-1.5, 0) {(a)};
        \node[circle, scale =0.4, fill=white, draw=black] (w1) at (-0.5,0.5) {};
        \node[circle, scale =0.4, fill=black, draw=black] (b1) at (0.5,0.5) {};
        \node[circle, scale =0.4, fill=white, draw=black] (w2) at (0.5,-0.5) {};
        \node[circle, scale =0.4, fill=black, draw=black] (b2) at (-0.5,-0.5) {};
        \draw[thick] (w1) -- (b1) -- (w2) -- (b2) -- (w1);
        \draw[thick] (-0.75,0.75) -- (w1);
        \draw[thick] (0.75,0.75) -- (b1);
        \draw[thick] (0.75,-0.75) -- (w2);
        \draw[thick] (-0.75,-0.75) -- (b2);
        \draw [stealth-stealth, gray] (1.5,0) -- (2.5,0);
        \node[circle, scale =0.4, fill=black, draw=black] (w1) at (3.5,0.5) {};
        \node[circle, scale =0.4, fill=white, draw=black] (b1) at (4.5,0.5) {};
        \node[circle, scale =0.4, fill=black, draw=black] (w2) at (4.5,-0.5) {};
        \node[circle, scale =0.4, fill=white, draw=black] (b2) at (3.5,-0.5) {};
        \draw[thick] (w1) -- (b1) -- (w2) -- (b2) -- (w1);
        \draw[thick] (3.25,0.75) -- (w1);
        \draw[thick] (4.75,0.75) -- (b1);
        \draw[thick] (4.75,-0.75) -- (w2);
        \draw[thick] (3.25,-0.75) -- (b2);
    \end{tikzpicture}\\
    \vspace{0.25in}
    \begin{tikzpicture}[scale=0.8]
        \node at (-1.5, 0) {(b)};
        \node[circle, scale=0.4, fill=white, draw=black] (w1) at (0,0) {};
        \node[circle, scale=0.4, fill=white, draw=black] (w2) at (1,0) {};
        \draw[thick] (w1) -- (w2);
        \draw[thick] (-0.5, 0.5)-- (w1) -- (-0.75, 0);
        \draw[thick] (-0.5, -0.5)-- (w1);
        \draw[thick] (1.5, 0.5)-- (w2) -- (1.5, -0.5);
        \draw [stealth-stealth, gray] (1.9,0) -- (2.9,0);
        \node[circle, scale=0.4, fill=white, draw=black] (w) at (4,0) {};
        \draw[thick] (3.5, 0.5)-- (w) -- (3.4, 0);
        \draw[thick] (3.5, -0.5)-- (w);
        \draw[thick] (4.5, 0.5)-- (w) -- (4.5, -0.5);
    \end{tikzpicture}\hspace{0.5in}
    \begin{tikzpicture}[scale=0.8]
        \node[circle, scale=0.4, fill=black, draw=black] (w1) at (0,0) {};
        \node[circle, scale=0.4, fill=black, draw=black] (w2) at (1,0) {};
        \draw[thick] (w1) -- (w2);
        \draw[thick] (-0.5, 0.5)-- (w1) -- (-0.75, 0);
        \draw[thick] (-0.5, -0.5)-- (w1);
        \draw[thick] (1.5, 0.5)-- (w2) -- (1.5, -0.5);
        \draw [stealth-stealth, gray] (1.9,0) -- (2.9,0);
        \node[circle, scale=0.4, fill=black, draw=black] (w) at (4,0) {};
        \draw[thick] (3.5, 0.5)-- (w) -- (3.4, 0);
        \draw[thick] (3.5, -0.5)-- (w);
        \draw[thick] (4.5, 0.5)-- (w) -- (4.5, -0.5);
    \end{tikzpicture}\vspace{0.25in}
    \begin{tikzpicture}[scale=0.8]
        \node at (-1.5, 0) {(c)};
        \node[circle, scale=0.4, fill=black, draw=black] (w) at (0,0) {};
        \draw[thick] (-1,0) -- (w) -- (1,0);
        \draw [stealth-stealth, gray] (1.5,0) -- (2.5,0);
        \draw[thick] (3.5,0) -- (4.5,0);
        \draw [stealth-stealth, gray] (5.5,0) -- (6.5,0);
        \node[circle, scale=0.4, fill=white, draw=black] (w) at (8,0) {};
        \draw[thick] (7,0) -- (w) -- (9,0);
    \end{tikzpicture}\vspace{0.25in}
    \begin{tikzpicture}[scale=0.8]
        \node at (-1.5, 0) {(d)};
        \node[circle, scale=0.4, fill=black, draw=black] (b) at (0,0) {};
        \node[circle, scale=0.2, fill=black, draw=black] (boundary1) at (0,-1) {};
        \node[circle, scale=0.2, fill=black, draw=black] (boundary2) at (8,-1) {};
        \draw[thick] (-1,0) -- (b) -- (1,0);
        \draw[thick] (b) -- (boundary1);
        \draw[dashed] (-1,-1) -- (1,-1);
        \draw [stealth-stealth, gray] (1.5,0) -- (2.5,0);
        \draw[thick] (3.5,0) -- (4.5,0);
        \draw[dashed] (3.5,-1) -- (4.5,-1);
        \draw [stealth-stealth, gray] (5.5,0) -- (6.5,0);
        \node[circle, scale=0.4, fill=white, draw=black] (w) at (8,0) {};
        \draw[thick] (7,0) -- (w) -- (9,0);
        \draw[thick] (w) -- (boundary2);
        \draw[dashed] (7,-1) -- (9,-1);
    \end{tikzpicture}
    \caption{Local moves on plabic graphs: (a) square move, (b) contraction/uncontraction, (c) middle vertex insertion/removal, and (d) tail addition/removal.}
    \label{fig:local moves}
\end{figure}
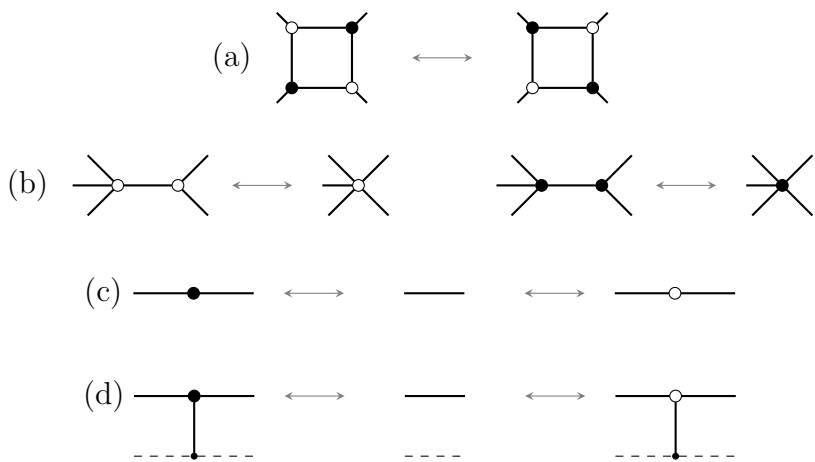
  Applying local moves (a) - (c) does not affect the property of being reduced since they do not change the strand permutation nor the number of faces in the graph. We will refer to the set of all plabic graphs obtained from a plabic graph $G$ using moves (a) - (c) as the \emph{move-equivalence class} of $G$. A graph is reduced if and only if no graph in its move-equivalence class contains either of the following, pictured in Figure \ref{fig:nonreduced}:
  \begin{enumerate}[label=(\roman*)]
      \item two trivalent vertices of opposite color which are connected by two edges, or
      \item an interior leaf which is connected to another interior vertex of the opposite color and degree at least 3.
  \end{enumerate}
 
  \begin{figure}
      \centering
      \begin{tikzpicture}
         \begin{scope}
          \node at (-1.5, 0) {(i)};
             \node[circle, scale =0.5, fill=white, draw=black] (w) at (-0.5,0) {};
             \node[circle, scale =0.5, fill=black, draw=black] (b) at (0.5,0) {};

             \draw[thick] (-1,0) -- (w) to [out=70, in = 110] (b) -- (1,0);
             \draw[thick] (b) to [out=250, in=290] (w);
         \end{scope} 
         
         \begin{scope}[xshift = 4cm]
          \node at (-1.5, 0) {(ii)};
             \node[circle, scale =0.5, fill=white, draw=black] (w) at (-0.5,0) {};
             \node[circle, scale =0.5, fill=black, draw=black] (b) at (0.5,0) {};

             \draw[thick] (-1,0.5) -- (w) -- (-1, -0.5);
             \draw[thick] (w)--(b);
         \end{scope} 
         \begin{scope}[xshift = 6.5cm]
          \node at (-1.5, 0) {,};
             \node[circle, scale =0.5, fill=black, draw=black] (b) at (-0.5,0) {};
             \node[circle, scale =0.5, fill=white, draw=black] (w) at (0.5,0) {};

             \draw[thick] (-1,0.5) -- (b) -- (-1, -0.5);
             \draw[thick] (w)--(b);
         \end{scope} 
      \end{tikzpicture}
      \caption{A plabic graph is reduced if and only if no graph in its move-equivalence class contains either configurations of type (i) or (ii).}
      \label{fig:nonreduced}
  \end{figure}
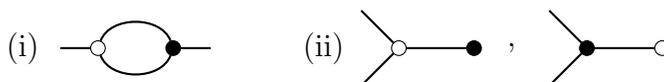
  
  All four local moves (a) - (d) result in isotopic plabic links and therefore will not affect any of the polynomial invariants we study in this paper. The \emph{bipartite reduction} of a plabic graph $G$ refers to the result of using local move (b) to contract all edges whose endpoints are interior vertices of the same color. The \emph{tail reduction} of a plabic graph $G$ refers to the graph which results from applying tail removal to $G$ until it is no longer possible to do so. An interior face $F$ of $G$ is said to be a \emph{boundary leaf face} if in the tail reduction of $G$ its boundary consists of two vertices of different colors which are connected by two edges, one of which separates $F$ from a boundary face and one of which separates $F$ from an interior face.

One can associate a directed, planar graph $Q_G$ to each plabic graph $G$ as follows. There will be one vertex $v_F$ placed inside each interior face $F$ of $G$. An edge is placed between vertices $v_F$ and $v_{F'}$ for each edge $e$ in $G$ with opposite colored endpoints such that $F$ and $F'$ are adjacent to $e$. The edge in $Q_G$ is oriented so that as one travels along it, the white vertex in $e$ is to the left. If this graph $Q_G$ contains no loops or 2-cycles, i.e. it is a quiver, then $G$ is said to be \emph{simple}. In this case, we may refer to $Q_G$ as the quiver associated to $G$. Throughout the rest of this paper, we will work only with simple plabic graphs $G$ since we will assume that $Q_G$ is a forest quiver. Note that local moves (b) - (d) will not change the quiver of a plabic graph.

If $G$ is a connected, simple plabic graph, then $G$ can be assumed to be trivalent, i.e. we can assume that all internal vertices have degree 3. Local move (b) can be used to uncontract any interior vertex of degree higher than 3 into trivalent vertices while local move (c) can be used to remove degree 2 vertices. If $v$ is an interior leaf which is adjacent to a vertex $\tilde{v}$ of the same color, local move (b) can be used to contract the edge containing $\tilde{v}$ and $v$ into a single vertex. If on the other hand $v$ is an interior leaf which is adjacent to a vertex $\tilde{v}$ of the opposite color, then the face $F$ containing $v$ would need to be a boundary face. If it were an interior face, then there would be an edge in $Q_G$ which crosses the edge between $v$ and $\tilde{v}$ and connects $v_F \in Q_G$ to itself, contradicting the assumption that $Q_G$ has no loops. Let $G' = G-\{v\}$. Since $v$ is contained in a boundary face, $Q_G=Q_{G'}$. The assumption that $G$ is connected removes the possibility that $v$ is adjacent to a boundary vertex in a different connected component from other internal vertices of $G$. Therefore, one can also verify that $P_{L_G}^\plab=P_{L_{G'}}^\plab$. Therefore, for our purposes, one can replace $G$ with $G'$. After applying these local moves and removing internal leaves as necessary, we see that $G$ can be assumed to be trivalent.

When we restrict our setting to trivalent graphs, we will use local move (b) to refer instead to the following move\\
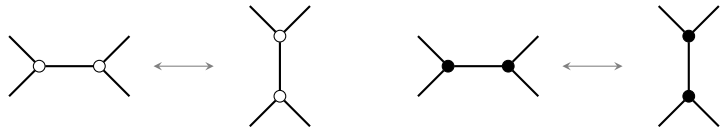
 \begin{center}
     \begin{tikzpicture}[scale=0.8]
        \node[circle, scale=0.4, fill=white, draw=black] (w1) at (0,0) {};
        \node[circle, scale=0.4, fill=white, draw=black] (w2) at (1,0) {};
        \draw[thick] (w1) -- (w2);
        \draw[thick] (-0.5, 0.5)-- (w1);
        \draw[thick] (-0.5, -0.5)-- (w1);
        \draw[thick] (1.5, 0.5)-- (w2) -- (1.5, -0.5);
        \draw [stealth-stealth, gray] (1.9,0) -- (2.9,0);
        \node[circle, scale=0.4, fill=white, draw=black] (w1) at (4,0.5) {};
        \node[circle, scale=0.4, fill=white, draw=black] (w2) at (4,-0.5) {};
        \draw[thick] (3.5, 1)-- (w1);
        \draw[thick] (3.5, -1)-- (w2);
        \draw[thick] (4.5, 1)-- (w1) --(w2)-- (4.5, -1);
    \end{tikzpicture}\hspace{0.5in}
    \begin{tikzpicture}[scale=0.8]
        \node[circle, scale=0.4, fill=black, draw=black] (w1) at (0,0) {};
        \node[circle, scale=0.4, fill=black, draw=black] (w2) at (1,0) {};
        \draw[thick] (w1) -- (w2);
        \draw[thick] (-0.5, 0.5)-- (w1);
        \draw[thick] (-0.5, -0.5)-- (w1);
        \draw[thick] (1.5, 0.5)-- (w2) -- (1.5, -0.5);
        \draw [stealth-stealth, gray] (1.9,0) -- (2.9,0);
        \node[circle, scale=0.4, fill=black, draw=black] (w1) at (4,0.5) {};
        \node[circle, scale=0.4, fill=black, draw=black] (w2) at (4,-0.5) {};
        \draw[thick] (3.5, 1)-- (w1);
        \draw[thick] (3.5, -1)-- (w2);
        \draw[thick] (4.5, 1)-- (w1) --(w2)-- (4.5, -1);
    \end{tikzpicture}
 \end{center}
 which follows from two applications of move (b) in Figure \ref{fig:local moves}.

Our focus in this paper will be on plabic links arising from simple plabic graphs $G$ whose quivers $Q_G$ are orientations of forests. Given any forest quiver $Q$, it is possible to find a plabic graph whose quiver is $Q$. In fact, one can choose such a graph to be reduced so that the resulting link is a positroid link. It was known to Lam and Speyer that one can find a reduced plabic graph whose quiver is $Q$ for any tree quiver $Q$, as mentioned in \cite{GLplabiclinks}, but their proof has not been published. We describe one way to construct such a graph in the following proposition. Then, we describe how to use this result to construct a reduced plabic graph $G$ with $Q_G=Q$ for any forest quiver $Q$.
\begin{prop}\label{prop:exists reduced}
    Let $Q$ be a tree quiver. Then there exists a connected, reduced plabic graph $G$ with $Q_G=Q$.
\end{prop}
\begin{proof}
  We will describe one algorithm to find such a graph $G$. To each edge $e$ in $Q$, draw an edge in $G$ perpendicularly across it. The endpoints of the edge in $G$ should be colored so that the black (resp. white) vertex is on the right (resp. left) as one travels along $e$ in the direction indicated by its orientation. Then, connect these edges drawn in $G$ as they appear if one travels clockwise around each vertex of degree at least two in $Q$. At the leaves of $Q$, form a bipartite face of degree six by adding four extra vertices and connecting those four to the boundary. This process is pictured in the first two steps in Figure \ref{fig:constructing graph}. At this point, this intermediate graph $G'$ is a plabic graph with $Q_{G'}=Q$, but it will not necessarily be reduced. 

  We will build a reduced graph $G$ from $G'$ which still satisfies $Q_G=Q$ as follows. Given any edge $e$ in $G'$ which separates a boundary face from an interior face that does not correspond to a leaf in $Q$, we will add vertices which are connected to new boundary vertices. In particular, given any such edge which connects two vertices of the same color, add one vertex of the opposite color in the middle of the edge and connect it to a new boundary vertex. To any such edge which connects vertices of opposite colors, we add two vertices in the middle of the edge, placed so that the induced subgraph on these four vertices is bipartite. Then we connect these two newly added vertices to the boundary. See Figure \ref{fig:constructing graph} for the result of this procedure for the given example. 
  \begin{figure}
      \centering
      \begin{tikzpicture}
      \begin{scope}
          \node [fill=red, scale =0.5] (1) at (0,1) {};
     \node [fill=red, scale =0.5] (2) at (1, 1) {};
     \node [fill=red, scale =0.5] (3) at (2, 1) {};
     \node [fill=red, scale =0.5] (4) at (3, 1) {};
     \node [fill=red,scale =0.5] (5) at (4, 1) {};
     \node [fill=red, scale =0.5] (6) at (2, 2) {};
     \draw [thick, red, Stealth-] (1)--(2);
     \draw [thick, red, Stealth-] (2)--(3);
     \draw [thick, red, -Stealth] (3)--(4);
     \draw [thick, red, Stealth-] (4)--(5);
     \draw [thick,red, Stealth-] (3)--(6);
      \end{scope}
      \begin{scope}[xshift=5cm, yshift=1cm]
          \draw[thick, -Stealth, gray] (-0.5, 0) -- (0.5,0);
      \end{scope}
      \begin{scope}[xshift = 6cm]
          \node [fill=red, scale =0.5] (1) at (0,1) {};
     \node [fill=red, scale =0.5] (2) at (1, 1) {};
     \node [fill=red, scale =0.5] (3) at (2, 1) {};
     \node [fill=red, scale =0.5] (4) at (3, 1) {};
     \node [fill=red,scale =0.5] (5) at (4, 1) {};
     \node [fill=red, scale =0.5] (6) at (2, 2.5) {};
     \draw [thick, red, Stealth-] (1)--(2);
     \draw [thick, red, Stealth-] (2)--(3);
     \draw [thick, red, -Stealth] (3)--(4);
     \draw [thick, red, Stealth-] (4)--(5);
     \draw [thick,red, Stealth-] (3)--(6);
     \node [scale =0.5, circle, draw=black, fill=black] (1b) at (0.5, 1.5) {};
     \node [scale =0.5, circle, draw=black, fill=white] (1w) at (0.5,0.5) {};
     \node [scale =0.5, circle, draw=black, fill=black] (2b) at (1.5, 1.5) {};
     \node [scale =0.5, circle, draw=black, fill=white] (2w) at (1.5,0.5) {};
     \node [scale =0.5, circle, draw=black, fill=black] (3b) at (1.75, 2) {};
     \node [scale =0.5, circle, draw=black, fill=white] (3w) at (2.25,2) {};
     \node [scale =0.5, circle, draw=black, fill=black] (4b) at (2.5, 0.5) {};
     \node [scale =0.5, circle, draw=black, fill=white] (4w) at (2.5,1.5) {};
     \node [scale =0.5, circle, draw=black, fill=black] (5b) at (3.5, 1.5) {};
     \node [scale =0.5, circle, draw=black, fill=white] (5w) at (3.5,0.5) {};
     \draw[thick] (1b) -- (1w);
     \draw[thick] (2b) -- (2w);
     \draw[thick] (3b) -- (3w);
     \draw[thick] (4b) -- (4w);
     \draw[thick] (5b) -- (5w);
      \end{scope}
      
      \begin{scope}[xshift = 5cm, yshift = -0.5cm]
          \draw[thick, -Stealth, gray] (1, 0.5) -- (-1,-0.5);
      \end{scope}
      \begin{scope}[xshift = -1cm, yshift = -5cm]
         \node [fill=red, scale =0.5] (1) at (0,1) {};
     \node [fill=red, scale =0.5] (2) at (1, 1) {};
     \node [fill=red, scale =0.5] (3) at (2, 1) {};
     \node [fill=red, scale =0.5] (4) at (3, 1) {};
     \node [fill=red,scale =0.5] (5) at (4, 1) {};
     \node [fill=red, scale =0.5] (6) at (2, 2.5) {};
     \draw [thick, red, Stealth-] (1)--(2);
     \draw [thick, red, Stealth-] (2)--(3);
     \draw [thick, red, -Stealth] (3)--(4);
     \draw [thick, red, Stealth-] (4)--(5);
     \draw [thick,red, Stealth-] (3)--(6);
     \node [scale =0.5, circle, draw=black, fill=black] (1b) at (0.5, 1.5) {};
     \node [scale =0.5, circle, draw=black, fill=white] (1w) at (0.5,0.5) {};
     \node [scale =0.5, circle, draw=black, fill=black] (2b) at (1.5, 1.5) {};
     \node [scale =0.5, circle, draw=black, fill=white] (2w) at (1.5,0.5) {};
     \node [scale =0.5, circle, draw=black, fill=black] (3b) at (1.75, 2) {};
     \node [scale =0.5, circle, draw=black, fill=white] (3w) at (2.25,2) {};
     \node [scale =0.5, circle, draw=black, fill=black] (4b) at (2.5, 0.5) {};
     \node [scale =0.5, circle, draw=black, fill=white] (4w) at (2.5,1.5) {};
     \node [scale =0.5, circle, draw=black, fill=black] (5b) at (3.5, 1.5) {};
     \node [scale =0.5, circle, draw=black, fill=white] (5w) at (3.5,0.5) {};
     
     \node [scale =0.5, circle, draw=black, fill=white] (l1w1) at (0, 1.75) {};
     \node [scale =0.5, circle, draw=black, fill=black] (l1b1) at (-0.5,1.5) {};
     \node [scale =0.5, circle, draw=black, fill=white] (l1w2) at (-0.5, 0.5) {};
     \node [scale =0.5, circle, draw=black, fill=black] (l1b2) at (0,0.25) {};
     
     \node [scale =0.5, circle, draw=black, fill=white] (l3w1) at (1.5, 2.5) {};
     \node [scale =0.5, circle, draw=black, fill=black] (l3b1) at (1.75,3) {};
     \node [scale =0.5, circle, draw=black, fill=white] (l3w2) at (2.25, 3) {};
     \node [scale =0.5, circle, draw=black, fill=black] (l3b2) at (2.5,2.5) {};

     \node [scale =0.5, circle, draw=black, fill=white] (l5w1) at (4, 1.75) {};
     \node [scale =0.5, circle, draw=black, fill=black] (l5b1) at (4.5, 1.5) {};
     \node [scale =0.5, circle, draw=black, fill=white] (l5w2) at (4.5,0.5) {};
     \node [scale =0.5, circle, draw=black, fill=black] (l5b2) at (4,0.25) {};
     
     \draw[thick] (1b) -- (1w);
     \draw[thick] (2b) -- (2w);
     \draw[thick] (3b) -- (3w);
     \draw[thick] (4b) -- (4w);
     \draw[thick] (5b) -- (5w);
     \draw[thick] (l1b2) --(l1w2) -- (l1b1) -- (l1w1) -- (1b) -- (2b) -- (3b) --(l3w1) -- (l3b1) -- (l3w2) -- (l3b2) -- (3w) -- (4w) -- (5b) -- (l5w1) -- (l5b1) -- (l5w2) -- (l5b2) -- (5w) -- (4b) -- (2w) -- (1w) -- (l1b2);
     \node [circle, scale =0.2, fill=black, draw=black] (boundary1) at ($(150:3)+(2,1)$) {};
     \node [circle, scale =0.2, fill=black, draw=black] (boundary2) at ($(170:3)+(2,1)$) {};
     \node [circle, scale =0.2, fill=black, draw=black] (boundary3) at ($(190:3)+(2,1)$) {};
     \node [circle, scale =0.2, fill=black, draw=black] (boundary4) at ($(210:3)+(2,1)$) {};
     \draw[thick] (l1w1) -- (boundary1);
     \draw[thick] (l1b1) -- (boundary2);
     \draw[thick] (l1w2) -- (boundary3);
     \draw[thick] (l1b2) -- (boundary4);

     \node [circle, scale =0.2, fill=black, draw=black] (boundary5) at ($(30:3)+(2,1)$) {};
     \node [circle, scale =0.2, fill=black, draw=black] (boundary6) at ($(10:3)+(2,1)$) {};
     \node [circle, scale =0.2, fill=black, draw=black] (boundary7) at ($(350:3)+(2,1)$) {};
     \node [circle, scale =0.2, fill=black, draw=black] (boundary8) at ($(330:3)+(2,1)$) {};
     \draw[thick] (l5w1) -- (boundary5);
     \draw[thick] (l5b1) -- (boundary6);
     \draw[thick] (l5w2) -- (boundary7);
     \draw[thick] (l5b2) -- (boundary8);

     \node [circle, scale =0.2, fill=black, draw=black] (boundary9) at ($(120:3)+(2,1)$) {};
     \node [circle, scale =0.2, fill=black, draw=black] (boundary10) at ($(100:3)+(2,1)$) {};
     \node [circle, scale =0.2, fill=black, draw=black] (boundary11) at ($(80:3)+(2,1)$) {};
     \node [circle, scale =0.2, fill=black, draw=black] (boundary12) at ($(60:3)+(2,1)$) {};
     \draw[thick] (l3w1) -- (boundary9);
     \draw[thick] (l3b1) -- (boundary10);
     \draw[thick] (l3w2) -- (boundary11);
     \draw[thick] (l3b2) -- (boundary12);
     \draw[dashed] (2, 1) circle (3cm);

     \node () at (2, -2.5) {$G'$};
      \end{scope}

      \begin{scope}[xshift=5cm, yshift=-4cm]
          \draw[thick, -Stealth, gray] (-0.5, 0) -- (0.5,0);
      \end{scope}

      \begin{scope}[xshift = 7cm, yshift = -5cm]
         \node [fill=red, scale =0.5] (1) at (0,1) {};
     \node [fill=red, scale =0.5] (2) at (1, 1) {};
     \node [fill=red, scale =0.5] (3) at (2, 1) {};
     \node [fill=red, scale =0.5] (4) at (3, 1) {};
     \node [fill=red,scale =0.5] (5) at (4, 1) {};
     \node [fill=red, scale =0.5] (6) at (2, 2.5) {};
     \draw [thick, red, Stealth-] (1)--(2);
     \draw [thick, red, Stealth-] (2)--(3);
     \draw [thick, red, -Stealth] (3)--(4);
     \draw [thick, red, Stealth-] (4)--(5);
     \draw [thick,red, Stealth-] (3)--(6);
     \node [scale =0.5, circle, draw=black, fill=black] (1b) at (0.5, 1.5) {};
     \node [scale =0.5, circle, draw=black, fill=white] (1w) at (0.5,0.5) {};
     \node [scale =0.5, circle, draw=black, fill=black] (2b) at (1.5, 1.5) {};
     \node [scale =0.5, circle, draw=black, fill=white] (2w) at (1.5,0.5) {};
     \node [scale =0.5, circle, draw=black, fill=black] (3b) at (1.75, 2) {};
     \node [scale =0.5, circle, draw=black, fill=white] (3w) at (2.25,2) {};
     \node [scale =0.5, circle, draw=black, fill=black] (4b) at (2.5, 0.5) {};
     \node [scale =0.5, circle, draw=black, fill=white] (4w) at (2.5,1.5) {};
     \node [scale =0.5, circle, draw=black, fill=black] (5b) at (3.5, 1.5) {};
     \node [scale =0.5, circle, draw=black, fill=white] (5w) at (3.5,0.5) {};
     
     \node [scale =0.5, circle, draw=black, fill=white] (l1w1) at (0, 1.75) {};
     \node [scale =0.5, circle, draw=black, fill=black] (l1b1) at (-0.5,1.5) {};
     \node [scale =0.5, circle, draw=black, fill=white] (l1w2) at (-0.5, 0.5) {};
     \node [scale =0.5, circle, draw=black, fill=black] (l1b2) at (0,0.25) {};
     
     \node [scale =0.5, circle, draw=black, fill=white] (l3w1) at (1.5, 2.5) {};
     \node [scale =0.5, circle, draw=black, fill=black] (l3b1) at (1.75,3) {};
     \node [scale =0.5, circle, draw=black, fill=white] (l3w2) at (2.25, 3) {};
     \node [scale =0.5, circle, draw=black, fill=black] (l3b2) at (2.5,2.5) {};

     \node [scale =0.5, circle, draw=black, fill=white] (l5w1) at (4, 1.75) {};
     \node [scale =0.5, circle, draw=black, fill=black] (l5b1) at (4.5, 1.5) {};
     \node [scale =0.5, circle, draw=black, fill=white] (l5w2) at (4.5,0.5) {};
     \node [scale =0.5, circle, draw=black, fill=black] (l5b2) at (4,0.25) {};
     
     \draw[thick] (1b) -- (1w);
     \draw[thick] (2b) -- (2w);
     \draw[thick] (3b) -- (3w);
     \draw[thick] (4b) -- (4w);
     \draw[thick] (5b) -- (5w);
     \draw[thick] (l1b2) --(l1w2) -- (l1b1) -- (l1w1) -- (1b) -- (2b) -- (3b) --(l3w1) -- (l3b1) -- (l3w2) -- (l3b2) -- (3w) -- (4w) -- (5b) -- (l5w1) -- (l5b1) -- (l5w2) -- (l5b2) -- (5w) -- (4b) -- (2w) -- (1w) -- (l1b2);
     \node [circle, scale =0.2, fill=black, draw=black] (boundary1) at ($(150:3)+(2,1)$) {};
     \node [circle, scale =0.2, fill=black, draw=black] (boundary2) at ($(170:3)+(2,1)$) {};
     \node [circle, scale =0.2, fill=black, draw=black] (boundary3) at ($(190:3)+(2,1)$) {};
     \node [circle, scale =0.2, fill=black, draw=black] (boundary4) at ($(210:3)+(2,1)$) {};
     \draw[thick] (l1w1) -- (boundary1);
     \draw[thick] (l1b1) -- (boundary2);
     \draw[thick] (l1w2) -- (boundary3);
     \draw[thick] (l1b2) -- (boundary4);

     \node [circle, scale =0.2, fill=black, draw=black] (boundary5) at ($(30:3)+(2,1)$) {};
     \node [circle, scale =0.2, fill=black, draw=black] (boundary6) at ($(10:3)+(2,1)$) {};
     \node [circle, scale =0.2, fill=black, draw=black] (boundary7) at ($(350:3)+(2,1)$) {};
     \node [circle, scale =0.2, fill=black, draw=black] (boundary8) at ($(330:3)+(2,1)$) {};
     \draw[thick] (l5w1) -- (boundary5);
     \draw[thick] (l5b1) -- (boundary6);
     \draw[thick] (l5w2) -- (boundary7);
     \draw[thick] (l5b2) -- (boundary8);

     \node [circle, scale =0.2, fill=black, draw=black] (boundary9) at ($(120:3)+(2,1)$) {};
     \node [circle, scale =0.2, fill=black, draw=black] (boundary10) at ($(100:3)+(2,1)$) {};
     \node [circle, scale =0.2, fill=black, draw=black] (boundary11) at ($(80:3)+(2,1)$) {};
     \node [circle, scale =0.2, fill=black, draw=black] (boundary12) at ($(60:3)+(2,1)$) {};
     \draw[thick] (l3w1) -- (boundary9);
     \draw[thick] (l3b1) -- (boundary10);
     \draw[thick] (l3w2) -- (boundary11);
     \draw[thick] (l3b2) -- (boundary12);
     \draw[dashed] (2, 1) circle (3cm);

     \node [scale =0.5, circle, draw=black, fill=white] (e1) at (1, 1.5) {};
     \node [scale =0.5, circle, draw=black, fill=black] (e2) at (1, 0.5) {};
     \node [scale =0.5, circle, draw=black, fill=white] (e3) at (1.625, 1.75) {};
     \node [scale =0.5, circle, draw=black, fill=black] (e4) at (2.375, 1.75) {};
     \node [scale =0.5, circle, draw=black, fill=black] (e5) at (2.833, 1.5) {};
     \node [scale =0.5, circle, draw=black, fill=white] (e6) at (3.1667, 1.5) {};
     \node [scale =0.5, circle, draw=black, fill=white] (e7) at (2.833, 0.5) {};
     \node [scale =0.5, circle, draw=black, fill=black] (e8) at (3.1667, 0.5) {};
     \node [scale =0.5, circle, draw=black, fill=black] (e9) at (1.833, 0.5) {};
     \node [scale =0.5, circle, draw=black, fill=white] (e10) at (2.1667, 0.5) {};

     \node [circle, scale =0.2, fill=black, draw=black] (boundary13) at ($(140:3)+(2,1)$) {};
     \node [circle, scale =0.2, fill=black, draw=black] (boundary14) at ($(240:3)+(2,1)$) {};
     \node [circle, scale =0.2, fill=black, draw=black] (boundary15) at ($(130:3)+(2,1)$) {};
     \node [circle, scale =0.2, fill=black, draw=black] (boundary16) at ($(50:3)+(2,1)$) {};
     \node [circle, scale =0.2, fill=black, draw=black] (boundary17) at ($(40:3)+(2,1)$) {};
     \node [circle, scale =0.2, fill=black, draw=black] (boundary18) at ($(35:3)+(2,1)$) {};
     \node [circle, scale =0.2, fill=black, draw=black] (boundary19) at ($(300:3)+(2,1)$) {};
     \node [circle, scale =0.2, fill=black, draw=black] (boundary20) at ($(310:3)+(2,1)$) {};
     \node [circle, scale =0.2, fill=black, draw=black] (boundary21) at ($(260:3)+(2,1)$) {};
     \node [circle, scale =0.2, fill=black, draw=black] (boundary22) at ($(280:3)+(2,1)$) {};
     
     \draw[thick] (e1) -- (boundary13);
     \draw[thick] (e2) -- (boundary14);
     \draw[thick] (e3) -- (boundary15);
     \draw[thick] (e4) -- (boundary16);
     \draw[thick] (e5) -- (boundary17);
     \draw[thick] (e6) -- (boundary18);
     \draw[thick] (e7) -- (boundary19);
     \draw[thick] (e8) -- (boundary20);
     \draw[thick] (e9) -- (boundary21);
     \draw[thick] (e10) -- (boundary22);

     \node () at (2, -2.5) {$G$};
      \end{scope}
      \end{tikzpicture}
      \caption{A procedure for finding a reduced plabic graph with a given tree quiver $Q$.}
      \label{fig:constructing graph}
  \end{figure}
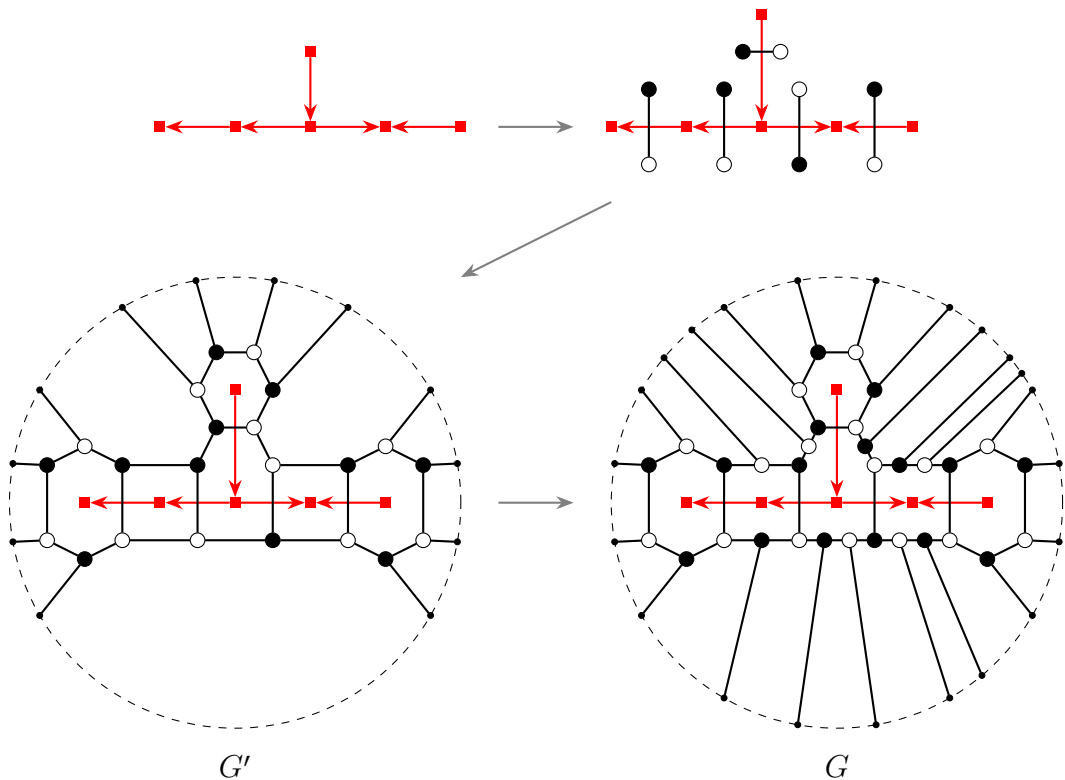
  To see that the result is reduced, first observe that the graph does not contain any of the forbidden configurations pictured in Figure \ref{fig:nonreduced}. It remains to check that none of the graphs in the move-equivalence class of $G$ contain these configurations either. Each interior face in $G$ has degree at least six. Therefore, local move (a) cannot be applied directly to $G$. One can check that all all graphs in the move-equivalence class of $G$ which are obtained by local moves (b) and (c) look like $G$ up to the addition of degree two vertices along existing edges or repeated creation of interior leaves which are the same color as the vertex they are adjacent to. Local move (a) still cannot be applied to these graphs, so these graphs constitute the entire move-equivalence class of $G$. None of these graphs contain any of the configurations in Figure \ref{fig:nonreduced}, so $G$ is reduced.
\end{proof}
\begin{rem}
    In general, the procedure described in the proof of Proposition \ref{prop:exists reduced} will not result in a plabic graph with the minimal number of boundary vertices out of all reduced plabic graphs with a given tree quiver. For example, see Figure \ref{fig:tworeduced} for two reduced graphs whose quiver is $A_2$, one of which comes from the procedure described above and another with fewer boundary vertices.
\end{rem}

\begin{figure}
    \centering
    \begin{tikzpicture}
    \begin{scope}
    \draw[dashed] (0,0) circle (2);
        \node [circle, fill=white, draw=black, scale=0.5] (w1) at (-1, 0.75) {};
        \node [circle, fill=black, draw=black, scale=0.5] (b1) at (0, 0.75) {};
        \node [circle, fill=white, draw=black, scale=0.5] (w2) at (1, 0.75) {};
        \node [circle, fill=black, draw=black, scale=0.5] (b2) at (1, -0.75) {};
        \node [circle, fill=white, draw=black, scale=0.5] (w3) at (0, -0.75) {};
        \node [circle, fill=black, draw=black, scale=0.5] (b3) at (-1, -0.75) {};
        \node [circle, fill=black, draw=black, scale=0.2] (1) at (90:2) {};
        \node [circle, fill=black, draw=black, scale=0.2] (2) at (30:2) {};
        \node [circle, fill=black, draw=black, scale=0.2] (3) at (330:2) {};
        \node [circle, fill=black, draw=black, scale=0.2] (4) at (210:2) {};
        \node [circle, fill=black, draw=black, scale=0.2] (5) at (150:2) {};
        
        \draw[thick] (w1) -- (b1) -- (w2) -- (b2) -- (w3) -- (b3) -- (w1);
        \draw[thick] (b1) -- (w3);
        \draw[thick] (5) -- (w1);
        \draw[thick] (1) -- (b1);
        \draw[thick] (2) -- (w2);
        \draw[thick] (3) -- (b2);
        \draw[thick] (4) -- (b3);
        \node [fill=red, scale = 0.7] (v1) at (-0.5,0) {};
        \node [fill=red, scale = 0.7] (v2) at (0.5,0) {};
        \draw[thick, red, -Stealth] (v2) -- (v1);
    \end{scope}

    \begin{scope}[xshift = 6cm]
    \draw[dashed] (0,0) circle (2);
        \node [circle, fill=white, draw=black, scale=0.5] (w1) at (-0.75, 1) {};
        \node [circle, fill=black, draw=black, scale=0.5] (b1) at (0, 0.75) {};
        \node [circle, fill=white, draw=black, scale=0.5] (w2) at (0.75, 1) {};
        \node [circle, fill=black, draw=black, scale=0.5] (b2) at (1.25, 0.5) {};
        \node [circle, fill=white, draw=black, scale=0.5] (w3) at (1.25, -0.5) {};
        \node [circle, fill=black, draw=black, scale=0.5] (b3) at (0.75, -1) {};
         \node [circle, fill=white, draw=black, scale=0.5] (w4) at (0, -0.75) {};
        \node [circle, fill=black, draw=black, scale=0.5] (b4) at (-0.75, -1) {};
         \node [circle, fill=white, draw=black, scale=0.5] (w5) at (-1.25, -0.5) {};
        \node [circle, fill=black, draw=black, scale=0.5] (b5) at (-1.25, 0.5) {};

        \node [circle, fill=black, draw=black, scale=0.2] (1) at (120:2) {};
        \node [circle, fill=black, draw=black, scale=0.2] (2) at (60:2) {};
        \node [circle, fill=black, draw=black, scale=0.2] (3) at (20:2) {};
        \node [circle, fill=black, draw=black, scale=0.2] (4) at (340:2) {};
        \node [circle, fill=black, draw=black, scale=0.2] (5) at (300:2) {};
        \node [circle, fill=black, draw=black, scale=0.2] (6) at (240:2) {};
        \node [circle, fill=black, draw=black, scale=0.2] (7) at (200:2) {};
        \node [circle, fill=black, draw=black, scale=0.2] (8) at (160:2) {};

        \draw[thick] (w1) -- (b1) -- (w2) -- (b2) -- (w3) -- (b3) -- (w4) -- (b4) -- (w5) -- (b5) -- (w1);
        \draw[thick] (b1) -- (w4);
        \draw[thick] (w1) -- (1);
        \draw[thick] (w2) -- (2);
        \draw[thick] (b2) -- (3);
        \draw[thick] (w3) -- (4);
        \draw[thick] (b3) -- (5);
        \draw[thick] (b4) -- (6);
        \draw[thick] (w5) -- (7);
        \draw[thick] (b5) -- (8);
        
        \node [fill=red, scale = 0.7] (v1) at (-0.5,0) {};
        \node [fill=red, scale = 0.7] (v2) at (0.5,0) {};
        \draw[thick, red, -Stealth] (v2) -- (v1);
    \end{scope}

    \end{tikzpicture}
    \caption{Two reduced graphs with $A_2$ as their quiver.}
    \label{fig:tworeduced}
\end{figure}
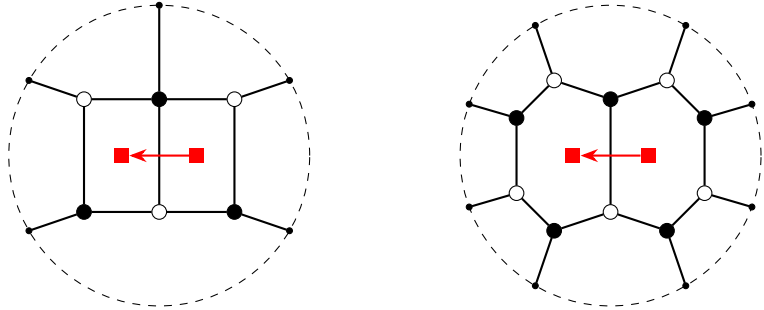

\begin{cor}
  Let $Q$ be a forest quiver. Then there exists a connected, reduced plabic graph $G$ with $Q_G=Q$.  
\end{cor}
\begin{proof}
    Suppose $Q = Q_1 \sqcup Q_2 \sqcup \dots \sqcup Q_k$ where $Q_i$ is a non-empty tree quiver for $i=1,\dots, k$. Let $G_i$ be the reduced plabic graph with quiver $Q_i$ which results from the algorithm in the proof of Proposition \ref{prop:exists reduced}. Place the graphs $G_1, G_2, \dots, G_k$ clockwise in order in a larger disk. Then, for $i=1,\dots, k-1$, pick a pair of edges $e_i$ and $e_{i+1}$ where $e_i$ (resp. $e_{i+1}$) is an edge in $G_i$ (resp. $G_{i+1}$) which has a boundary vertex $b_i$ (resp. $b_{i+1}$) as one of its endpoints such that $b_i$ and $b_{i+1}$ are adjacent boundary vertices. Place a vertex on $e_i$ and a vertex on $e_{i+1}$, and connect these two vertices with an edge. Let $G$ be the result of this process. 
    
    Since there were no paths connecting vertices in $G_i$ and $G_{i+1}$ before the $i$-th step of the above procedure, the addition of this edge in this step does not create any additional interior faces. Therefore, once again local move (a) cannot be applied directly to $G$. The move-equivalence class of $G$ consists of graphs which differ from it by additional of degree two vertices, the repeated creation of leaves which are the same color as the vertex they are adjacent to, or potentially contraction of an edge that has endpoints of the same color and which separates two boundary faces. None of these graphs contain the configurations in Figure \ref{fig:nonreduced}, so $G$ is reduced.
\end{proof}

\begin{exmp}
    See Figure \ref{fig:forest reduced} for an example of a reduced plabic graph whose quiver is the pictured orientation of $E_6 \sqcup A_2$.
\end{exmp}

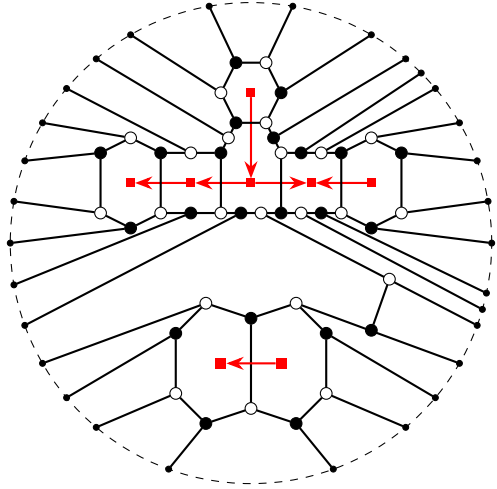
\begin{figure}
    \centering
    \begin{tikzpicture}[scale=0.8]
        \begin{scope}
         \node [fill=red, scale =0.4] (1) at (0,2) {};
     \node [fill=red, scale =0.4] (2) at (1, 2) {};
     \node [fill=red, scale =0.4] (3) at (2, 2) {};
     \node [fill=red, scale =0.4] (4) at (3, 2) {};
     \node [fill=red,scale =0.4] (5) at (4, 2) {};
     \node [fill=red, scale =0.4] (6) at (2, 3.5) {};
     \draw [thick, red, Stealth-] (1)--(2);
     \draw [thick, red, Stealth-] (2)--(3);
     \draw [thick, red, -Stealth] (3)--(4);
     \draw [thick, red, Stealth-] (4)--(5);
     \draw [thick,red, Stealth-] (3)--(6);
     \node [scale =0.4, circle, draw=black, fill=black] (1b) at (0.5, 2.5) {};
     \node [scale =0.4, circle, draw=black, fill=white] (1w) at (0.5,1.5) {};
     \node [scale =0.4, circle, draw=black, fill=black] (2b) at (1.5, 2.5) {};
     \node [scale =0.4, circle, draw=black, fill=white] (2w) at (1.5,1.5) {};
     \node [scale =0.4, circle, draw=black, fill=black] (3b) at (1.75, 3) {};
     \node [scale =0.4, circle, draw=black, fill=white] (3w) at (2.25,3) {};
     \node [scale =0.4, circle, draw=black, fill=black] (4b) at (2.5, 1.5) {};
     \node [scale =0.4, circle, draw=black, fill=white] (4w) at (2.5,2.5) {};
     \node [scale =0.4, circle, draw=black, fill=black] (5b) at (3.5, 2.5) {};
     \node [scale =0.4, circle, draw=black, fill=white] (5w) at (3.5,1.5) {};
     
     \node [scale =0.4, circle, draw=black, fill=white] (l1w1) at (0, 2.75) {};
     \node [scale =0.4, circle, draw=black, fill=black] (l1b1) at (-0.5,2.5) {};
     \node [scale =0.4, circle, draw=black, fill=white] (l1w2) at (-0.5, 1.5) {};
     \node [scale =0.4, circle, draw=black, fill=black] (l1b2) at (0,1.25) {};
     
     \node [scale =0.4, circle, draw=black, fill=white] (l3w1) at (1.5, 3.5) {};
     \node [scale =0.4, circle, draw=black, fill=black] (l3b1) at (1.75,4) {};
     \node [scale =0.4, circle, draw=black, fill=white] (l3w2) at (2.25, 4) {};
     \node [scale =0.4, circle, draw=black, fill=black] (l3b2) at (2.5,3.5) {};

     \node [scale =0.4, circle, draw=black, fill=white] (l5w1) at (4, 2.75) {};
     \node [scale =0.4, circle, draw=black, fill=black] (l5b1) at (4.5, 2.5) {};
     \node [scale =0.4, circle, draw=black, fill=white] (l5w2) at (4.5,1.5) {};
     \node [scale =0.4, circle, draw=black, fill=black] (l5b2) at (4,1.25) {};
     
     \draw[thick] (1b) -- (1w);
     \draw[thick] (2b) -- (2w);
     \draw[thick] (3b) -- (3w);
     \draw[thick] (4b) -- (4w);
     \draw[thick] (5b) -- (5w);
     \draw[thick] (l1b2) --(l1w2) -- (l1b1) -- (l1w1) -- (1b) -- (2b) -- (3b) --(l3w1) -- (l3b1) -- (l3w2) -- (l3b2) -- (3w) -- (4w) -- (5b) -- (l5w1) -- (l5b1) -- (l5w2) -- (l5b2) -- (5w) -- (4b) -- (2w) -- (1w) -- (l1b2);
     \node [circle, scale =0.2, fill=black, draw=black] (boundary1) at ($(150:4)+(2,1)$) {};
     \node [circle, scale =0.2, fill=black, draw=black] (boundary2) at ($(160:4)+(2,1)$) {};
     \node [circle, scale =0.2, fill=black, draw=black] (boundary3) at ($(170:4)+(2,1)$) {};
     \node [circle, scale =0.2, fill=black, draw=black] (boundary4) at ($(180:4)+(2,1)$) {};
     \draw[thick] (l1w1) -- (boundary1);
     \draw[thick] (l1b1) -- (boundary2);
     \draw[thick] (l1w2) -- (boundary3);
     \draw[thick] (l1b2) -- (boundary4);

     \node [circle, scale =0.2, fill=black, draw=black] (boundary5) at ($(30:4)+(2,1)$) {};
     \node [circle, scale =0.2, fill=black, draw=black] (boundary6) at ($(20:4)+(2,1)$) {};
     \node [circle, scale =0.2, fill=black, draw=black] (boundary7) at ($(10:4)+(2,1)$) {};
     \node [circle, scale =0.2, fill=black, draw=black] (boundary8) at ($(0:4)+(2,1)$) {};
     \draw[thick] (l5w1) -- (boundary5);
     \draw[thick] (l5b1) -- (boundary6);
     \draw[thick] (l5w2) -- (boundary7);
     \draw[thick] (l5b2) -- (boundary8);

     \node [circle, scale =0.2, fill=black, draw=black] (boundary9) at ($(120:4)+(2,1)$) {};
     \node [circle, scale =0.2, fill=black, draw=black] (boundary10) at ($(100:4)+(2,1)$) {};
     \node [circle, scale =0.2, fill=black, draw=black] (boundary11) at ($(80:4)+(2,1)$) {};
     \node [circle, scale =0.2, fill=black, draw=black] (boundary12) at ($(60:4)+(2,1)$) {};
     \draw[thick] (l3w1) -- (boundary9);
     \draw[thick] (l3b1) -- (boundary10);
     \draw[thick] (l3w2) -- (boundary11);
     \draw[thick] (l3b2) -- (boundary12);

     \node [scale =0.4, circle, draw=black, fill=white] (e1) at (1, 2.5) {};
     \node [scale =0.4, circle, draw=black, fill=black] (e2) at (1, 1.5) {};
     \node [scale =0.4, circle, draw=black, fill=white] (e3) at (1.625, 2.75) {};
     \node [scale =0.4, circle, draw=black, fill=black] (e4) at (2.375, 2.75) {};
     \node [scale =0.4, circle, draw=black, fill=black] (e5) at (2.833, 2.5) {};
     \node [scale =0.4, circle, draw=black, fill=white] (e6) at (3.1667, 2.5) {};
     \node [scale =0.4, circle, draw=black, fill=white] (e7) at (2.833, 1.5) {};
     \node [scale =0.4, circle, draw=black, fill=black] (e8) at (3.1667, 1.5) {};
     \node [scale =0.4, circle, draw=black, fill=black] (e9) at (1.833, 1.5) {};
     \node [scale =0.4, circle, draw=black, fill=white] (e10) at (2.1667, 1.5) {};

     \node [circle, scale =0.2, fill=black, draw=black] (boundary13) at ($(140:4)+(2,1)$) {};
     \node [circle, scale =0.2, fill=black, draw=black] (boundary14) at ($(190:4)+(2,1)$) {};
     \node [circle, scale =0.2, fill=black, draw=black] (boundary15) at ($(130:4)+(2,1)$) {};
     \node [circle, scale =0.2, fill=black, draw=black] (boundary16) at ($(50:4)+(2,1)$) {};
     \node [circle, scale =0.2, fill=black, draw=black] (boundary17) at ($(45:4)+(2,1)$) {};
     \node [circle, scale =0.2, fill=black, draw=black] (boundary18) at ($(40:4)+(2,1)$) {};
     \node [circle, scale =0.2, fill=black, draw=black] (boundary19) at ($(344:4)+(2,1)$) {};
     \node [circle, scale =0.2, fill=black, draw=black] (boundary20) at ($(348:4)+(2,1)$) {};
     \node [circle, scale =0.2, fill=black, draw=black] (boundary21) at ($(200:4)+(2,1)$) {};
     \node [circle, scale =0.2, fill=black, draw=black] (boundary22) at ($(340:4)+(2,1)$) {};
     
     \draw[thick] (e1) -- (boundary13);
     \draw[thick] (e2) -- (boundary14);
     \draw[thick] (e3) -- (boundary15);
     \draw[thick] (e4) -- (boundary16);
     \draw[thick] (e5) -- (boundary17);
     \draw[thick] (e6) -- (boundary18);
     \draw[thick] (e7) -- (boundary19);
     \draw[thick] (e8) -- (boundary20);
     \draw[thick] (e9) -- (boundary21);
     \draw[thick] (e10) -- (boundary22);

      \end{scope}

      \begin{scope}
        \node [circle, fill=white, draw=black, scale=0.4] (w1) at (1.25, 0) {};
        \node [circle, fill=black, draw=black, scale=0.4] (b1) at (2, -0.25) {};
        \node [circle, fill=white, draw=black, scale=0.4] (w2) at (2.75, 0) {};
        \node [circle, fill=black, draw=black, scale=0.4] (b2) at (3.25, -0.5) {};
        \node [circle, fill=white, draw=black, scale=0.4] (w3) at (3.25, -1.5) {};
        \node [circle, fill=black, draw=black, scale=0.4] (b3) at (2.75, -2) {};
         \node [circle, fill=white, draw=black, scale=0.4] (w4) at (2, -1.75) {};
        \node [circle, fill=black, draw=black, scale=0.4] (b4) at (1.25, -2) {};
         \node [circle, fill=white, draw=black, scale=0.4] (w5) at (0.75, -1.5) {};
        \node [circle, fill=black, draw=black, scale=0.4] (b5) at (0.75, -0.5) {};

        \node [circle, fill=white, draw=black, scale=0.4] (aw) at (4.3, 0.4) {};
        \node [circle, fill=black, draw=black, scale=0.4] (ab) at (4, -0.45) {};

        \node [circle, fill=black, draw=black, scale=0.2] (1) at ($(210:4)+(2,1)$) {};
        \node [circle, fill=black, draw=black, scale=0.2] (2) at ($(330:4)+(2,1)$) {};
        \node [circle, fill=black, draw=black, scale=0.2] (3) at ($(320:4)+(2,1)$) {};
        \node [circle, fill=black, draw=black, scale=0.2] (4) at ($(310:4)+(2,1)$) {};
        \node [circle, fill=black, draw=black, scale=0.2] (5) at ($(290:4)+(2,1)$) {};
        \node [circle, fill=black, draw=black, scale=0.2] (6) at ($(250:4)+(2,1)$) {};
        \node [circle, fill=black, draw=black, scale=0.2] (7) at ($(230:4)+(2,1)$) {};
        \node [circle, fill=black, draw=black, scale=0.2] (8) at ($(220:4)+(2,1)$) {};

        \draw[thick] (w1) -- (b1) -- (w2) -- (b2) -- (w3) -- (b3) -- (w4) -- (b4) -- (w5) -- (b5) -- (w1);
        \draw[thick] (b1) -- (w4);
        \draw[thick] (w1) -- (1);
        \draw[thick] (w2) -- (2);
        \draw[thick] (b2) -- (3);
        \draw[thick] (w3) -- (4);
        \draw[thick] (b3) -- (5);
        \draw[thick] (b4) -- (6);
        \draw[thick] (w5) -- (7);
        \draw[thick] (b5) -- (8);
        
        \draw[thick] (aw) -- (ab);
        
        \node [fill=red, scale = 0.5] (v1) at (1.5,-1) {};
        \node [fill=red, scale = 0.5] (v2) at (2.5,-1) {};
        \draw[thick, red, -Stealth] (v2) -- (v1);
    \end{scope}
      \draw[dashed] (2, 1) circle (4cm);
    \end{tikzpicture}
    \caption{A reduced plabic graph with the pictured forest quiver.}
    \label{fig:forest reduced}
\end{figure}

\subsection{The HOMFLY and Alexander Polynomials}
The Alexander polynomial $\Delta(t)$, named after its discoverer J.W. Alexander \cite{Alexander}, was the first knot polynomial invariant to be discovered. John Conway studied a version of this polynomial called the Alexander-Conway polynomial which takes a value of 1 on the unknot satisfies the skein relation 

\begin{equation}\label{AlexanderConwaySkein}
  \nabla(L_+)-\nabla(L_-)=z\nabla(L_0)  
\end{equation}
where $L_+$, $L_-$, and $L_0$ are links whose diagrams are the same except locally at one location where they are related as follows:
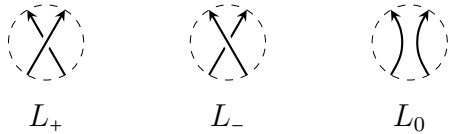
\begin{center}
    \begin{tikzpicture}
 \draw[line width=0.3pt, dashed] (0,0) circle (0.5); 
 \draw[thick, -stealth] (300:0.5) -- (120:0.5);
 \draw[line width=4pt, white] (240:0.25) -- (60:0.25);
 \draw[thick, -stealth] (240:0.5) -- (60:0.5);
\node (Lplus) at (0,-1) {$L_+$};
\end{tikzpicture}\hspace{0.5in}
\begin{tikzpicture}
 \draw[line width=0.3pt, dashed] (0,0) circle (0.5); 
 \draw[thick, -stealth] (240:0.5) -- (60:0.5);
 \draw[line width=4pt, white] (300:0.25) -- (120:0.25);
 \draw[thick, -stealth] (300:0.5) -- (120:0.5);
\node (Lminus) at (0,-1) {$L_-$};
\end{tikzpicture}\hspace{0.5in}
\begin{tikzpicture}
 \draw[line width=0.3pt, dashed] (0,0) circle (0.5); 
 \draw[thick, -stealth] (240:0.5) to [out=60, in = 270] (-0.1, 0) to [out=90, in = 300] (120:0.5);
 \draw[thick, -stealth] (300:0.5) to [out=120, in = 270] (0.1, 0) to [out=90, in = 240] (60:0.5);
\node (Lzero) at (0,-1) {$L_0$};
\end{tikzpicture}
    \label{skeincrossings}
\end{center}
The Alexander-Conway polynomial $\nabla$ is related to the Alexander polynomial $\Delta$ via the relation $\nabla(t^{1/2}-t^{-1/2})= \Delta(t)$. Therefore, the Alexander polynomial satisfies the skein relation
\begin{equation}\label{AlexanderSkein}
  \Delta(L_+)-\Delta(L_-)=\left(t^{1/2}-t^{-1/2}\right)\Delta(L_0)  
\end{equation}
 Setting $\Delta(\textrm{unknot})=1$ fixes a specific choice of the Alexander polynomial for each oriented link, although typically the polynomial is defined up to multiplication by $\pm t^k$ for some $k$.

The Alexander polynomial is also a specialization of a stronger invariant called the HOMFLY polynomial, introduced in \cite{HOMFLY} and also studied independently in \cite{PT}. The HOMFLY polynomial is a Laurent polynomial in $a$ and $z$ defined by the skein relation
\begin{equation}\label{Homflyskein}
aP(L_+)-a^{-1}P(L_-)=zP(L_0)
\end{equation}
and setting $P(\textrm{unknot})=1$. Setting $a=1$ and $z=t^{1/2}-t^{-1/2}$ in the HOMFLY polynomial recovers the Alexander polynomial. In \cite{GLplabiclinks}, Galashin and Lam used this skein relation to prove the following lemma.
\begin{lem}\label{skeinrelation}
    Let $G$ be a simple plabic graph with a boundary leaf face $F$. Let $x$ and $y$ be the vertices on the boundary of $F$ and $e$ be the edge separating $F$ from a boundary face. Let $G'=G-e$ and $G''=G-\{x,y\}$. Then the HOMFLY polynomials of their plabic links satisfy
    \begin{equation}\label{blfskein}
        aP(L_G^\plab)-a^{-1}P(L_{G''}^\plab)=zP(L_{G'}^\plab)
    \end{equation}
\end{lem}
This lemma will be a key tool in later proofs. We will also need the following well-known fact about the HOMFLY polynomial of a connected sum of two links. See \cite{lickorishintroknottheory} for a proof of this fact.
\begin{prop}
 Let $L$ and $L'$ be two oriented links, and let $L\#L'$ be a connected sum of these two links. Then $P(L\# L')=P(L)\cdot P(L')$.   
\end{prop}

\section{The HOMFLY Polynomial of a Forest Quiver}

\subsection{Defining the HOMFLY polynomial of a forest quiver}

\begin{defn}\label{def: HOMFLY of forest}
    Let $Q$ be a quiver whose underlying graph is a forest. The \emph{HOMFLY polynomial of $Q$}, denoted $f(Q)$, is defined recursively by setting
    \begin{itemize}
        \item $f(Q)=1$ if $Q$ is empty,
        \item $f(Q) = \frac{z+z^{-1}}{a}-\frac{z^{-1}}{a^3}$ if $Q$ is a single vertex,
        \item $f(Q) = \frac{z}{a}f(Q-\{v\})+\frac{1}{a^2}f(Q-\{v,\tilde{v}\})$ if $v$ is a leaf in $Q$ which is adjacent to $\tilde{v}$, and
        \item $f(Q) = f(Q_1)\cdot f(Q_2)$ if $Q=Q_1 \sqcup Q_2$.
    \end{itemize}

\end{defn}
\begin{rem}
    Since the definition does not depend on the orientation of the edges in $Q$, we may occasionally write $f(Q)$ where $Q$ is an undirected forest.
\end{rem}

\begin{prop}\label{prop: HOMFLY well-defined}
    The function $f$ is well-defined.
\end{prop}
\begin{proof}

We will proceed by induction on the number $n$ of vertices in $Q$. The formula is well-defined for $n=0$; if $Q$ is empty, then $f(Q)=1$. If $n=1$, then $f(Q)$ must equal $\frac{z+z^{-1}}{a}-\frac{z^{-1}}{a^3}$. If $n \geq 2$, there will be multiple options for leaves which can be removed from $Q$, and if $Q=Q_1\sqcup Q_2$ where $Q_1$ and $Q_2$, one could alternatively compute $f(Q_1)\cdot f(Q_2)$ to obtain $f(Q)$. We will show first that if one chooses to remove a leaf $v$ instead of a leaf $u$, the result does not change. We will then show that if one chooses at some step to remove a leaf from $Q$ rather than take the product of $f$ over the connected components of $Q$, the result is the same.

 Assume that $f(Q')$ is well-defined for all forests $Q'$ with fewer than $n$ vertices. Let $u$ and $v$ be leaves of $Q$ adjacent to vertices $\tilde{u}$ and $\tilde{v}$, respectively. We will handle the proof that the choice to remove $u$ instead of $v$ first does not matter using a few cases:\\
 
 \textbf{Case 1:} Suppose $u$, $v$, $\tilde{u}$, and $\tilde{v}$ are all distinct. By the definition of $f$, 
\begin{equation*}
    f(Q) = \frac{z}{a}f(Q-\{u\}) +\frac{1}{a^2} f(Q-\{u,\tilde{u}\}).
\end{equation*}
In $Q-\{u\}$ and $Q-\{u,\tilde{u}\}$ the leaf $v$, its adjacent vertex $\tilde{v}$, and the edge between them are left untouched, so $v$ remains a leaf. Thus, since $Q-\{u\}$ and $Q-\{u,\tilde{u}\}$ have fewer than $n$ vertices, $f(Q-\{u\})$ and $f(Q-\{u,\tilde{u}\})$ are well-defined and may be computed recursively by first removing the leaf $v$. This yields 
\begin{align*}
    f(Q) &= \frac{z}{a}f(Q-\{u\}) +\frac{1}{a^2} f(Q-\{u,\tilde{u}\})\\
    &= \frac{z}{a}\left(\frac{z}{a}f(Q-\{u,v\}) +\frac{1}{a^2}f(Q-\{u,v,\tilde{v}\}) \right)+\frac{1}{a^2} \left( \frac{z}{a}f(Q-\{u,\tilde{u},v\}) +\frac{1}{a^2}f(Q-\{u,\tilde{u},v,\tilde{v}\})\right)\\
    &= \frac{z^2}{a^2}f(Q-\{u,v\}) +\frac{z}{a^3}f(Q-\{u,v,\tilde{v}\}) + \frac{z}{a^3}f(Q-\{u,\tilde{u},v\}) +\frac{1}{a^4}f(Q-\{u,\tilde{u},v,\tilde{v}\}).
\end{align*}
By the symmetry of the roles of $u$ and $v$ in this case, it follows that when one first removes $v$ and then $u$, the recursion yields
\begin{align*}
    f(Q) &= \frac{z}{a}\left(\frac{z}{a}f(Q-\{v,u\}) +\frac{1}{a^2}f(Q-\{v,u,\tilde{u}\}) \right)+\frac{1}{a^2} \left( \frac{z}{a}f(Q-\{v,\tilde{v},u\}) +\frac{1}{a^2}f(Q-\{v,\tilde{v},u,\tilde{u}\})\right)\\
    &=\frac{z^2}{a^2}f(Q-\{u,v\}) +\frac{z}{a^3}f(Q-\{u,\tilde{u},v\}) + \frac{z}{a^3}f(Q-\{u,v,\tilde{v}\}) +\frac{1}{a^4}f(Q-\{u,\tilde{u},v,\tilde{v}\})
\end{align*}
and we see that two results coincide. \\

\textbf{Case 2:} Suppose that $\tilde{u}=\tilde{v}$ so that $u$ and $v$ are leaves incident to the same vertex $\tilde{u}$. In this case, if we first remove the leaf $u$, we see that $\tilde{u}$ and $v$ are unaffected in $Q-\{u\}$ so we can remove the leaf $v$ to this graph to compute $f(Q-\{u\})$. However, $v$ is now an isolated vertex in $Q-\{u,\tilde{u}\}$, so we cannot compute $f(Q-\{u,\tilde{u}\})$ in the same way. Therefore, we obtain
\begin{align*}
    f(Q) &= \frac{z}{a}f(Q-\{u\}) + \frac{1}{a^2}f(Q-\{u,\tilde{u}\}) \\
    &=\frac{z}{a}\left(\frac{z}{a}f(Q-\{u,v\}) +\frac{1}{a^2}f(Q-\{u,v,\tilde{u}\}) \right) + \frac{1}{a^2}f(Q-\{u, \tilde{u}\})\\
    &=\frac{z^2}{a^2}f(Q-\{u,v\}) +\frac{z}{a^3}f(Q-\{u,v,\tilde{u}\})  + \frac{1}{a^2}f(Q-\{u, \tilde{u}\}).
\end{align*}
If we had instead removed $v$ first, then the recursive formula for $f$ yields
\begin{align*}
    f(Q) 
    &=\frac{z^2}{a^2}f(Q-\{u,v\}) +\frac{z}{a^3}f(Q-\{u,v,\tilde{u}\})  + \frac{1}{a^2}f(Q-\{v,\tilde{u}\}).
\end{align*}
These two formulas differ only in that one contains the term $\frac{1}{a^2}f(Q-\{u,\tilde{u}\})$ while the other contains the term $\frac{1}{a^2}f(Q-\{v,\tilde{u}\})$. However, $Q-\{u,\tilde{u}\}$ and $Q-\{v,\tilde{u}\}$ are both isomorphic to the union of $Q-\{u,\tilde{u}\}$ and an isolated vertex. It follows that $f(Q-\{u, \tilde{u}\})=f(Q-\{v, \tilde{u}\})$.\\

\textbf{Case 3:} Suppose $u=\tilde{v}$ and $v=\tilde{u}$ so that the component of $Q$ containing these vertices consists of a single edge connecting two vertices. Then when first removing $u$ to obtain
$$f(Q) = \frac{z}{a}f(Q-\{u\}) + \frac{1}{a^2}f(Q-\{u,\tilde{u}\})$$
we see that $Q-\{u\}$ consists an isolated vertex and the other connected components of $Q$ which do not contain $u$ or $v$. Similarly, $Q-\{u,\tilde{u}\}$ consists of solely the other connected components of $Q$. The same holds for $Q-\{v\}$ and $Q-\{v,\tilde{v}\}$, so the two results for $f(Q)$ coincide whether $u$ or $v$ is removed first.

Finally, we now show that if $Q$ is disconnected one can either compute $f(Q)$ first by removing a leaf or by computing the product $f(Q_1)\cdot f(Q_2)$ if $Q=Q_1 \sqcup Q_2$ where $Q_1$ and $Q_2$ are non-empty forest quivers. Pick a leaf $v$ in $Q$ which is adjacent to a vertex $\tilde{v}$. By the above, the leaf can be arbitrary. Suppose without loss of generality that $v$ is in $Q_1$. If we choose to first remove $v$ to compute $f(Q)$, the recursion yields
\begin{align*}
    f(Q) &= \frac{z}{a}f(Q-\{v\}) + \frac{1}{a^2}f(Q-\{v,\tilde{v}\})\\
    &= \frac{z}{a}f((Q_1-\{v\}) \sqcup Q_2) + \frac{1}{a^2}f((Q_1-\{v,\tilde{v}\})\sqcup Q_2)\\
    &= \frac{z}{a}f(Q_1-\{v\})\cdot f(Q_2) + \frac{1}{a^2}f(Q_1-\{v,\tilde{v}\})\cdot f(Q_2)\\
    &= \left( \frac{z}{a}f(Q_1-\{v\}) + \frac{1}{a^2}f(Q_1-\{v,\tilde{v}\}) \right)\cdot f(Q_2)\\
    &= f(Q_1)\cdot f(Q_2).
\end{align*}
\end{proof}

\begin{exmp}\label{ex: star HOMFLY}
  Fix $n \geq 4$. Let $S_n$ be the star graph on $n$ vertices which has one vertex of degree $n-1$ connected to $n-1$ leaves.
    When $n=4$, the HOMFLY polynomial of $S_4=D_4$ is
    \begin{align*}
        P(S_4) &= \frac{z}{a}P(A_3)+\frac{1}{a^2}\left(P(A_1) \right)^2 \\
        &=\frac{z^2}{a^2}P(A_2)+ \frac{z}{a^3}P(A_1)+\frac{1}{a^2}\left(P(A_1) \right)^2 \\
        &= \frac{z^2}{a^4} +\frac{z^3}{a^3}P(A_1) + \frac{z}{a^3}P(A_1)+\frac{1}{a^2}\left(P(A_1) \right)^2\\
         &= \frac{z^2}{a^4} +\frac{z^3+z}{a^3}\left( \frac{z+z^{-1}}{a}-\frac{z^{-1}}{a^3}\right)+\frac{1}{a^2}\left(\left( \frac{z+z^{-1}}{a}-\frac{z^{-1}}{a^3}\right) \right)^2\\
         &= \frac{z^4 + 4z^2 + 3+z^{-2}}{a^4} - \frac{z^2  + 3+2z^{-2}}{a^6} + \frac{z^{-2}}{a^8}.
    \end{align*}
    Using this, one can also compute the HOMFLY polynomial of $S_n$ for $n > 4$:
    \begin{align*}
        P(S_n) &= \frac{z}{a}P(S_{n-1})+\frac{1}{a^2}\left(P(A_1)\right)^{n-2}\\
        &= \frac{z^2}{a^2}P(S_{n-2})+\frac{z}{a^3}\left(P(A_1)\right)^{n-3}+\frac{1}{a^2}\left(P(A_1)\right)^{n-2}\\
        &= \frac{z^3}{a^3}P(S_{n-3})+\frac{z^2}{a^4}\left(P(A_1)\right)^{n-4}+\frac{z}{a^3}\left(P(A_1)\right)^{n-3}+\frac{1}{a^2}\left(P(A_1)\right)^{n-2}\\
        &\quad\vdots \\
        & = \frac{z^{n-4}}{a^{n-4}}P(S_4) + \frac{z^{n-5}}{a^{n-3}}\left(P(A_1)\right)^3 + \dots +\frac{z}{a^3}\left(P(A_1)\right)^{n-3}+\frac{1}{a^2}\left(P(A_1)\right)^{n-2}\\
        &=\frac{z^{n-2}}{a^n} + \frac{z^{n-1}+z^{n-3}}{a^{n-1}}P(A_1)+\frac{z^{n-4}}{a^{n-2}}\left(P(A_1) \right)^2+\frac{z^{n-5}}{a^{n-3}}\left(P(A_1)\right)^3 + \dots \\
        &\qquad+\frac{z}{a^3}\left(P(A_1)\right)^{n-3}+\frac{1}{a^2}\left(P(A_1)\right)^{n-2}\\
        &=\frac{z^{n-2}}{a^n} + \frac{z^{n-1}+z^{n-3}}{a^{n-1}}P(A_1) + \sum_{k=2}^{n-2}\left( \frac{z^{k-2}}{a^{k}}\left(P(A_1) \right)^{n-k}\right)\\
        &=\frac{z^{n-2}}{a^n} + \frac{z^{n-1}+z^{n-3}}{a^{n-1}}\left( \frac{z+z^{-1}}{a}-\frac{z^{-1}}{a^3}\right) + \sum_{k=2}^{n-2}\left( \frac{z^{k-2}}{a^{k}}\left( \frac{z+z^{-1}}{a}-\frac{z^{-1}}{a^3}\right)^{n-k}\right).
    \end{align*}  
\end{exmp}

The Alexander polynomial $\Delta(Q)$ of a forest quiver $Q$ can be obtained from $f(Q)$ via the substitution $a=1$ and $z=t^{1/2}-t^{-1/2}$, as in the knot theory sense. We obtain some simple formulas for the Alexander polynomials of type $A_n$ and $D_n$ quivers.

\begin{exmp}
    We will consider the Alexander polynomial of a type $A_n$ quiver. Since $f(A_1) = \frac{z+z^{-1}}{a}-\frac{z^{-1}}{a^3}$, we find $\Delta(A_1) = ((t^{1/2}-t^{-1/2})+(t^{1/2}-t^{-1/2})^{-1}) - (t^{1/2}-t^{-1/2})^{-1} = t^{1/2}-t^{-1/2}$. It follows that 
    \begin{align*}
        \Delta(A_2) &= (t^{1/2}-t^{-1/2})\Delta(A_1) + \Delta(A_0)\\
        &= (t^{1/2}-t^{-1/2})(t^{1/2}-t^{-1/2})+1\\
        &= t-1+t^{-1}.
    \end{align*}
    This pattern holds for higher $n$ as well; one can prove by induction that $$\Delta(A_n) = t^{-n/2}\cdot \sum\limits_{k=0}^n (-1)^{n-k} t^k.$$
\end{exmp}

\begin{exmp}
    The Alexander polynomial of the $D_4$ quiver is
    \begin{align*}
        \Delta(D_4) &= (t^{1/2}-t^{-1/2})\Delta(A_3) + \Delta(A_1)^2\\
        &= (t^{1/2}-t^{-1/2}) (t^{3/2} - t^{1/2} +t^{-1/2} -t^{-3/2}) + (t-2+t^{-1})\\
        &= t^2-2t+2-2t^{-1}+t^{-2} + (t-2+t^{-1})\\
        &= t^2-t-t^{-1}+t^{-2}.
    \end{align*}
    Similarly,
    \begin{align*}
        \Delta(D_5) &= (t^{1/2}-t^{-1/2})\Delta(D_4)+\Delta(A_3)\\
        &= (t^{1/2}-t^{-1/2})(t^2-t-t^{-1}+t^{-2})+t^{3/2} - t^{1/2} +t^{-1/2} -t^{-3/2}\\
        &= t^{5/2}-t^{3/2}+t^{-3/2}-t^{-5/2}.
    \end{align*}
   Using these base cases, one can prove via induction that $$\Delta(D_n) = t^{-n/2}\left(t^n-t^{n-1} + (-1)^{n-1}t+(-1)^n \right).$$
\end{exmp}

\subsection{Connections to the HOMFLY polynomial of a plabic link}

The main result of this section is that the HOMFLY polynomial of a forest quiver $Q$ is equal to the HOMFLY polynomial of the plabic link associated to any connected plabic graph whose quiver is $Q$. Thus, Definition \ref{def: HOMFLY of forest} gives a way to go directly from a forest quiver to a corresponding link invariant. 

\begin{rem}
 If the plabic graphs are not assumed to be connected, it is possible to have two plabic graphs with the same quiver but where the HOMFLY polynomials of the associated plabic links are different. See Figure \ref{fig: disconnected graphs} for an example of two plabic graphs with the same quiver but whose plabic links have different HOMFLY polynomials. The plabic link of $G_1$ is a disjoint union of two positive Hopf links, and $$P(L_{G_1}^\plab) = \left(\frac{a-a^{-1}}{z}\right)\left( \frac{z+z^{-1}}{a}-\frac{z^{-1}}{a^3}\right)^2.$$ On the other hand, $G_2$ is a connected sum of two positive Hopf links, so $$P(L_{G_2}^\plab) = \left( \frac{z+z^{-1}}{a}-\frac{z^{-1}}{a^3}\right)^2.$$

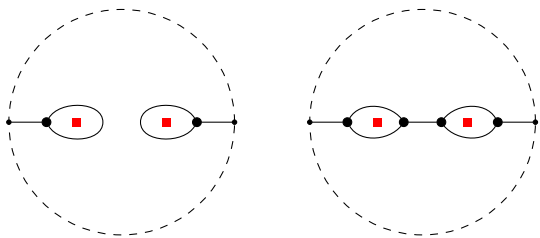
\begin{figure}
    \centering
    \begin{tikzpicture}
    \begin{scope}
         \draw[dashed] (0,0) circle (1.5);
         \node[scale =0.3, circle, draw=black, fill=black] (l) at (-1,0) {};
         \node[scale =0.3, circle, draw=black, fill=black] (r) at (1,0) {};

         \node[scale =0.15, circle, draw=black, fill=black] (bl) at (-1.5,0) {};
         \node[scale =0.15, circle, draw=black, fill=black] (br) at (1.5,0) {};

         \draw (bl) -- (l);
         \draw (br) -- (r);

         \draw (l) to [out = 60, in = 90] (-0.25, 0) to [out=270, in = 300] (l);
         \draw (r) to [out = 120, in = 90] (0.25, 0) to [out=270, in = 240] (r);

         \node[fill=red, scale = 0.4] () at (-0.6,0) {};
         \node[fill=red, scale = 0.4] () at (0.6,0) {};
    \end{scope}
    \begin{scope}[xshift=4cm]
         \draw[dashed] (0,0) circle (1.5);
         \node[scale =0.3, circle, draw=black, fill=black] (l) at (-1,0) {};
         \node[scale =0.3, circle, draw=black, fill=black] (r) at (1,0) {};

         \node[scale =0.15, circle, draw=black, fill=black] (bl) at (-1.5,0) {};
         \node[scale =0.15, circle, draw=black, fill=black] (br) at (1.5,0) {};

         \draw (bl) -- (l);
         \draw (br) -- (r);

         \draw (l) to [out = 60, in = 120] (-0.25, 0) to [out=240, in = 300] (l);
         \draw (r) to [out = 120, in = 60] (0.25, 0) to [out=300, in = 240] (r);

         \node[fill=red, scale = 0.4] () at (-0.6,0) {};
         \node[fill=red, scale = 0.4] () at (0.6,0) {};

         \node[scale =0.3, circle, draw=black, fill=black] (ml) at (-0.25,0) {};
         \node[scale =0.3, circle, draw=black, fill=black] (mr) at (0.25,0) {};

         \draw (ml) -- (mr);
    \end{scope}
       
    \end{tikzpicture}
    \caption{Two plabic graphs with the same quiver but whose plabic links have different HOMFLY polynomials.}
    \label{fig: disconnected graphs}
\end{figure}   
\end{rem}

In order to prove this main result, we will consider operations on plabic graphs and links which will be analogous to taking the product of the HOMFLY polynomial of connected components of a quiver or removing leaves from a quiver. Leaf removal will correspond to the skein relation (\ref{blfskein}) applied to boundary leaf faces. Meanwhile, a taking a product over connected components of a quiver will correspond to taking a connected sum of links. In order to prove the latter correspondence, we first fix some terminology. 

Suppose that $G$ is a plabic graph with an edge $e$ such that two (not necessarily distinct) boundary faces $B$ and $B'$ lie on either side of $e$. We will refer to such an edge as a \emph{dividing edge}. Then one can draw a line from the boundary of the disk through $B$, across $e$, and through $B'$ back to the boundary of the disk which divides $G$ as pictured in Figure \ref{fig: dividing edge}.  Let $G_1$ (resp. $G_2$) be the induced subgraph on all vertices to the left (resp. right) of this dividing line.

\begin{figure}
    \centering
    \begin{tikzpicture}
        \begin{scope}
            \draw[dashed] (0,0) circle (1.5);
            \node [circle, scale = 0.4, fill=black, draw=black] (v1) at (-0.5, 0) {};
            \node [circle, scale = 0.4, fill=black, draw=black] (v2) at (0.5, 0) {};

            \node (1) at ($(v1)+(120:0.75)$) {};
            \node (2) at ($(v1)+(180:0.75)$) {};
            \node (3) at ($(v1)+(240:0.75)$) {};

            \node (4) at ($(v2)+(60:0.75)$) {};
            \node (5) at ($(v2)+(20:0.75)$) {};
            \node (6) at ($(v2)+(340:0.75)$) {};
            \node (7) at ($(v2)+(300:0.75)$) {};

            \draw (v1) -- (v2);
            \draw (v1) -- (1);
            \draw (v1) -- (2);
            \draw (v1) -- (3);
            \draw (v2) -- (4);
            \draw (v2) -- (5);
            \draw (v2) -- (6);
            \draw (v2) -- (7);

            \draw[dashed, blue, thick] ($(90:1.5)+(0.1,0)$) -- ($(270:1.5)+(0.1,0)$);

            \node (b) at (-0.16, 0.3) {\scriptsize{$B$}};
            \node (b') at (-0.16, -0.3) {\scriptsize{$B'$}};

            \node (g1) at (-0.4, 1) {\scriptsize{$G_1$}};
            \node (g2) at (0.5, 1) {\scriptsize{$G_2$}};
        \end{scope}
    \end{tikzpicture}
    \caption{If there is an edge $e$ in $G$ which separates two (not necessarily distinct) boundary faces $B$ and $B'$, one can divide $G$ by drawing a line through these boundary faces across $e$, pictured here in blue. If one removes $e$, the result can be considered as the disjoint union of two smaller plabic graphs $G_1$ and $G_2$ on the left and right respectively of this dividing line.}
    \label{fig: dividing edge}
\end{figure}
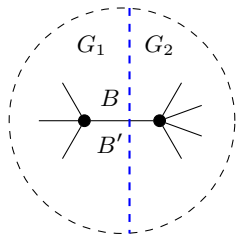

\begin{lem}\label{lem: dividing edge}
    Suppose $G$ be a plabic graph with a dividing edge $e$, and let $G_1$ and $G_2$ be as defined above. Then $L_G^\plab$ is a connected sum of $L_{G_1}^\plab$ and $L_{G_2}^\plab$.
\end{lem}
\begin{proof}
    See Figure \ref{fig: connected sums dividing edge} for what $G$ and $L_{G}^\plab$ look like locally around $e$. The three different rows correspond to the possibilities for the colors of the endpoints of $e$. If one were to delete $e$ the resulting link would be $L_{G_1}^\plab \sqcup L_{G_2}^\plab$, as pictured in the left hand column of Figure \ref{fig: connected sums dividing edge}. The middle column shows the connected sum of these two links. In the last row, the link $L_{G_1}^\plab$ has been flipped vertically to make the connected sum more evident. If one then flips the portion of this connected sum which comes from $L_{G_1}^\plab$ twice in the first two rows and once in the last row so that the top portion begins by coming out of the page, it follows that $L_{G_1}^\plab \# L_{G_2}^\plab$ is isotopic to $L_{G}^\plab$.

    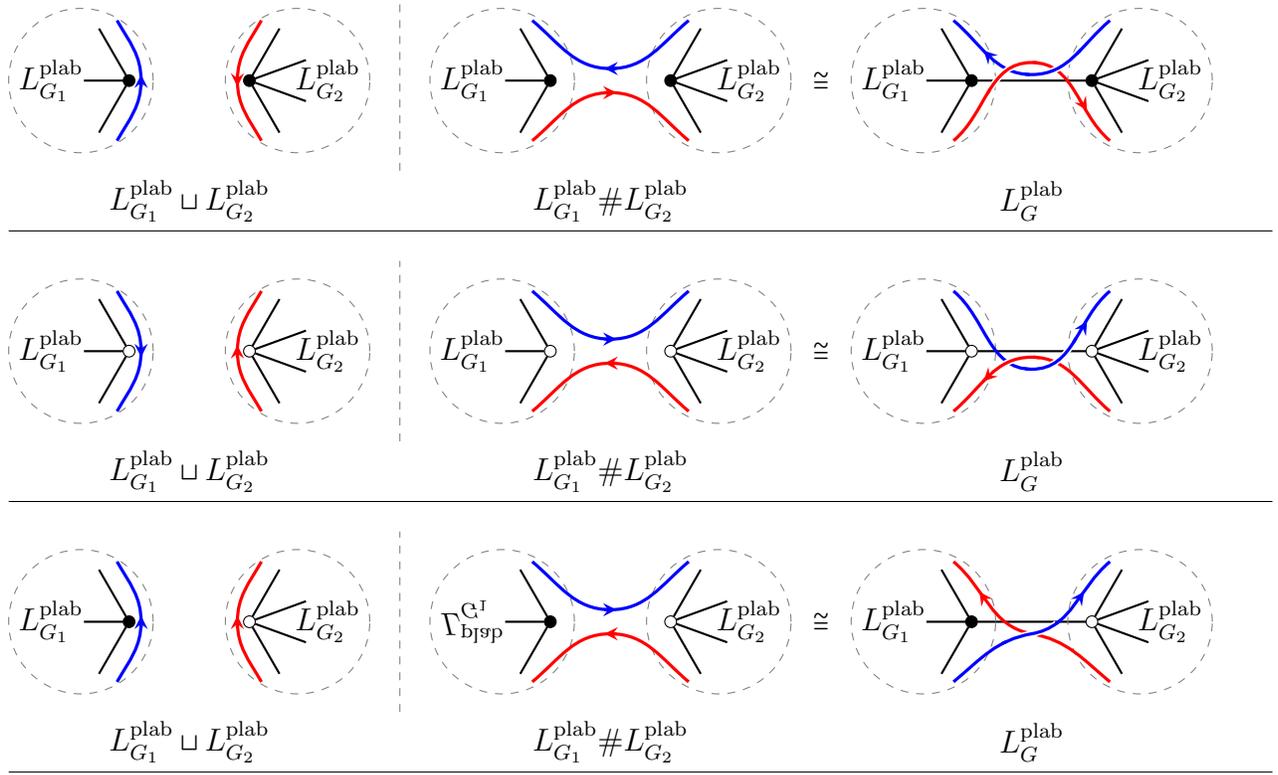
\begin{figure}
        \centering
        \begin{tikzpicture}[scale=0.8]
        \begin{scope}
            \begin{scope}
                 \node[circle, fill=black, draw=black, scale = 0.4] (v1) at (-1,0) {};
                \node[circle, fill=black, draw=black, scale = 0.4] (v2) at (1,0) {};
                \draw[thick] ($(120:1)+(-1,0)$) -- (v1);
                \draw[thick] ($(180:1)+(-0.75,0)$) -- (v1);
                \draw[thick] ($(240:1)+(-1,0)$) -- (v1);
                
                \draw[thick] ($(60:1)+(1,0)$) -- (v2);
                \draw[thick] ($(20:1)+(1,0)$) -- (v2);
                \draw[thick] ($(340:1)+(1,0)$) -- (v2);
                \draw[thick] ($(300:1)+(1,0)$) -- (v2);
                \begin{knot}[clip width = 5]
                    \strand[very thick, blue] (-1.2,-1) to [out=60, in = 270] (-0.8, 0) to [out=90, in = 300] (-1.2,1);
                    \strand[very thick, red] (1.2,-1) to [out=120, in = 270] (0.8, 0) to [out=90, in = 240] (1.2,1);
                \end{knot}
                \draw[blue, very thick, -stealth] (-0.8, -0.1) -- (-0.8, 0.1);
                \draw[red, very thick, -stealth] (0.8, 0.1) -- (0.8, -0.1);
                \draw[gray, dashed] (-1.8, 0) circle (1.2);
                \node at (-2.3,0) {$L_{G_1}^\plab$};
                \draw[gray, dashed] (1.8, 0) circle (1.2);
                \node at (2.3,0) {$L_{G_2}^\plab$};
             \node at (0,-2) {$L_{G_1}^\plab \sqcup L_{G_2}^\plab$};
             \draw[gray, dashed] (3.5, -1.5) -- (3.5,1.5);
            \end{scope}
            \begin{scope}[xshift=7cm]
                \begin{scope}
                 \node[circle, fill=black, draw=black, scale = 0.4] (v1) at (-1,0) {};
                \node[circle, fill=black, draw=black, scale = 0.4] (v2) at (1,0) {};
                \draw[thick] ($(120:1)+(-1,0)$) -- (v1);
                \draw[thick] ($(180:1)+(-0.75,0)$) -- (v1);
                \draw[thick] ($(240:1)+(-1,0)$) -- (v1);
                
                \draw[thick] ($(60:1)+(1,0)$) -- (v2);
                \draw[thick] ($(20:1)+(1,0)$) -- (v2);
                \draw[thick] ($(340:1)+(1,0)$) -- (v2);
                \draw[thick] ($(300:1)+(1,0)$) -- (v2);
                \draw[gray, dashed] (-1.8, 0) circle (1.2);
                \draw[gray, dashed] (1.8, 0) circle (1.2);
                \node at (-2.3,0)  {$L_{G_1}^\plab$};
                \begin{knot}[clip width = 2]
                    \strand[very thick, blue] (-1.3,1) to [out=320, in = 180] (0, 0.2) to [out=0, in = 220] (1.3,1);
                    \strand[very thick, red] (-1.3,-1) to [out=40, in = 180] (0, -0.2) to [out=0, in = 140] (1.3,-1);
                \end{knot}
                \draw[blue, very thick, stealth-] (-0.1, 0.2) -- (0.1, 0.2);
                \draw[red, very thick, -stealth] (-0.1, -0.2) -- (0.1, -0.2);
                \node at (2.3,0) 
                {$L_{G_2}^\plab$};
             \node at (0,-2) {$L_{G_1}^\plab \# L_{G_2}^\plab$};
             \node at (3.5, 0) {$\cong$};
            \end{scope}
            \begin{scope}[xshift=7cm]
                 \node[circle, fill=black, draw=black, scale = 0.4] (v1) at (-1,0) {};
                \node[circle, fill=black, draw=black, scale = 0.4] (v2) at (1,0) {};
                \draw[thick] ($(120:1)+(-1,0)$) -- (v1);
                \draw[thick] ($(180:1)+(-0.75,0)$) -- (v1);
                \draw[thick] ($(240:1)+(-1,0)$) -- (v1);
                
                \draw[thick] ($(60:1)+(1,0)$) -- (v2);
                \draw[thick] ($(20:1)+(1,0)$) -- (v2);
                \draw[thick] ($(340:1)+(1,0)$) -- (v2);
                \draw[thick] ($(300:1)+(1,0)$) -- (v2);
                \draw[thick] (v1) -- (v2);
                \draw[gray, dashed] (-1.8, 0) circle (1.2);
                \draw[gray, dashed] (1.8, 0) circle (1.2);
                \node at (-2.3,0)  {$L_{G_1}^\plab$};
                \begin{knot}[clip width = 2, flip crossing = 1]
                    \strand[very thick, blue] (-1.3,1) to [out=320, in = 180] (0, 0.1) to [out=0, in = 220] (1.3,1);
                    \strand[very thick, red] (1.3,-1) to [out=140, in = 360] (0, 0.3) to [out=180, in = 40] (-1.3,-1);
                \end{knot}
                \draw[blue, very thick, -stealth] (-0.7, 0.4) -- (-0.8, 0.5);
                \draw[red, very thick, -stealth] (0.86, -0.4) -- (0.92, -0.5);
                \node at (2.3,0) {$L_{G_2}^\plab$};
             \node at (0,-2) {$L_{G}^\plab$};
            \end{scope}
            \end{scope}
        \draw (-3, -2.5) -- (18,-2.5);
        \end{scope}
        
         \begin{scope}[yshift = -4.5cm]
            \begin{scope}
                 \node[circle, fill=white, draw=black, scale = 0.4] (v1) at (-1,0) {};
                \node[circle, fill=white, draw=black, scale = 0.4] (v2) at (1,0) {};
                \draw[thick] ($(120:1)+(-1,0)$) -- (v1);
                \draw[thick] ($(180:1)+(-0.75,0)$) -- (v1);
                \draw[thick] ($(240:1)+(-1,0)$) -- (v1);
                
                \draw[thick] ($(60:1)+(1,0)$) -- (v2);
                \draw[thick] ($(20:1)+(1,0)$) -- (v2);
                \draw[thick] ($(340:1)+(1,0)$) -- (v2);
                \draw[thick] ($(300:1)+(1,0)$) -- (v2);
                \begin{knot}[clip width = 5]
                    \strand[very thick, blue] (-1.2,-1) to [out=60, in = 270] (-0.8, 0) to [out=90, in = 300] (-1.2,1);
                    \strand[very thick, red] (1.2,-1) to [out=120, in = 270] (0.8, 0) to [out=90, in = 240] (1.2,1);
                \end{knot}
                \draw[blue, very thick, -stealth] (-0.8, 0.1) -- (-0.8, -0.1);
                \draw[red, very thick, -stealth] (0.8, -0.1) -- (0.8, 0.1);
                \draw[gray, dashed] (-1.8, 0) circle (1.2);
                \node at (-2.3,0) {$L_{G_1}^\plab$};
                \draw[gray, dashed] (1.8, 0) circle (1.2);
                \node at (2.3,0) {$L_{G_2}^\plab$};
             \node at (0,-2) {$L_{G_1}^\plab \sqcup L_{G_2}^\plab$};
             \draw[gray, dashed] (3.5, -1.5) -- (3.5,1.5);
            \end{scope}
            \begin{scope}[xshift=7cm]
                \begin{scope}
                 \node[circle, fill=white, draw=black, scale = 0.4] (v1) at (-1,0) {};
                \node[circle, fill=white, draw=black, scale = 0.4] (v2) at (1,0) {};
                \draw[thick] ($(120:1)+(-1,0)$) -- (v1);
                \draw[thick] ($(180:1)+(-0.75,0)$) -- (v1);
                \draw[thick] ($(240:1)+(-1,0)$) -- (v1);
                
                \draw[thick] ($(60:1)+(1,0)$) -- (v2);
                \draw[thick] ($(20:1)+(1,0)$) -- (v2);
                \draw[thick] ($(340:1)+(1,0)$) -- (v2);
                \draw[thick] ($(300:1)+(1,0)$) -- (v2);
                \draw[gray, dashed] (-1.8, 0) circle (1.2);
                \draw[gray, dashed] (1.8, 0) circle (1.2);
                \node at (-2.3,0)  {$L_{G_1}^\plab$};
                \begin{knot}[clip width = 2]
                    \strand[very thick, blue] (-1.3,1) to [out=320, in = 180] (0, 0.2) to [out=0, in = 220] (1.3,1);
                    \strand[very thick, red] (-1.3,-1) to [out=40, in = 180] (0, -0.2) to [out=0, in = 140] (1.3,-1);
                \end{knot}
                \draw[blue, very thick, stealth-] (0.1, 0.2) -- (-0.1, 0.2);
                \draw[red, very thick, -stealth] (0.1, -0.2) -- (-0.1, -0.2);
                \node at (2.3,0) 
                {$L_{G_2}^\plab$};
             \node at (0,-2) {$L_{G_1}^\plab \# L_{G_2}^\plab$};
             \node at (3.5, 0) {$\cong$};
            \end{scope}
            \begin{scope}[xshift=7cm]
                 \node[circle, fill=white, draw=black, scale = 0.4] (v1) at (-1,0) {};
                \node[circle, fill=white, draw=black, scale = 0.4] (v2) at (1,0) {};
                \draw[thick] ($(120:1)+(-1,0)$) -- (v1);
                \draw[thick] ($(180:1)+(-0.75,0)$) -- (v1);
                \draw[thick] ($(240:1)+(-1,0)$) -- (v1);
                
                \draw[thick] ($(60:1)+(1,0)$) -- (v2);
                \draw[thick] ($(20:1)+(1,0)$) -- (v2);
                \draw[thick] ($(340:1)+(1,0)$) -- (v2);
                \draw[thick] ($(300:1)+(1,0)$) -- (v2);
                \draw[thick] (v1) -- (v2);
                \draw[gray, dashed] (-1.8, 0) circle (1.2);
                \draw[gray, dashed] (1.8, 0) circle (1.2);
                \node at (-2.3,0)  {$L_{G_1}^\plab$};
                \begin{knot}[clip width = 2, flip crossing = 1]
                    \strand[very thick, blue] (-1.3,1) to [out=320, in = 180] (0, -0.3) to [out=0, in = 220] (1.3,1);
                    \strand[very thick, red] (1.3,-1) to [out=140, in = 360] (0, -0.1) to [out=180, in = 40] (-1.3,-1);
                \end{knot}
                \draw[red, very thick, -stealth] (-0.7, -0.4) -- (-0.8, -0.5);
                \draw[blue, very thick, -stealth] (0.86, 0.4) -- (0.92, 0.5);
                \node at (2.3,0) {$L_{G_2}^\plab$};
             \node at (0,-2) {$L_{G}^\plab$};
            \end{scope}
            \end{scope}
            \draw (-3, -2.5) -- (18,-2.5);
        \end{scope}

        \begin{scope}[yshift = -9cm]
            \begin{scope}
                 \node[circle, fill=black, draw=black, scale = 0.4] (v1) at (-1,0) {};
                \node[circle, fill=white, draw=black, scale = 0.4] (v2) at (1,0) {};
                \draw[thick] ($(120:1)+(-1,0)$) -- (v1);
                \draw[thick] ($(180:1)+(-0.75,0)$) -- (v1);
                \draw[thick] ($(240:1)+(-1,0)$) -- (v1);
                
                \draw[thick] ($(60:1)+(1,0)$) -- (v2);
                \draw[thick] ($(20:1)+(1,0)$) -- (v2);
                \draw[thick] ($(340:1)+(1,0)$) -- (v2);
                \draw[thick] ($(300:1)+(1,0)$) -- (v2);
                \begin{knot}[clip width = 5]
                    \strand[very thick, blue] (-1.2,-1) to [out=60, in = 270] (-0.8, 0) to [out=90, in = 300] (-1.2,1);
                    \strand[very thick, red] (1.2,-1) to [out=120, in = 270] (0.8, 0) to [out=90, in = 240] (1.2,1);
                \end{knot}
                \draw[blue, very thick, -stealth] (-0.8, -0.1) -- (-0.8, 0.1);
                \draw[red, very thick, -stealth] (0.8, -0.1) -- (0.8, 0.1);
                \draw[gray, dashed] (-1.8, 0) circle (1.2);
                \node at (-2.3,0) {$L_{G_1}^\plab$};
                \draw[gray, dashed] (1.8, 0) circle (1.2);
                \node at (2.3,0) {$L_{G_2}^\plab$};
             \node at (0,-2) {$L_{G_1}^\plab \sqcup L_{G_2}^\plab$};
             \draw[gray, dashed] (3.5, -1.5) -- (3.5,1.5);
            \end{scope}
            \begin{scope}[xshift=7cm]
                \begin{scope}
                 \node[circle, fill=black, draw=black, scale = 0.4] (v1) at (-1,0) {};
                \node[circle, fill=white, draw=black, scale = 0.4] (v2) at (1,0) {};
                \draw[thick] ($(120:1)+(-1,0)$) -- (v1);
                \draw[thick] ($(180:1)+(-0.75,0)$) -- (v1);
                \draw[thick] ($(240:1)+(-1,0)$) -- (v1);
                
                \draw[thick] ($(60:1)+(1,0)$) -- (v2);
                \draw[thick] ($(20:1)+(1,0)$) -- (v2);
                \draw[thick] ($(340:1)+(1,0)$) -- (v2);
                \draw[thick] ($(300:1)+(1,0)$) -- (v2);
                \draw[gray, dashed] (-1.8, 0) circle (1.2);
                \draw[gray, dashed] (1.8, 0) circle (1.2);
                \node at (-2.3,0)  {\raisebox{\depth}{\scalebox{1}[-1]{$L_{G_1}^\plab$}}};;
                \begin{knot}[clip width = 2]
                    \strand[very thick, blue] (-1.3,1) to [out=320, in = 180] (0, 0.2) to [out=0, in = 220] (1.3,1);
                    \strand[very thick, red] (-1.3,-1) to [out=40, in = 180] (0, -0.2) to [out=0, in = 140] (1.3,-1);
                \end{knot}
                \draw[blue, very thick, stealth-] (0.1, 0.2) -- (-0.1, 0.2);
                \draw[red, very thick, -stealth] (0.1, -0.2) -- (-0.1, -0.2);
                \node at (2.3,0) 
                {$L_{G_2}^\plab$};
             \node at (0,-2) {$L_{G_1}^\plab \# L_{G_2}^\plab$};
             \node at (3.5, 0) {$\cong$};
            \end{scope}
            \begin{scope}[xshift=7cm]
                 \node[circle, fill=black, draw=black, scale = 0.4] (v1) at (-1,0) {};
                \node[circle, fill=white, draw=black, scale = 0.4] (v2) at (1,0) {};
                \draw[thick] ($(120:1)+(-1,0)$) -- (v1);
                \draw[thick] ($(180:1)+(-0.75,0)$) -- (v1);
                \draw[thick] ($(240:1)+(-1,0)$) -- (v1);
                
                \draw[thick] ($(60:1)+(1,0)$) -- (v2);
                \draw[thick] ($(20:1)+(1,0)$) -- (v2);
                \draw[thick] ($(340:1)+(1,0)$) -- (v2);
                \draw[thick] ($(300:1)+(1,0)$) -- (v2);
                \draw[thick] (v1) -- (v2);
                \draw[gray, dashed] (-1.8, 0) circle (1.2);
                \draw[gray, dashed] (1.8, 0) circle (1.2);
                \node at (-2.3,0)  {$L_{G_1}^\plab$};
                \begin{knot}[clip width = 2]
                    \strand[very thick, blue] (-1.3,-1) to [out=40, in = 190] (0, -0.2
                    ) to [out=10, in = 220] (1.3,1);
                    \strand[very thick, red] (1.3,-1) to [out=140, in = 350] (0, -0.2) to [out=170, in = 320] (-1.3,1);
                \end{knot}
                \draw[red, very thick, -stealth] (-0.8, 0.4) -- (-0.87, 0.5);
                \draw[blue, very thick, -stealth] (0.8, 0.4) -- (0.87, 0.5);
                \node at (2.3,0) {$L_{G_2}^\plab$};
             \node at (0,-2) {$L_{G}^\plab$};
            \end{scope}
            \end{scope}
            \draw (-3, -2.5) -- (18,-2.5);
        \end{scope}
            
            \end{tikzpicture}
        \caption{The plabic link of a plabic graph $G$ is a connected sum of the plabic links of the subgraphs $G_1$ and $G_2$ on either side of a dividing edge $e$ in $G$.}
        \label{fig: connected sums dividing edge}
    \end{figure}
\end{proof}

\begin{prop}\label{prop: disconnected-quiver-connected-sum}
    Let $G$ be a connected plabic graph whose quiver $Q_G$ is a disjoint union of non-empty tree quivers $Q_1,Q_2,\dots, Q_k$ for some $k \geq 2$. Then $L_G^\plab$ is isotopic to a connected sum of links $L_{G_1}^\plab, \dots L_{G_k}^\plab$ for some choice of connected plabic graphs $G_i$ with $Q_{G_i}=Q_i$ for $i=1,\dots, k$.
\end{prop}
\begin{proof}
    We will proceed by induction on $k$. For $k=1$, the statement is trivial. For $k>1$, then by Lemma \ref{lem: dividing edge} it suffices to find a dividing edge $e$ in $G$ such that $Q_{G_1}$ and $Q_{G_2}$ are both non-empty. More specifically, one has $Q_{G_1}\sqcup Q_{G_2} = Q_G$. Suppose without loss of generality that $Q_{G_1} = Q_1 \sqcup Q_2 \sqcup \dots \sqcup Q_\ell$ and $Q_{G_2} = Q_{\ell+1} \sqcup Q_{\ell+2} \sqcup \dots \sqcup Q_k$ for some $1 \leq \ell < k$. Note that $G_1$ and $G_2$ are both connected if $G$ is connected. Therefore, by the inductive hypothesis, $L_{G_1}^\plab$ is a connected sum of plabic links $L_1, L_2,\dots, L_\ell$ for some connected plabic graphs whose quivers are $Q_1, Q_2,\dots, Q_\ell$ respectively. Similarly, $L_{G_2}^\plab$ is a connected sum of plabic links $L_{\ell+1}, L_{\ell+2},\dots, L_k$ for some connected plabic graphs whose quivers are $Q_{\ell+1}, Q_{\ell+2},\dots, Q_k$ respectively. The result now follows since $L_G^\plab$ is a connected sum of $L_{G_1}^\plab$ and $L_{G_2}^\plab$.

    For convenience, we will replace $G$ with the bipartite reduction of its tail reduction. That is, we can assume $G$ has no boundary vertices and that all edges with distinct endpoints of the same color have been contracted into one vertex. For each $i\in \{1,\dots, k\}$ let $V_i$ be the subset of vertices in $G$ which are on the boundary of some interior face $F$ where the corresponding vertex $v_{F}\in Q_G$ is in $Q_i$. We will break the inductive step into a few cases based on whether or not the sets $V_1, V_2,\dots, V_k$ are pairwise disjoint. 
    
    First we will consider the case where the sets are not pairwise disjoint. Then there exists a vertex $v \in V_i \cap V_j$ for some $i \neq j$. Without loss of generality, we will assume that $v$ is a black vertex. First, let us consider the case that there is a loop edge at $v$. In this case, the face $F$ enclosed by this loop edge corresponds to a single isolated vertex in $Q_G$. Thus, any other faces with $v$ on their boundary must either be boundary faces or must correspond to vertices in different connected components of $Q_G$. Since $v$ is in the intersection of $V_i$ and $V_j$, there must be at least one other interior face with $v$ on its boundary. We claim that there must also be at least one boundary face with $v$ on its boundary. If not, this would create a cycle in $Q_G$ as pictured in Figure \ref{fig: cycle at v} because there is an edge in $Q_G$ between each pair of adjacent non-loop interior faces with $v$ on their boundary.
    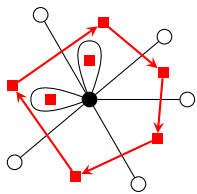
\begin{figure}
        \centering
        \begin{tikzpicture}[scale = 1.3]
            \node[scale = 0.5, circle, fill=black, draw=black] (v) at (0,0) {};
            \node[scale = 0.5, circle, fill=white, draw=black] (1) at (0:1) {};
            \node[scale = 0.5, circle, fill=white, draw=black] (2) at (40:1) {};
            \node[scale = 0.5, circle, fill=white, draw=black] (3) at (120:1) {};
            \node[scale = 0.5, circle, fill=white, draw=black] (4) at (220:1) {};
            \node[scale = 0.5, circle, fill=white, draw=black] (5) at (300:1) {};

            \node[scale = 0.5, fill=red] (q1) at (20:0.8) {};
            \node[scale = 0.5, fill=red] (q2) at (80:0.8) {};
            \node[scale = 0.5, fill=red] (q3) at (170:0.8) {};
            \node[scale = 0.5, fill=red] (q4) at (260:0.8) {};
            \node[scale = 0.5, fill=red] (q5) at (330:0.8) {};
            \node[scale = 0.5, fill=red] (l1) at (180:0.4) {};
            \node[scale = 0.5, fill=red] (l2) at (90:0.4) {};

            \draw (v) to [out= 110, in = 180] (0,0.6) to [out=0, in = 70] (v);
            \draw (v) to [out= 160, in = 90] (-0.6,0) to [out=270, in = 200] (v);
            \draw (v) -- (1);
            \draw (v) -- (2);
            \draw (v) -- (3);
            \draw (v) -- (4);
            \draw (v) -- (5);

            \draw[red, thick, stealth-] (q1) -- (q2);
            \draw[red, thick, stealth-] (q2) -- (q3);
            \draw[red, thick, stealth-] (q3) -- (q4);
            \draw[red, thick, stealth-] (q4) -- (q5);
            \draw[red, thick, stealth-] (q5) -- (q1);
        \end{tikzpicture}
        \caption{If there are no boundary faces at $v$, there is a cycle in $Q_G$.}
        \label{fig: cycle at v}
    \end{figure}
    
    Now, the face $F$ enclosed by a loop edge at $v$ is either surrounded by an interior face or a boundary face. The former case can be reduced to the latter. In particular, if $F$ is surrounded by an interior face $F'$, then one can uncontract $v$ into an edge as pictured in Figure \ref{fig: push loop through} so that the boundary of $F$ is two black vertices connected by two edges. The edge which was formerly a loop can be contracted to ``push" the face $F$ through so that it is instead surrounded by a boundary face $B$. Once in this scenario where a loop face is surrounded by a boundary face, the last step in Figure \ref{fig: push loop through} shows how $v$ can be uncontracted into a dividing edge. Let $G'$ be the result of this procedure. Then one subgraph, say $G_1$, formed by splitting along this dividing edge consists of the loop edge at a vertex $u$ which encloses $F$. The other subgraph $G_2$ is the graph $G'-\{u\}$. Then $Q_{G_1}$ is an isolated vertex and $Q_{G_2}$ is $Q_{G'}$ is $Q_{G}$ minus this isolated vertex. Since both are non-empty, the result now follows.
    \begin{figure}
        \centering
        \begin{tikzpicture}[scale = 1.3]
        \begin{scope}
            \node[scale = 0.5, circle, fill=black, draw=black] (v) at (0,0) {};
            \node[scale = 0.5, circle, fill=white, draw=black] (1) at (0:1) {};
            \node[scale = 0.5, circle, fill=white, draw=black] (2) at (40:1) {};
            \node[scale = 0.5, circle, fill=white, draw=black] (3) at (120:1) {};
            \node[scale = 0.5, circle, fill=white, draw=black] (4) at (220:1) {};
            \node[scale = 0.5, circle, fill=white, draw=black] (5) at (300:1) {};
            
            \draw[thick] (v) to [out= 110, in = 180] (0,0.8) to [out=0, in = 70] (v);
            \draw[thick] (v) to [out= 160, in = 90] (-0.8,0) to [out=270, in = 200] (v);

            \node (F) at (90:0.5) {\small{$F$}};
            \node (B) at (250:0.5) {\small{$B$}};
            \draw[thick] (v) -- (1);
            \draw[thick] (v) -- (2);
            \draw[thick] (v) -- (3);
            \draw[thick] (v) -- (4);
            \draw[thick] (v) -- (5);

            \draw[gray, -stealth] (1.5,0) -- (2,0);
        \end{scope}
        \begin{scope}[xshift = 3.25cm]
            \node[scale = 0.5, circle, fill=black, draw=black] (v1) at (-0.25,0) {};
            \node[scale = 0.5, circle, fill=black, draw=black] (v2) at (0.25,0) {};
            \node[scale = 0.5, circle, fill=white, draw=black] (1) at (0:1) {};
            \node[scale = 0.5, circle, fill=white, draw=black] (2) at (40:1) {};
            \node[scale = 0.5, circle, fill=white, draw=black] (3) at (120:1) {};
            \node[scale = 0.5, circle, fill=white, draw=black] (4) at (220:1) {};
            \node[scale = 0.5, circle, fill=white, draw=black] (5) at (300:1) {};
            
            \draw[thick] (v1) to [out= 90, in = 180] (0,0.8) to [out=0, in = 90] (v2);
            \draw[thick] (v1) to [out= 160, in = 90] (-0.8,0) to [out=270, in = 200] (v1);

            \node (F) at (90:0.5) {\small{$F$}};
            \node (B) at (250:0.5) {\small{$B$}};
            \draw[thick] (v1) -- (v2);
            \draw[thick] (v2) -- (1);
            \draw[thick] (v2) -- (2);
            \draw[thick] (v1) -- (3);
            \draw[thick] (v1) -- (4);
            \draw[thick] (v2) -- (5);

            \draw[gray, -stealth] (1.5,0) -- (2,0);
        \end{scope}
        \begin{scope}[xshift = 6.5cm]
            \node[scale = 0.5, circle, fill=black, draw=black] (v) at (0,0) {};
            \node[scale = 0.5, circle, fill=white, draw=black] (1) at (0:1) {};
            \node[scale = 0.5, circle, fill=white, draw=black] (2) at (40:1) {};
            \node[scale = 0.5, circle, fill=white, draw=black] (3) at (120:1) {};
            \node[scale = 0.5, circle, fill=white, draw=black] (4) at (220:1) {};
            \node[scale = 0.5, circle, fill=white, draw=black] (5) at (300:1) {};
            
            \draw[thick] (v) to [out= 250, in = 180] (0,-0.8) to [out=0, in = 290] (v);
            \draw[thick] (v) to [out= 160, in = 90] (-0.8,0) to [out=270, in = 200] (v);

            \node (F) at (270:0.5) {\small{$F$}};
            \node (B) at (235:0.7) {\small{$B$}};
            \draw[thick] (v) -- (1);
            \draw[thick] (v) -- (2);
            \draw[thick] (v) -- (3);
            \draw[thick] (v) -- (4);
            \draw[thick] (v) -- (5);
            \draw[gray, -stealth] (1.5,0) -- (2,0);
        \end{scope}
        \begin{scope}[xshift = 9.75cm]
            \node[scale = 0.5, circle, fill=black, draw=black] (v1) at (0,0) {};
            \node[scale = 0.5, circle, fill=black, draw=black] (v2) at (0,-0.5) {};
            \node[scale = 0.5, circle, fill=white, draw=black] (1) at (0:1) {};
            \node[scale = 0.5, circle, fill=white, draw=black] (2) at (40:1) {};
            \node[scale = 0.5, circle, fill=white, draw=black] (3) at (120:1) {};
            \node[scale = 0.5, circle, fill=white, draw=black] (4) at (220:1) {};
            \node[scale = 0.5, circle, fill=white, draw=black] (5) at (300:1) {};

            \draw[thick] (v1) -- (v2);
             
            \draw[thick] (v2) to [out= 250, in = 180] (0,-1.3) to [out=0, in = 290] (v2);
            \draw[thick] (v1) to [out= 160, in = 90] (-0.8,0) to [out=270, in = 200] (v1);

            \node (F) at (270:1) {\small{$F$}};
            \node (B) at (235:0.7) {\small{$B$}};
            \draw[thick] (v1) -- (1);
            \draw[thick] (v1) -- (2);
            \draw[thick] (v1) -- (3);
            \draw[thick] (v1) -- (4);
            \draw[thick] (v1) -- (5);

            \draw[thick, dashed, blue] (-0.25,-0.3) -- (0.15, -0.3);
        \end{scope}
        \end{tikzpicture}
        \caption{If $F$ is a face whose boundary is a loop edge at $v$ and which is surrounded by another interior face, then $F$ can be ``pushed through" so that it is instead surround by a boundary face $B$. Then $v$ can be uncontracted into a dividing edge.}
        \label{fig: push loop through}
    \end{figure}
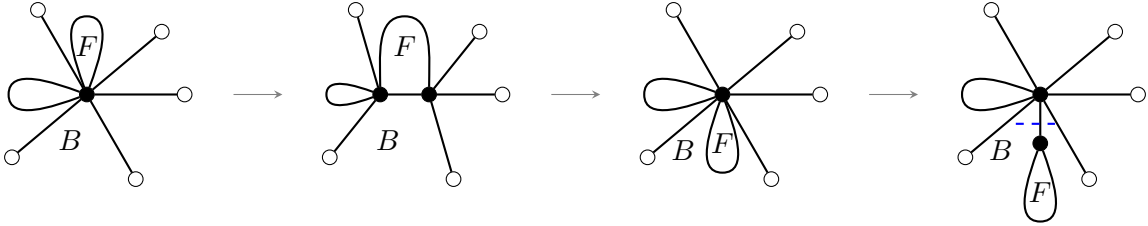

    If instead $v$ is a vertex in $V_i \cap V_j$ for $i \neq j$ but there are no loop edges at $v$, we claim that as one looks locally at $G$ around $v$, at least two (not necessarily distinct) boundary faces should appear, and they should not share an edge coming out of $v$. This holds because, as evident in Figure \ref{fig: cycle at v}, if there are two interior faces with $v$ on their boundary which share an edge coming out of $v$, they must correspond to vertices which are in the same connected component in $Q_G$. Therefore, if there are no boundary faces with $v$ on their boundary, the induced subquiver of $Q_G$ on vertices corresponding to faces with $v$ on their boundary is a cycle. If there is only one boundary face with $v$ on its boundary or all boundary faces appear consecutively in one block as one moves clockwise around $v$, the induced subquiver of $Q_G$ on vertices corresponding to faces with $v$ on their boundary is a path. In both scenarios, all interior faces with $v$ on their boundary would belong to the same connected component of $Q_G$, contradicting our assumption on $v$. 
    
    Therefore, there must be two (not necessarily distinct) boundary faces $B$ and $B'$ which appear around $v$ but such that as one travels clockwise around $v$, at least one interior face appears as one travels from $B$ to $B'$ and from $B'$ to $B$. Without loss of generality, let us say that a face corresponding to a vertex in $Q_i$ appears as one travels clockwise from $B'$ to $B$ and a face corresponding to a vertex in $Q_j$ as one travels from $B$ to $B'$. Now, one can uncontract $v$ as pictured in Figure \ref{fig: two boundary faces} into a dividing edge. One subgraph $G_1$ formed by dividing the graph along this dividing edge has $Q_i$ as a connected component in its quiver while the other subgraph $G_2$ has $Q_j$ as a connected component of its quiver, so both are non-empty. The result now follows in all cases where there is a vertex which is in the intersection $V_i \cap V_j$ for $i \neq j$.
    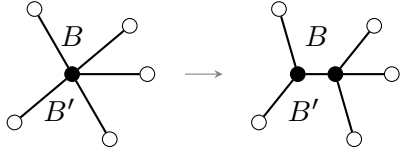
\begin{figure}
        \centering
        \begin{tikzpicture}
            \begin{scope}
            \node[scale = 0.5, circle, fill=black, draw=black] (v) at (0,0) {};
            \node[scale = 0.5, circle, fill=white, draw=black] (1) at (0:1) {};
            \node[scale = 0.5, circle, fill=white, draw=black] (2) at (40:1) {};
            \node[scale = 0.5, circle, fill=white, draw=black] (3) at (120:1) {};
            \node[scale = 0.5, circle, fill=white, draw=black] (4) at (220:1) {};
            \node[scale = 0.5, circle, fill=white, draw=black] (5) at (300:1) {};

            \node (F) at (90:0.5) {\small{$B$}};
            \node (B) at (250:0.5) {\small{$B'$}};
            \draw[thick] (v) -- (1);
            \draw[thick] (v) -- (2);
            \draw[thick] (v) -- (3);
            \draw[thick] (v) -- (4);
            \draw[thick] (v) -- (5);

            \draw[gray, -stealth] (1.5,0) -- (2,0);
        \end{scope}
        \begin{scope}[xshift = 3.25cm]
            \node[scale = 0.5, circle, fill=black, draw=black] (v1) at (-0.25,0) {};
            \node[scale = 0.5, circle, fill=black, draw=black] (v2) at (0.25,0) {};
            \node[scale = 0.5, circle, fill=white, draw=black] (1) at (0:1) {};
            \node[scale = 0.5, circle, fill=white, draw=black] (2) at (40:1) {};
            \node[scale = 0.5, circle, fill=white, draw=black] (3) at (120:1) {};
            \node[scale = 0.5, circle, fill=white, draw=black] (4) at (220:1) {};
            \node[scale = 0.5, circle, fill=white, draw=black] (5) at (300:1) {};

            \node (F) at (90:0.5) {\small{$B$}};
            \node (B) at (250:0.5) {\small{$B'$}};
            \draw[thick] (v1) -- (v2);
            \draw[thick] (v2) -- (1);
            \draw[thick] (v2) -- (2);
            \draw[thick] (v1) -- (3);
            \draw[thick] (v1) -- (4);
            \draw[thick] (v2) -- (5);
        \end{scope}
        \end{tikzpicture}
        \caption{If there are two (not necessarily distinct) boundary faces $B$ and $B'$ which appear around the vertex $v$, then $v$ can be uncontracted into a dividing edge.}
        \label{fig: two boundary faces}
    \end{figure}

    Finally, we consider the case in which the sets $V_1, V_2,\dots, V_k$ are pairwise disjoint. In this case, there can be no loop edges at a vertex $v$ unless all other faces with $v$ on their boundary are boundary faces. In this case, $v$ can be uncontracted into a dividing edge, similar to what is pictured in the last step of Figure \ref{fig: push loop through}. 
    
    Otherwise, as one travels clockwise around $v$, all interior faces must appear consecutively in one block, and all boundary faces must appear consecutively in one block. There also must be at least one boundary face since otherwise there would be a cycle in $Q_G$, as discussed above. See Figure \ref{fig: disjoint no loops} for an example where there are three interior faces $F_1$, $F_2$, and $F_3$ and two boundary faces $B_1$ and $B_2$ which have $v$ on their boundary.
    \begin{figure}
        \centering
        \begin{tikzpicture}
            \begin{scope}
            \node[scale = 0.5, circle, fill=black, draw=black] (v) at (0,0) {};
            \node[scale = 0.5, circle, fill=white, draw=black] (1) at (0:1) {};
            \node[scale = 0.5, circle, fill=white, draw=black] (2) at (40:1) {};
            \node[scale = 0.5, circle, fill=white, draw=black] (3) at (120:1) {};
            \node[scale = 0.5, circle, fill=white, draw=black] (4) at (220:1) {};
            \node[scale = 0.5, circle, fill=white, draw=black] (5) at (300:1) {};

            \node (F3) at (90:0.5) {\small{$F_3$}};
            \node (B1) at (20:0.7) {\small{$B_1$}};
            \node (B2) at (330:0.7) {\small{$B_2$}};
            \node (F1) at (260:0.5) {\small{$F_1$}};
            \node (F2) at (170:0.5) {\small{$F_2$}};
             
            \draw[thick] (v) -- (1);
            \draw[thick] (v) -- (2);
            \draw[thick] (v) -- (3);
            \draw[thick] (v) -- (4);
            \draw[thick] (v) -- (5);

        \end{scope}
        \end{tikzpicture}
        \caption{If a vertex $v$ is only in a single $V_i$ and there are no loop edges at $v$, the interior faces must appear consecutively in one block, as must the boundary faces.}
        \label{fig: disjoint no loops}
    \end{figure}
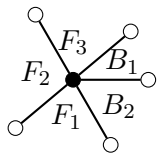
    Let $v_i$ and $v_j$ be distinct vertices in $V_i$ and $V_j$ respectively for some $i \neq j$. Since $G$ is connected, there must be a path between $v_i$ and $v_j$. Fix such a path, and let $v$ be the last vertex in $V_i$ which appears as one travels along this path. Let $e$ be the edge in the path which goes from $v$ to the next vertex $u$. Then $e$ must have two boundary faces on either side of it. If this were not the case, then since $v \in V_i$, $e$ would have an interior face corresponding to a vertex in $Q_i$ one one side of it. This would imply that $u$ is in $V_i$, a contradiction. Therefore, $e$ is a dividing edge; see Figure \ref{fig: disjoint no loop dividing edge}. Observe that the subgraph to the left of this dividing edge as pictured in Figure \ref{fig: disjoint no loop dividing edge} contains $Q_i$ as a connected component of its quiver. On the other hand, the subgraph to the right of this dividing edge has $Q_j$ as one of the connected components of its quiver. Since the quivers of both graphs are therefore non-empty, the proof is complete.
    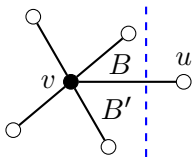
\begin{figure}
        \centering
        \begin{tikzpicture}
            \begin{scope}
            \node[scale = 0.5, circle, fill=black, draw=black] (v) at (0,0) {};
            \node[scale = 0.5, circle, fill=white, draw=black] (1) at (0:1.5) {};
            \node[scale = 0.5, circle, fill=white, draw=black] (2) at (40:1) {};
            \node[scale = 0.5, circle, fill=white, draw=black] (3) at (120:1) {};
            \node[scale = 0.5, circle, fill=white, draw=black] (4) at (220:1) {};
            \node[scale = 0.5, circle, fill=white, draw=black] (5) at (300:1) {};

            \node (B) at (20:0.7) {\small{$B$}};
            \node (B') at (330:0.7) {\small{$B'$}};
            \node (vlab) at (180:0.3) {\small{$v$}};
            \node (ulab) at ($(0:1.5)+(0,0.3)$) {\small{$u$}};
             
            \draw[thick] (v) -- (1);
            \draw[thick] (v) -- (2);
            \draw[thick] (v) -- (3);
            \draw[thick] (v) -- (4);
            \draw[thick] (v) -- (5);

            \draw[thick, dashed, blue] (1,-1) -- (1,1);

        \end{scope}
        \end{tikzpicture}
        \caption{When the sets $V_1,V_2,\dots, V_k$ are pairwise disjoint, there is a dividing edge at a vertex $v \in V_i$ if there is a path from $v$ to a vertex in $V_j$ for $i \neq j$ which does not go through any other vertices in $V_i$.}
        \label{fig: disjoint no loop dividing edge}
    \end{figure}
    
\end{proof}

 In order to utilize Lemma \ref{skeinrelation} to compute the HOMFLY polynomial of a plabic link, we first show that for a connected plabic graph $G$ whose quiver is connected, a leaf $v$ in $Q_G$ corresponds to a boundary leaf face.

\begin{prop}\label{prop:blf}
    Let $G$ be a trivalent, connected plabic graph whose quiver is connected. Suppose $v$ is a leaf in $Q_G$. Then the face $F$ in $G$ which corresponds to $v$ can be assumed to be a boundary leaf face, possibly after applying local move (b). 
\end{prop}
\begin{proof}
 Replace $G$ with its tail reduction. By the assumption that $Q_G$ contains a leaf, $G$ has at least two interior faces, so the tail reduction has no boundary vertices. The leaf $v$ must be connected to exactly one other vertex $\tilde{v} \in Q_G$ which corresponds to some interior face $F'$. Thus, there must be exactly one edge $e$ along the boundary of $F$ with opposite colored endpoints and another interior face on its other side. 
 
 Additionally, we observe that there cannot be more than one boundary face adjacent to $F$. In particular, consider all the faces which are adjacent to $F$. Let us say that there are $n$ of them and label them $F_1=F', F_2,\dots, F_n$ as they appear if one travels in a clockwise order along the boundary of $F$, starting on the edge $e$ that separates $F$ and $F'$. Suppose first that there were an adjacent pair of these faces, $F_i$ and $F_{i+1}$, which are both boundary faces. Any edge that separates two adjacent boundary faces and connects the boundary of $F$ to the boundary of the disk has been removed via the process of tail removal. Therefore, any edge separating $F_i$ and $F_{i+1}$ which has one endpoint on the boundary of $F$ must have its other endpoint on another interior face in $G$. However, this would imply $Q_G$ had more than one connected component, contradicting our assumptions. 
 
 Similarly, suppose there were two nonadjacent boundary faces adjacent to $F$. Pick $i$ and $j$ with $i<j$ which minimize $|i-j|$ such that $F_i$ and $F_j$ are nonadjacent boundary faces which are adjacent to $F$. Then $F_{i+1},F_{i+2},\dots, F_{j-1}$ forms a collection of interior faces such that each pair of consecutive faces are adjacent. Let $1\leq \ell\leq n$ be the minimum value such that $F_\ell$ is a boundary face, and let $1\leq m\leq n$ be the maximum value such that $F_m$ is a boundary face. Then $F_{m+1},\dots, F_n,F_1,F_2,\dots, F_{\ell-1}$ form a different collection of interior faces such that each pair of consecutive faces are adjacent. Since each of these two collections are separated by boundary faces $F_i$ and $F_j$ and there are no edges in $Q_G$ between the face $F$ and any of the faces $F_{i+1},F_{i+2},\dots, F_{j-1}$, it follows that $Q_G$ has at least two connected components, which again forms a contradiction.
 
 It is also possible to have edges along the boundary of $F$ which have another interior face on their other side and whose endpoints have the same color. Since there is one edge, $e$, along the boundary of $F$ with endpoints of different colors, the color of the vertices along the remainder of the boundary of $F$ must change at some other point. Thus there must be at least one other edge with opposite colored endpoints along the boundary of $F$. By the restrictions discussed above, there can only be one such edge, and it must separate $F$ from a boundary face. Therefore, the general form that $G$ has around $F$ is as pictured in the left side of Figure \ref{fig:leaf-face}. 

 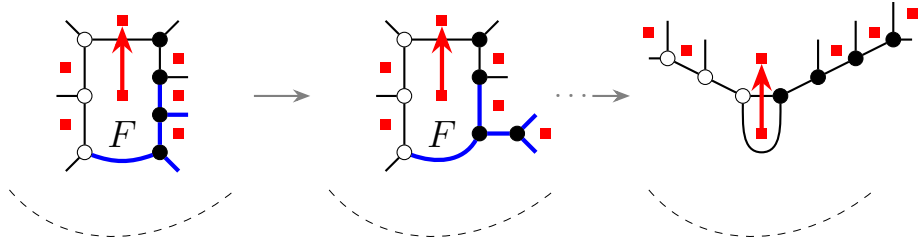
\begin{figure}
    \centering
    \begin{tikzpicture}[scale=0.5]
    \draw [line width=0.3pt, dashed] (-3,0) to [out=310, in=220] (3,0);

    \node [circle, fill=white, draw=black, scale=0.5] (w1) at (-1,4) {};
    \node [circle, fill=black, draw=black, scale=0.5] (b1) at (1,4) {};
    \node [circle, fill=white, draw=black,scale=0.5] (w2) at (-1,2.5) {};
    \node [circle, fill=black, draw=black,scale=0.5] (b2) at (1,3) {};
    \node [circle, fill=white, draw=black,scale=0.5] (w3) at (-1,1) {};
    \node [circle, fill=black, draw=black,scale=0.5] (b3) at (1,2) {};
    \node [circle, fill=black, draw=black, scale=0.5] (b4) at (1,1) {};

    \draw[thick] (w1)--(b1) -- (b2);
    \draw[ultra thick, blue] (b2) -- (b3) -- (b4) to [out = 200, in = 34
    0] (w3);
    \draw[thick] (w3) -- (w2) -- (w1);
    \draw[thick] (w1) -- (-1.5, 4.5);
    \draw[thick] (w2) -- (-1.75, 2.5);
    \draw[thick] (w3) -- (-1.5, 0.5);
    \draw[thick] (b1) -- (1.5, 4.5);
    \draw[thick] (b2) -- (1.75, 3);
    \draw[ultra thick, blue] (b3) -- (1.75, 2);
    \draw[ultra thick, blue] (b4) -- (1.5, 0.5);

    \node [ fill=red, scale=0.5] (Fv) at (0, 2.5) {};
    \node [ fill=red, scale=0.5] (FV) at (0, 4.5) {};
    \node [ fill=red, scale=0.5] (l1) at (-1.5, 3.25) {};
    \node [ fill=red, scale=0.5] (l2) at (-1.5, 1.75) {};
    \node [ fill=red, scale=0.5] (r1) at (1.5, 3.5) {};
    \node [ fill=red, scale=0.5] (r2) at (1.5, 2.5) {};
    \node [ fill=red, scale=0.5] (r3) at (1.5, 1.5) {};
    \draw[ultra thick, red, -Stealth] (Fv) -- (FV);
    \node at (0,1.5) {\large $F$};
    \draw[thick, gray, -Stealth] (3.5, 2.5) -- (5, 2.5); 

    \draw [line width=0.3pt, dashed] (5.5,0) to [out=310, in=220] (11.5,0); 

    \node [circle, fill=white, draw=black, scale=0.5] (w1) at (7.5,4) {};
    \node [circle, fill=black, draw=black, scale=0.5] (b1) at (9.5,4) {};
    \node [circle, fill=white, draw=black,scale=0.5] (w2) at (7.5,2.5) {};
    \node [circle, fill=black, draw=black,scale=0.5] (b2) at (9.5,3) {};
    \node [circle, fill=white, draw=black,scale=0.5] (w3) at (7.5,1) {};
    \node [circle, fill=black, draw=black,scale=0.5] (b3) at (9.5,1.5) {};
    \node [circle, fill=black, draw=black, scale=0.5] (b4) at (10.5,1.5) {};

    \draw[thick] (w1)--(b1) -- (b2);
    \draw[ultra thick, blue] (b2) -- (b3) to [out = 240, in = 34
    0] (w3);
    \draw[ultra thick, blue] (b3) -- (b4);
    \draw[thick] (w3) -- (w2) -- (w1);
    \draw[thick] (w1) -- (7, 4.5);
    \draw[thick] (w2) -- (6.75, 2.5);
    \draw[thick] (w3) -- (7, 0.5);
    \draw[thick] (b1) -- (10, 4.5);
    \draw[thick] (b2) -- (10.25, 3);
    \draw[ultra thick, blue] (b4) -- (11, 2);
    \draw[ultra thick, blue] (b4) -- (11, 1);

    \node [ fill=red, scale=0.5] (Fv) at (8.5, 2.5) {};
    \node [ fill=red, scale=0.5] (FV) at (8.5, 4.5) {};
    \node [ fill=red, scale=0.5] (l1) at (7, 3.25) {};
    \node [ fill=red, scale=0.5] (l2) at (7, 1.75) {};
    \node [ fill=red, scale=0.5] (r1) at (10, 3.5) {};
    \node [ fill=red, scale=0.5] (r2) at (10, 2.25) {};
    \node [ fill=red, scale=0.5] (r3) at (11.25, 1.5) {};
    \draw[ultra thick, red, -Stealth] (Fv) -- (FV);
    \node at (8.5,1.5) {\large $F$};

    \node at (12,2.5) {\textcolor{gray}{$\dots$}};
    \draw[thick, gray, -Stealth] (12.5, 2.5) -- (13.5, 2.5); 

    \draw [line width=0.3pt, dashed] (14,0) to [out=310, in=220] (20,0);

    \node [circle, fill=white, draw=black, scale=0.5] (w1) at (16.5,2.5) {};
    \node [circle, fill=black, draw=black,scale=0.5] (b1) at (17.5,2.5) {};
    \node [circle, fill=white, draw=black, scale=0.5] (w2) at (15.5,3) {};
    \node [circle, fill=black, draw=black, scale=0.5] (b2) at (18.5,3) {};
    \node [circle, fill=white, draw=black, scale=0.5] (w3) at (14.5,3.5) {};
    \node [circle, fill=black, draw=black,scale=0.5] (b3) at (19.5,3.5) {};
    \node [circle, fill=black, draw=black,scale=0.5] (b4) at (20.5,4) {};

    \draw[thick] (w3)--(w2) -- (w1) -- (b1) -- (b2) -- (b3) -- (b4);
    \draw[thick] (b1) to [out = 270, in = 0] (17, 1) to [out=180, in=270]  (w1);
    \draw[thick] (w2) -- (15.5, 4);
    \draw[thick] (w3) -- (14, 3.5);
    \draw[thick] (w3) -- (14.5, 4.5);
    \draw[thick] (b2) -- (18.5, 4);
    \draw[thick] (b3) -- (19.5, 4.5);
    \draw[thick] (b4) -- (20.5, 5);
    \draw[thick] (b4) -- (21, 4);

    \node [ fill=red, scale=0.5] (Fv) at (17, 1.5) {};
    \node [ fill=red, scale=0.5] (FV) at (17, 3.5) {};
    \node [ fill=red, scale=0.5] (l1) at (15, 3.7) {};
    \node [ fill=red, scale=0.5] (l2) at (14, 4.2) {};
    \node [ fill=red, scale=0.5] (r1) at (19, 3.7) {};
    \node [ fill=red,scale=0.5] (r2) at (20, 4.2) {};
    \node [ fill=red, scale=0.5] (r3) at (21, 4.7) {};
    \draw[ultra thick, red, -Stealth] (Fv) -- (FV);

    \end{tikzpicture}
    \caption{A face in $G$ corresponding to a leaf can be taken to be a boundary leaf face after applying local moves.}
    \label{fig:leaf-face}
\end{figure}
 
We apply the trivalent version of local move (b) beginning from the bottom of the picture so that each step, the edge separating $F$ from the boundary face is involved. This is pictured in the first step in Figure \ref{fig:leaf-face} where the edges involved in the move are colored blue. Observe that the number of edges along the boundary of $F$ decreases each time this move is applied. Therefore, after repeatedly applying this move along the boundary of $F$ until it is no longer possible, $F$ will be a boundary leaf face.
\end{proof}

\begin{thm}\label{thm:welldef}
Let $G$ be a connected plabic graph, and suppose the quiver $Q_G$ of $G$ is a forest quiver. Then $P(L_G^\plab) = f(Q_G)$.
\end{thm}
\begin{proof}

Since $G$ is assumed to be connected and simple, we may assume that $G$ is trivalent. Let it be replaced by its tail reduction. We will proceed by induction on the number of vertices in $Q_G$.
    \par If $Q_G$ is the empty quiver, then by Proposition 6.1 of \cite{GLplabiclinks}, one has that plabic graph $G$ and its link $L_G^\plab$, drawn in blue, have the form pictured in Figure \ref{fig:empty quiver}. 
\begin{figure}
    \begin{center}
    \begin{tikzpicture}
        [dot/.style = {circle, draw=black, fill=white, minimum size=3pt, 
              inner sep=0pt, outer sep=0pt} ]  
              
[dot2/.style = {circle, draw=black, fill=white, minimum size=6pt, 
              inner sep=0pt, outer sep=0pt} ]                       
\draw[dotted] (0,0) circle (1);
\node (b1) [dot, minimum size =2pt, fill=black] at(-1, 0) {};
\node (b2) [dot,minimum size =2pt, fill=black] at(1, 0) {};
\node (e1) at (-1.1,0) {};
\node (e2) at (1.1,0) {};
\draw[blue, thick] [postaction={decorate, decoration={markings,
            mark=at position 0.275 with {\arrowreversed[blue]{stealth};}
        }}] (-1,0) to [out=90, in=90] (1,0) to [out=270, in =270] (-1,0);

\draw (b1) -- (b2);

\end{tikzpicture}
\caption{The tail reduction of a plabic graph with empty mutable quiver.}
\label{fig:empty quiver}
\end{center}
\end{figure}
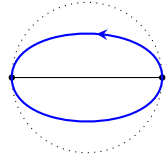
    The link $L_G^\plab$ is the unknot, so $P(L_G^\plab)=f(Q_G)=1$. If $Q_G$ consists of a single isolated vertex then by Proposition 6.1 of \cite{GLplabiclinks}, $G$ and its plabic link have the form pictured in Figure \ref{fig:A1quiver}.
    \begin{figure}
    \centering
        \begin{subfigure}[b]{0.3\textwidth}
        \centering
            \begin{tikzpicture}[scale=0.5, add arrow/.style={postaction={decorate}, decoration={ markings, mark=at position 0.5 with {\arrow[line width=3pt]{<}}}}]
        \draw[line width=0.3pt, dashed] (0,0) circle (3);
        \draw [fill] (-3,0) circle [radius = 0.05];
        \draw [fill] (-1.5,0) circle [radius = 0.1];
        \draw [thick] (-3,0) -- (-1.5,0);
        \draw [thick] (-1.5,0) to [out = 90, in = 180] (0.25, 1.25)  to [out=0, in =up] (2,0) to [out = down, in = 0] (0.25, -1.25) to [out=180, in = down] (-1.5,0);
        \draw [ultra thick, blue, rounded corners] 
         [postaction={decorate, decoration={markings,
            mark=at position 0.75 with {\arrow[blue]{stealth};}
        }}]
        (-1.25, 0) to [out=90, in = 100]  (1.75, 0) to [out=260, in=270] (-1.25, 0);
        \draw [red, ultra thick, rounded corners] 
        [postaction={decorate, decoration={markings,
            mark=at position 0.275 with {\arrowreversed[red]{stealth};}
        }}]
        (-3, 0) to [out = 90, in = 200] (-1.6, 0.6) to [out = 70, in=180] (0.25, 1.7) to [out=0, in = up] (2.2, 0) to [out = down, in = 0] (0.25, -1.7) to [out = 180, in = 290] (-1.6, -0.6) to [out = 160, in = 270]  (-3,0);
        \draw[red, fill=white] (0.375,-1.6) rectangle ++(-0.25cm, -0.25cm);
        \draw[blue, fill=white] (0.375,0.75) rectangle ++(-0.25cm, 0.25cm);

        \draw [-stealth, thick] (0, -3.5) -- (0,-4.5);
        \draw [white, decoration={markings, mark=at position 0.2 with {\arrowreversed[blue, very thick]{stealth}}},
    postaction={decorate}] (0.5, -6) circle (1);
    \draw [white, decoration={markings, mark=at position 0.4 with {\arrow[red, very thick]{stealth}}},
    postaction={decorate}] (-0.5, -7) circle (1.5);
        \begin{knot} [clip width=5, flip crossing = {2}]
            \strand[blue, ultra thick ] 
            (0.5, -6) circle (1);
            \strand[red, ultra thick] (-0.5, -7) circle (1.5);
        \end{knot}
    \end{tikzpicture}
        \end{subfigure}\hspace{1in}
        \begin{subfigure}[b]{0.3\textwidth}
        \centering
            \begin{tikzpicture}[scale=0.5, add arrow/.style={postaction={decorate}, decoration={
  markings,
  mark=at position 0.5 with {\arrow[line width=3pt]{<}}}}]
        \draw[line width=0.3pt, dashed] (0,0) circle (3);
        \node [circle, fill=white, draw=black, scale=0.3] (v) at (-1.5,0) {};
        \draw [fill] (-3,0) circle [radius = 0.05];
        \draw [thick] (-3,0) -- (v);
        \draw [thick] (v) to [out = 90, in = 180] (0.25, 1.25)  to [out=0, in =up] (2,0) to [out = down, in = 0] (0.25, -1.25) to [out=180, in = 270] (v);
        \draw [ultra thick, blue, rounded corners] 
         [postaction={decorate, decoration={markings,
            mark=at position 0.25 with {\arrowreversed[blue]{stealth};}
        }}]
        (-1.25, 0) to [out=90, in = 100]  (1.75, 0) to [out=260, in=270] (-1.25, 0);
        \draw [red, ultra thick, rounded corners] 
        [postaction={decorate, decoration={markings,
            mark=at position 0.70 with {\arrow[red]{stealth};}
        }}]
        (-3, 0) to [out = 270, in = 220] (-2, 0) to [out = 60, in =180] (0.25, 1.7) to [out=0, in = up] (2.2, 0) to [out = down, in = 0] (0.25, -1.7) to [out = 180, in = 300] (-2, 0) to [out = 120, in = 90]  (-3,0);
        \draw[red, fill=white] (0.375,1.6) rectangle ++(-0.25cm, 0.25cm);
        \draw[blue, fill=white] (0.375,-0.75) rectangle ++(-0.25cm, -0.25cm);

        \draw [-stealth, thick] (0, -3.5) -- (0,-4.5);
        \draw [white, decoration={markings, mark=at position 0.2 with {\arrow[blue, very thick]{stealth}}},
    postaction={decorate}] (0, -7.5) circle (1);
    \draw [white, decoration={markings, mark=at position 0.4 with {\arrowreversed[red, very thick]{stealth}}},
    postaction={decorate}] (0.5, -6.5) circle (1.5);
        \begin{knot} [clip width=5, flip crossing = {1}]
            \strand[blue, ultra thick ] 
            (0, -7.5) circle (1);
            \strand[red, ultra thick] (0.5, -6.5) circle (1.5);
        \end{knot}
    \end{tikzpicture}
        \end{subfigure}
        \caption{The only two possibilities for tail reduced graphs whose quiver is a single vertex.}
        \label{fig:A1quiver}
    \end{figure}
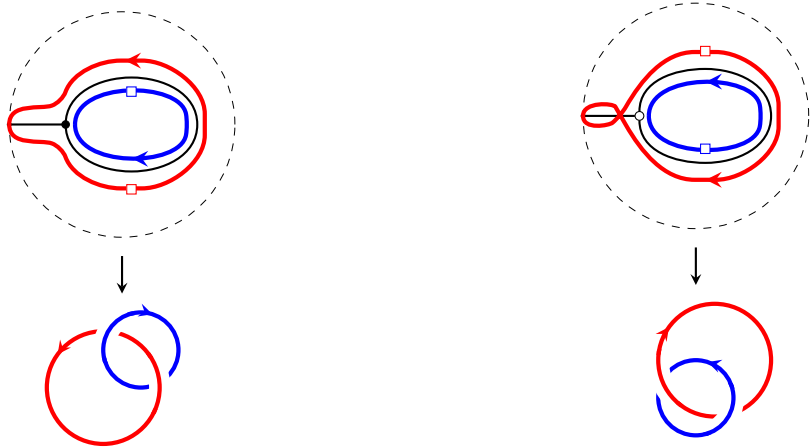
    In both cases, $L_G^\plab$ is the positively oriented Hopf link, so $P(L_G^\plab)=f(Q_G)= \frac{z+z^{-1}}{a}-\frac{z^{-1}}{a^3}$. If $n=2$ and $Q_G$ is connected, then the quiver is the $A_2$ quiver. By the definition of $f$,
    \begin{align*}
        f(A_2) &= \frac{z}{a}f(A_1)+\frac{1}{a^2}f(A_0) \\
        &= \frac{z^2+1}{a^2}-\frac{1}{a^4} +\frac{1}{a^2}\\
        &= \frac{z^2+2}{a^2}-\frac{1}{a^4}
    \end{align*}
    where $A_0$ denotes the empty quiver and $A_1$ consists of a single vertex. On the other hand, the corresponding plabic link is the right-handed trefoil (see Figure \ref{fig:trefoil}), whose HOMFLY polynomial is also $\frac{z^2+2}{a^2}-\frac{1}{a^4}$. 
    \begin{figure}
    \centering
    \begin{tikzpicture}
    \draw [line width=0.3pt, dashed] (-3,0) to [out=310, in=220] (3,0);

    \node [circle, fill=white, draw=black, scale=0.5] (w1) at (-1.5,0.5) {};
    \node [circle, fill=black, draw=black, scale=0.5] (b1) at (1.5,0.5) {};

    \draw[thick] (w1)--(b1);
    \draw[thick] (w1) to [out=90, in = 90] (b1) to [out=270, in =270] (w1);
    \begin{knot}[clip width=2, consider self intersections, end tolerance=0.03pt, flip crossing=2]
        \strand[blue, ultra thick] (0,-0.7) to [out=0, in = 270] (1.8, 0.5) to [out=90, in = 0] (0, 1.3) to [out=180, in =90] (-1.1, 0.9) to [out=270, in =80] (1.2, 0.1) to [out=260, in = 0] (0, -0.3) to [out=180, in = 270] (-1.7, 0.5) to [out=90, in = 180] (0, 1.6) to [out=0, in =90] (1.5, 1) to [out=270, in =30] (1, 0.6) to [out=210, in = 50] (-1.2, 0) to [out=230, in = 180] (0, -0.7);
    \end{knot}
    \draw[blue, thick, fill=white] (-0.1,-0.8) rectangle ++(0.2cm, 0.2cm);
    \draw[blue, thick, fill=white] (-0.1,1.5) rectangle ++(0.2cm, 0.2cm);
    \draw[blue, ultra thick, -stealth] (-1.7, 0.45) -- (-1.7,0.55);
    \draw[blue, ultra thick, -stealth] (1.8, 0.45) -- (1.8,0.55);
     \draw[thick, gray, -stealth] (3.5, 0) -- (4.5,0);
    \end{tikzpicture}
    \begin{tikzpicture}
    \draw [line width=0.3pt, dashed, white] (-2,0) to [out=310, in=220] (2,0);
        \begin{knot}[clip width=2, consider self intersections, end tolerance=0.03pt, flip crossing= 1, flip crossing=3]
        \strand[blue, ultra thick] (-1.2, 0.3) to [out=270, in = 270] (1, 0.5) to [out=90, in = 0] (0, 1.7) to [out=180, in =80] (-1.2, 1.2) to [out=260, in = 180] (1.2, 0.1) to [out=0, in=270] (1.6, 0.5) to [out=90, in = 90] (-1.2, 0.3);
        \draw[blue, ultra thick, -stealth] (1.6, 0.5) -- (1.6,0.6);
        \draw[blue, ultra thick, -stealth] (1, 0.55) -- (1,0.65);
    \end{knot}
    \end{tikzpicture}
    \caption{For a connected plabic graph whose quiver is $A_2$, the plabic link is the right-handed trefoil.}
    \label{fig:trefoil}
\end{figure}
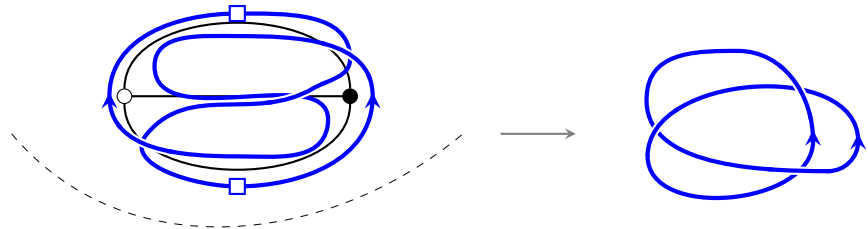
    If $n=2$ but $Q_G$ is disconnected, the quiver must consist of two isolated vertices. By Proposition \ref{prop: disconnected-quiver-connected-sum}, $L_G^\plab$ is a connected sum $L_{G_1}^\plab\# L_{G_2}^\plab$ where $G_1$ and $G_2$ are both connected plabic graphs whose quivers are each a single vertex. By the above portion of the proof, we must have that $P(L_{G_1}^\plab) = P(L_{G_2}^\plab) = \frac{z+z^{-1}}{a}-\frac{z^{-1}}{a^3} = f(Q_{G_1})=f(Q_{G_2})$. Therefore, $P(L_G^\plab) = P(L_{G_1}^\plab)\cdot P(L_{G_2}^\plab) = f(A_1)\cdot f(A_1) = f(Q_G)$. 
    \par For the inductive step, suppose that $Q_G$ is a forest quiver with $n \geq 3$ vertices. We will first handle the case where $Q_G$ is disconnected. Let us say that $Q_G = Q_1 \sqcup Q_2 \sqcup \dots \sqcup Q_k$ where $Q_1,\dots, Q_k$ are non-empty tree quivers. By the inductive hypothesis, given any connected plabic graphs $G_i$ with quiver $Q_i$ for $i=1,\dots, k$, the equality $P(L_{G_i}^\plab) = f(Q_i)$ holds. It then follows from Proposition \ref{prop: disconnected-quiver-connected-sum} that $P(L_G^\plab) = f(Q_1)\cdot f(Q_2) \cdot \dots \cdot f(Q_k) = f(Q_G)$.

    If on the other hand $Q_G$ is connected, then we pick a leaf $v$ in $Q_G$. Let $\tilde{v}$ be the vertex in $Q_G$ which is adjacent to $v$. By Proposition \ref{prop:blf}, this leaf $v$ can be taken to correspond to a boundary leaf face $F$ in $G$. Let $x$ and $y$ be the vertices on the boundary of the face $F$ and $e$ be the edge between $x$ and $y$ separating $F$ from a boundary face. Let $G'=G-e$ and $G''=G-\{x,y\}$. By Lemma \ref{skeinrelation},
    \begin{equation*}
        P(L_G^\plab)=\frac{z}{a}P(L_{G'}^\plab)+\frac{1}{a^2}P(L_{G''}^\plab).
    \end{equation*}
    \begin{figure}
        \centering
        \begin{subfigure}[b]{0.3\textwidth}
            \begin{tikzpicture}[scale=0.6]
                \draw [line width=0.3pt, dashed] (-3,1) to [out=310, in=220] (3,1);
                \node [circle, fill=white, draw=black, scale=0.6] (x) at (-1.5,2) {};
                \node [circle, fill=black, draw=black, scale=0.6] (y) at (1.5,2) {};
                \node [right] at (1.7,2) {\large$y$};
                \node [left] at (-1.7,2) {\large$x$};
                \node [circle, fill=black, draw=black, scale=0.5] (x') at (-2,3) {};
                \node [circle, fill=black, draw=black, scale=0.5] (y') at (2,3) {};
                \node [right] at (2.2,3) {\large$y'$};
                \node [left] at (-2.2,3) {\large$x'$};
                \draw (x') -- (x) -- (y) -- (y');
                \draw (x) to [out=300, in =180] (0, 0.8) to [out=0, in=240] (y);
                \node at (0, 1.15) { $F$};
                \draw (-2.4, 3.5) -- (x') -- (-2,3.5);
                \draw (2.4, 3.5) -- (y') -- (2,3.5);
                \node [below] at (0, 0.8) {\large $e$};
                \node [ fill=red, scale =0.6] (v) at (0, 1.65) {};
                \node [ fill=red, scale =0.6] (V) at (0, 3) {};
                \draw[ultra thick, red, -Stealth] (v) -- (V);
            \end{tikzpicture}
        \end{subfigure}\hspace{1in}
        \begin{subfigure}
            [b]{0.3\textwidth}
            \begin{tikzpicture}[scale=0.6]
                \draw [line width=0.3pt, dashed] (-3,1) to [out=310, in=220] (3,1);
                \node [circle, fill=black, draw=black, scale=0.5] (x') at (-2,3) {};
                \node [circle, fill=black, draw=black, scale=0.5] (y') at (2,3) {};
                \node [right] at (2.2,3) {\large$y'$};
                \node [left] at (-2.2,3) {\large$x'$};
                \draw (x') to [out=300, in=120] (-1.5,2) to  [out=300, in=240]  (1.5,2) to [out=60, in=240] (y');
                \draw (-2.4, 3.5) -- (x') -- (-2,3.5);
                \draw (2.4, 3.5) -- (y') -- (2,3.5);
                \node [ fill=red, scale =0.8] (v) at (0, 2.5) {};
            \end{tikzpicture}
        \end{subfigure}
        \caption{The graphs $G$ (left) and $G'$ (right).}
        \label{fig:primegraphs}
    \end{figure}
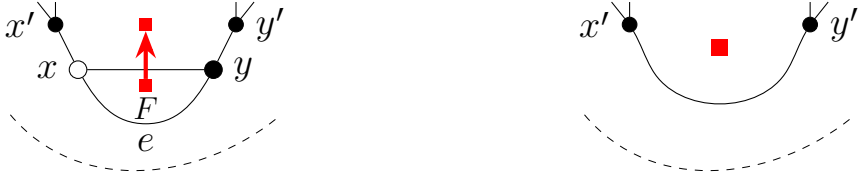
    Let us now consider each of the graphs $G'$ and $G''$ and their quivers. In $G'$, the vertices $x$ and $y$ have degree two. Therefore, we may remove these vertices using local move (c) without affecting the link or quiver. Thus, $G'$ can be taken to have the form pictured in Figure \ref{fig:primegraphs}, and $Q_{G'} = Q-\{v\}$. By the inductive hypothesis, $P(L_{G'}^\plab)=f(Q_{G'})=f(Q_G-\{v\})$.
    \par For $G''$, there are two potential cases to consider. To obtain this graph from $G$, we have removed $x$, $y$, the edges between them, and the edges between $x$ and $x'$ and between $y$ and $y'$. The first potential case is that $x=y'$ and $y='x$ so that $G''$ is empty. This would imply $Q_G=A_2$, which has already been handled above. Therefore, we may assume that $x\neq y'$ and $y\neq x'$. It follows that $G''$ is nonempty and still connected as $x'$ and $y'$ are still connected by the remainder of the boundary of $F'$. Therefore, by the inductive hypothesis $P(L_{G''}^\plab) = f(Q_{G''}) = f(Q_G-\{v, \tilde{v}\}) $. It now follows that
    \begin{align*}
        P(L_G^\plab) &= \frac{z}{a}P(L_{G'}^\plab)+\frac{1}{a^2}P(L_{G''}^\plab)\\
        &= \frac{z}{a}f(Q_G-\{v\}) +\frac{1}{a^2}f(Q_G-\{v,\tilde{v}\})\\
        &= f(Q_G).
    \end{align*}

\end{proof}

\begin{exmp}
    Recall the plabic graphs $G'$ and $G$ pictured in Figure \ref{fig:constructing graph}. The quiver of both graphs is $E_6$, so $$P(L_G^\plab) = P(L_{G'}^\plab) = f(E_6) = \frac{z^6 + 6z^4 + 10z^2 + 5}{a^6} - \frac{z^4 + 5z^2 + 5}{a^8} + \frac{1}{a^{10}}.$$
\end{exmp}

\section{A Closed Formula for the HOMFLY Polynomial and Related Invariants}

In this section we will prove a closed-form expression for the HOMFLY polynomial of a forest quiver and explore some implications of this result for other polynomial invariants.
\subsection{The Closed Formula}
Suppose $Q = Q_1 \sqcup Q_2 \sqcup \dots \sqcup Q_k$ is a forest quiver. In each connected component $Q_\ell$ of $Q$, fix a root vertex $v_\ell$. Let $R = \{v_1,\dots, v_k\}$ denote the set of these root vertices. Given an independent set $I$ in $Q$, let $p(I, R)$ be the number of vertices in $Q$ which are the parent of at least one vertex in $I$ when $Q$ is considered as a rooted forest with roots in $R$. Given $i,j \geq 0$, let $c_{i,j}(Q, R)$ be the number of independent sets $I$ in $Q$ of size $i$ with $p(I, R)=i-j$. We shall see that this quantity does not depend on the choice of root vertices in $R.$

\begin{lem}\label{lem: well def roots}
    Let $Q$ be a forest quiver. Suppose $R$ and $R'$ are two sets consisting of one choice of root vertex in each component of $Q$. Then for all $i,j \geq 0$ the equality $c_{i,j}(Q,R) = c_{i,j}(Q, R')$ holds.
\end{lem}
\begin{proof}
    First we note that there is only one way to choose an independent set of size $0$. It follows that $c_{0,0}(Q,R)=c_{0,0}(Q,R')=1$ and $c_{0,j}(Q,R)=c_{0,j}(Q,R')=0$ for all $j>0$. Therefore, it remains only to prove the desired equality for $i\geq 1$ and $j \geq 0$. We will prove this by induction on the number $n$ of vertices in $Q$. The base cases where $n \leq 2$ are clear due to the symmetry of these quivers. 
    Assume now that $n \geq 3$.
    We first observe that given a quiver $Q$ with root vertices $R=\{v_1,\dots, v_k\}$, the equality
    \begin{equation}\label{eqn: over components}
    c_{i,j}(Q,R) = \sum\limits_{\substack{i_1,j_1,\dots, i_k,j_k \geq 0 \\ i_1+\dots +i_k=i,\, j_1+\dots +j_k=j}}c_{i_1,j_1}(Q_1, v_1)\dots c_{i_k,j_k}(Q_k, v_k)\end{equation}
    holds. The right-hand side of the equality simply ranges over all ways to allocate the number of vertices in an independent set and their parents between components of $Q$. Therefore, if $Q$ has multiple connected components, the statement follows by the inductive hypothesis. It remains only to prove the statement of the lemma in the case where $Q$ is a tree. That is, we wish to show that given a tree quiver $Q$ and any two vertices $u,w$ in $Q$, one has
    \begin{align}\label{tree equaltiy}
        c_{i,j}(Q,u) = c_{i,j}(Q,w).
    \end{align}
    Fix two such vertices $u$ and $w$ in a tree quiver $Q$. Without loss of generality, we can assume that there exists another leaf $v$ in $Q$ which is distinct from both $u$ and $w$. Since any tree has at least two leaves, the only situation in which it would be impossible to pick such a leaf $v$ would be if $Q$ were a path and $u$ and $w$ were the two leaves. In this case the equality (\ref{tree equaltiy}) holds by symmetry. Fix such a leaf $v$, and let $\tilde{v}$ by the vertex adjacent to $v$. There are two distinct cases to consider according to whether or not $\tilde{v}$ is the same as either $u$ or $w$. 
    
    First let us handle the case where neither $u$ nor $w$ is equal to $\tilde{v}$ so that all the vertices $u$, $w$, $v$, and $\tilde{v}$ are distinct. Let $Q' = Q-\{v\}$ and $Q'' = Q-\{v,\tilde{v}\}$. The quiver $Q'$ is a tree, and one can still choose either $u$ or $w$ to be the root in this quiver. However, it is possible that $Q''$ is disconnected, so we must choose root vertices in each component. Let $R_u$ consist of the vertex $u$ along with any other children of $\tilde{v}$ besides $v$ when $Q$ is rooted at $u$. Define $R_w$ similarly. By the inductive hypothesis we also have that the equalities
    \begin{align}\label{eqn: inductivewelldef}
        c_{i,j}(Q',u) = c_{i,j}(Q',w), \qquad c_{i,j}(Q'',R_u) = c_{i,j}(Q'',R_w)
    \end{align}
    hold for all $i,j \geq 0$.
    We next observe that for all $i\geq 1$ and $j \geq 0$, we have
    \begin{equation}\label{eqn:dividing indep sets}
        c_{i,j}(Q, u) = c_{i,j}(Q', u) + c_{i-1, j}(Q'', R_u), \quad
        c_{i,j}(Q, w) = c_{i,j}(Q', w) + c_{i-1, j}(Q'', R_w).
    \end{equation}
    To see this, we note that the sets counted by $c_{i,j}(Q,u)$ can be divided depending on whether or not they include $v$. Those that do not can be considered as independent sets in $Q'$, and when $Q'$ is rooted at $u$ the number of parent vertices remains the same. Therefore, the number of such independent sets is $c_{i,j}(Q', u)$. For independent sets $I$ which contain $v$, consider the independent set $I'=I-\{v\}$ of size $i-1$ in $Q''$. 
    \begin{figure}
        \centering
        \begin{tikzpicture}
            \begin{scope}
         \node [circle, draw=black, fill=green, scale =0.3] (u) at (0,0) {};
         \draw (0,0) circle (0.12);
     \node [circle, draw=black, fill=red, scale =0.3] (1) at (-0.5, -0.5) {};
     \node [circle, draw=black, fill=green, scale =0.3] (vtilde) at (0.5, -0.5) {};
     \node [circle, draw=black, fill=red, scale =0.3] (v) at (0, -1) {};
     \node [circle, draw=black, fill=green, scale =0.3] (v1) at (0.5, -1.25) {};
     \node [circle, draw=black, fill=red, scale =0.3] (v2) at (1, -1) {};
     \node [circle, draw=black, fill=red, scale =0.3] (2) at (0, -1.75) {};
     \node [circle, draw=black, fill=black, scale =0.3] (3) at (1, -1.75) {};
     \draw [very thick] (1)--(u)--(vtilde);
     \draw [very thick] (v)--(vtilde);
     \draw [very thick] (v1) -- (vtilde) -- (v2);
     \draw [very thick] (2) -- (v1) -- (3);
     \node at (0,0.25) {$u$};
     \node at (-0.25, -1) {$v$};
     \node at (0.85, -0.45) {$\tilde{v}$};
     \node at (1.35, -1) {$v_2$};
     \node at (0.8, -1.25) {$v_1$};
        \end{scope}
   \begin{scope}[xshift = 4cm]
         \node [circle, draw=black, fill=green, scale =0.3] (u) at (0,0) {};
         \draw (0,0) circle (0.12);
     \node [circle, draw=black, fill=red, scale =0.3] (1) at (-0.5, -0.5) {};
     \node [circle, draw=black, fill=green, scale =0.3] (v1) at (0.25, -1) {};
     \draw (0.25,-1) circle (0.12);
     \node [circle, draw=black, fill=red, scale =0.3] (v2) at (1, -1) {};
     \draw (1,-1) circle (0.12);
     \node [circle, draw=black, fill=red, scale =0.3] (2) at (-0.25, -1.5) {};
     \node [circle, draw=black, fill=black, scale =0.3] (3) at (0.75, -1.5) {};
     \draw [very thick] (1)--(u);
     \draw [very thick] (2) -- (v1) -- (3);
     \node at (0,0.25) {$u$};
     \node at (1.35, -1) {$v_2$};
     \node at (-0.1, -1) {$v_1$};
        \end{scope}
        \end{tikzpicture}

        \caption{An independent set $I$ in $Q$ (left) which contains the leaf $v$ along with the corresponding independent set $I'$ in $Q''$ (right). The root $u$ in $Q$ and the roots in $R_u$ in $Q''$ are circled, elements of the independent sets are colored red, and the parent vertices are colored green.}
        \label{fig:case1well-def}
    \end{figure}
    We note the parent vertices of the vertices in $I'$ in $Q''$ are exactly the parent vertices of the vertices in $I$ except for $\tilde{v}$. See Figure \ref{fig:case1well-def} for an example. This implies that $I'$ is an independent set of size $i-1$ with $i-j-1$ parent vertices. This map from independent sets in $Q$ containing $v$ and independent sets in $Q''$ is easily reversible. It follows that the number of independent sets counted by $c_{i,j}(Q,u)$ that do contain $v$ is exactly $c_{i-1, j}(Q'', R_u)$. We conclude that the equations in (\ref{eqn:dividing indep sets}) hold. Combining these with (\ref{eqn: inductivewelldef}), we find that
    \begin{align*}
        c_{i,j}(Q,u) &= c_{i,j}(Q', u) + c_{i-1, j}(Q'', R_u)\\
        &= c_{i,j}(Q', w) + c_{i-1, j}(Q'', R_w) \\
        &= c_{i,j}(Q,w).
    \end{align*}

    Now we will handle the case where $\tilde{v}$ is equal to one of $u$ or $v$. Without loss of generality, say $\tilde{v}=u$. This case will proceed similarly to the case where $u$, $w$, $v$, and $\tilde{v}$ are all distinct. We will once again consider the quivers $Q' = Q-\{v\}$ and $Q'' = Q-\{v, \tilde{v}\}$. The quiver $Q'$ is still a tree and can still be rooted at either $u=\tilde{v}$ or $w$. In $Q''$, the vertex $u$ has been removed and can no longer be used as a root. Therefore, we will instead choose the set of roots in each component of $Q''$, denoted $R_u$, to be the set of children $v_i \neq v$ of $u=\tilde{v}$ in $Q$ when rooted at $u$. See Figure \ref{fig:case2welldef} for an example. We will define $R_w$, as in the previous case, to be the set of $w$ together with each child $v_i \neq v$ of $u=\tilde{v}$ in $Q$ when rooted at $w$. 
    \begin{figure}
        \centering
        \begin{tikzpicture}
            \begin{scope}
     \node [circle, draw=black, fill=green, scale =0.3] (vtilde) at (0.5, -0.5) {};
     \draw (0.5, -0.5) circle (0.12);
     \node [circle, draw=black, fill=red, scale =0.3] (v) at (0, -1) {};
     \node [circle, draw=black, fill=green, scale =0.3] (v1) at (0.5, -1.25) {};
     \node [circle, draw=black, fill=red, scale =0.3] (v2) at (1, -1) {};
     \node [circle, draw=black, fill=red, scale =0.3] (2) at (0.25, -1.75) {};
     \node [circle, draw=black, fill=black, scale =0.3] (3) at (0.75, -1.75) {};
     \node [circle, draw=black, fill=green, scale =0.3] (4) at (1.5, -1.5) {};
      \node [circle, draw=black, fill=black, scale =0.3] (5) at (1.25, -2) {};
       \node [circle, draw=black, fill=red, scale =0.3] (6) at (1.75, -2) {};
     \draw [very thick] (v)--(vtilde);
     \draw [very thick] (v1) -- (vtilde) -- (v2);
     \draw [very thick] (2) -- (v1) -- (3);
     \draw [very thick] (v2) -- (4);
     \draw [very thick] (5) -- (4)-- (6);
     \node at (-0.25, -1) {$v$};
     \node at (0.5, -0.25) {$u=\tilde{v}$};
     \node at (1.35, -1) {$v_2$};
     \node at (0.8, -1.25) {$v_1$};
        \end{scope}
        \begin{scope}[xshift = 4cm]
     \node [circle, draw=black, fill=green, scale =0.3] (v1) at (0.5, -1.25) {};
     \draw (0.5, -1.25) circle (0.12);
     \node [circle, draw=black, fill=red, scale =0.3] (v2) at (1.5, -1.15) {};
     \draw (1.5, -1.15) circle (0.12);
     \node [circle, draw=black, fill=red, scale =0.3] (2) at (0.25, -1.75) {};
     \node [circle, draw=black, fill=black, scale =0.3] (3) at (0.75, -1.75) {};
     \node [circle, draw=black, fill=green, scale =0.3] (4) at (1.5, -1.5) {};
      \node [circle, draw=black, fill=black, scale =0.3] (5) at (1.25, -2) {};
       \node [circle, draw=black, fill=red, scale =0.3] (6) at (1.75, -2) {};

     \draw [very thick] (2) -- (v1) -- (3);
     \draw [very thick] (v2) -- (4);
     \draw [very thick] (5) -- (4)-- (6);
     \node at (1.85, -1.15) {$v_2$};
     \node at (0.85, -1.25) {$v_1$};
        \end{scope}
        \end{tikzpicture}
        \caption{An independent set $I$ in $Q$ (left) which contains the leaf $v$ along with the corresponding independent set $I'$ in $Q''$ (right) in the case where $u=\tilde{v}$. The root $u=\tilde{v}$ in $Q$ and the roots in $R_u$ in $Q''$ are circled, elements of the independent sets are colored red, and the parent vertices are colored green.}
        \label{fig:case2welldef}
    \end{figure}
    
    By the inductive hypothesis the equations in (\ref{eqn: inductivewelldef}) still hold. By a similar logic to that used in the case above, we see that the equations in (\ref{eqn:dividing indep sets}) hold as well and conclude that $c_{i,j}(Q,u) = c_{i,j}(Q,w)$.
    \end{proof} 

Given the invariance of $c_{i,j}(Q,R)$ under the choice of roots, we will generally omit the set of roots $R$ from the notation and simply write $c_{i,j}(Q)$. Using this, we can now state a closed formula for the HOMFLY polynomial of $Q$.

\begin{thm}\label{thm:homflyformula}
Given a forest quiver $Q$ with $n$ vertices, the HOMFLY polynomial of $Q$ is given by
    \begin{equation}\label{eqn:homflyformula}
        f(Q) =\frac{1}{a^n}\left(\sum\limits_{i,j \geq 0} c_{i,j}(Q) z^{n-2i} (1- a^{-2})^j \right)
    \end{equation}
\end{thm}
\begin{proof}
    We will first verify that this formula satisfies the condition that $f(Q)=\prod\limits_{\ell=1}^k f(Q_\ell)$ for a forest quiver $Q = Q_1 \sqcup Q_2 \sqcup \dots \sqcup Q_k$.
    Let $n_\ell$ be the number of vertices in the connected component $Q_\ell$ for $\ell=1,\dots, k$. We must verify that
    \begin{align*}
       \frac{1}{a^n}\left(\sum\limits_{i,j \geq 0} c_{i,j}(Q) z^{n-2i} (1- a^{-2})^j \right) &=  \prod\limits_{\ell=1}^k \frac{1}{a^{n_\ell}}\left(\sum\limits_{i,j \geq 0} c_{i,j}(Q_\ell) z^{n_\ell-2i} (1- a^{-2})^j \right).
    \end{align*}
    By re-expressing the right hand side of the above equation as 
    \begin{align*}
        \frac{1}{a^{n_1+n_2+\dots +n_k}} &\sum\limits_{i_1,j_1,\dots, i_k,j_k \geq 0} c_{i_1,j_1}(Q_1)\dots c_{i_k,j_k}(Q_k)z^{(n_1+n_2+\dots n_k)-2(i_1+i_2+\dots+i_k)}(1-a^{-2})^{j_1+j_2+\dots+j_k}
        \\&= \frac{1}{a^{n}} \sum\limits_{i_1,j_1,\dots, i_k,j_k \geq 0} c_{i_1,j_1}(Q_1)\dots c_{i_k,j_k}(Q_k)z^{n-2(i_1+i_2+\dots+i_k)}(1-a^{-2})^{j_1+j_2+\dots+j_k}
    \end{align*}
    we see that it suffices to show that for all $i,j \geq 0$ the equality 
    \begin{align*}
        c_{i,j}(Q) = \sum\limits_{\substack{i_1,j_1,\dots, i_k,j_k \geq 0 \\ i_1+\dots +i_k=i,\, j_1+\dots +j_k=j}}c_{i_1,j_1}(Q_1)\dots c_{i_k,j_k}(Q_k)
    \end{align*}
    holds. This is the same as the equality (\ref{eqn: over components}) which we established in the proof of Lemma \ref{lem: well def roots}.

    Now we will proceed by induction on the number of vertices in $Q$. If $n=1$, then we have $f(A_1)= \frac{z+z^{-1}}{a}-\frac{z^{-1}}{a^3}$. Meanwhile, since $Q$ consists of a single vertex, the only two independent sets are the set of zero vertices and the set consisting of the only vertex. It follows that the only non-zero coefficients $c_{i.j}(Q)$ are $c_{0,0}(Q)=1$ and $c_{1,1}(Q)=1$. Therefore, the formula in (\ref{eqn:homflyformula}) becomes $ \frac{1}{a^1}\left( z^1(1-a^{-2})^0 + z^{-1}(1-a^{-2})  \right) = \frac{z+z^{-1}}{a}
    -\frac{z^{-1}}{a^3}$. 
    
    When $n=2$ and $Q$ consists of two disconnected vertices, then the above argument shows that the formula in (\ref{eqn:homflyformula}) agrees with $f(A_1)\cdot f(A_1)$ as expected. When $n=2$ and $Q=A_2$, the HOMFLY polynomial is $f(A_2)= \frac{z^2+2}{a^2}-\frac{1}{a^4}$. On the other hand, there are three independent sets in $Q$. They are the set of zero vertices, the set consisting of the non-root vertex, and the set consisting of the root vertex. Therefore, the only non-zero coefficients are $c_{0,0}(Q) = 1$, $c_{1,0}(Q)=1$, and $c_{1,1}(Q)=1$. We can then verify that
    \begin{align*}
         \frac{1}{a^2}\left(\sum\limits_{i,j \geq 0} c_{i,j}(A_2) z^{2-2i} (1- a^{-2})^j \right) &= \frac{1}{a^2} \left( 1\cdot z^2 + 1\cdot z^0 (1-a^{-2})^0 + 1\cdot z^0(1-a^{-2})^1 \right)\\
         &= \frac{1}{a^2}\left(z^2+1+1-a^{-2} \right)\\
         &=\frac{z^2+2}{a^2}-\frac{1}{a^4}.
    \end{align*}

    Now assume that $Q$ has $n\geq 3$ vertices. Since we have handled the case where $Q$ is disconnected above, we may assume that $Q$ is connected. Fix a leaf $v$ in $Q$ which is incident to a vertex $\tilde{v}$. We will choose a root vertex $u$ which is not $v$ or $\tilde{v}$. By the inductive hypothesis, we assume that the formula in (\ref{eqn:homflyformula}) holds for $Q' = Q-\{v\}$ and $Q''=Q-\{v, \tilde{v}\}$. That is,

    $$f(Q') = \frac{1}{a^{n-1}}\left(\sum\limits_{i,j \geq 0} c_{i,j}(Q') z^{n-1-2i} (1- a^{-2})^j \right)$$
    and
    $$f(Q'') = \frac{1}{a^{n-2}}\left(\sum\limits_{i,j \geq 0} c_{i,j}(Q'') z^{n-2(i+1)} (1- a^{-2})^j \right).$$
    Since $f(Q) = \frac{z}{a}f(Q') +\frac{1}{a^2}f(Q'')$ it remains to show that
    \begin{align*}
        \frac{1}{a^n}\left(\sum\limits_{i,j \geq 0} c_{i,j}(Q) z^{n-2i} (1- a^{-2})^j \right) &= \frac{z}{a^{n}}\left(\sum\limits_{i,j \geq 0} c_{i,j}(Q') z^{n-1-2i} (1- a^{-2})^j \right) \\
        &\qquad + \frac{1}{a^{n}}\left(\sum\limits_{i,j \geq 0} c_{i,j}(Q'') z^{n-2(i+1)} (1- a^{-2})^j \right)\\
        &= \frac{1}{a^{n}}\left(\sum\limits_{i,j \geq 0} c_{i,j}(Q') z^{n-2i} (1- a^{-2})^j \right) \\
        &\qquad + \frac{1}{a^{n}}\left(\sum\limits_{i \geq 1, \; j \geq 0} c_{i-1,j}(Q'') z^{n-2i} (1- a^{-2})^j \right)\\
        &= \frac{1}{a^n}\left(z^n + \sum\limits_{i \geq 1, \; j \geq 0} \left(c_{i,j}(Q')+c_{i-1,j}(Q'')\right) z^{n-2i} (1- a^{-2})^j \right)
    \end{align*}
    Since $c_{0,0}(Q)=1$, the coefficients on the $\frac{z^n}{a^n}$ terms agree. It remains to show that for $i\geq 1$ and $j \geq 0$ one has $$c_{i,j}(Q) = c_{i,j}(Q')+c_{i-1,j}(Q'').$$
    This follows from the equations in (\ref{eqn:dividing indep sets}) which we established in the proof of Lemma \ref{lem: well def roots}.
\end{proof}

\begin{figure}
         \begin{center}
         \begin{tikzpicture}[scale =0.45]
        \begin{scope}
            \node () at (-3.5,0.5) {$c_{0,0}(D_4)=1$};
         \node [circle, draw=black, fill=black, scale =0.3] (1) at (0,0.25) {};
     \node [circle, draw=black, fill=black, scale =0.3] (2) at (1, 0.5) {};
     \node [circle, draw=black, fill=black, scale =0.3] (3) at (2, 0.25) {};
     \node [circle, draw=black, fill=black, scale =0.3] (6) at (1, 1.25) {};
     \draw [thin, gray] (1,1.25) circle (0.3);
     \draw [very thick] (1)--(2)--(3);
     \draw [very thick] (2)--(6);   
        \end{scope}
        \draw [very thick, gray, dashed]  (-6.5, -0.5) -- (10.5, -0.5);
        \begin{scope}[yshift = -3cm]
        \begin{scope}
        \begin{scope}
            \node () at (-3.5,0.5) {$c_{1,0}(D_4)=3$};
         \node [circle, draw=black, fill=red, scale =0.3] (1) at (0,0.25) {};
     \node [circle, draw=black, fill=green, scale =0.3] (2) at (1, 0.5) {};
     \node [circle, draw=black, fill=black, scale =0.3] (3) at (2, 0.25) {};
     \node [circle, draw=black, fill=black, scale =0.3] (6) at (1, 1.25) {};
     \draw [thin, gray] (1,1.25) circle (0.3);
     \draw [very thick] (1)--(2)--(3);
     \draw [very thick] (2)--(6);   
        \end{scope}
        \draw [very thick, gray, dashed]  (-6.5, -0.5) -- (10.5, -0.5);
        \end{scope}

        \begin{scope}[xshift = 4cm]
         \begin{scope}
         \node [circle, draw=black, fill=black, scale =0.3] (1) at (0,0.25) {};
     \node [circle, draw=black, fill=red, scale =0.3] (2) at (1, 0.5) {};
     \node [circle, draw=black, fill=black, scale =0.3] (3) at (2, 0.25) {};
     \node [circle, draw=black, fill=green, scale =0.3] (6) at (1, 1.25) {};
     \draw [thin, gray] (1,1.25) circle (0.3);
     \draw [very thick] (1)--(2)--(3);
     \draw [very thick] (2)--(6);   
        \end{scope}
        \end{scope}

        \begin{scope}[xshift = 8cm]
         \begin{scope}
         \node [circle, draw=black, fill=black, scale =0.3] (1) at (0,0.25) {};
     \node [circle, draw=black, fill=green, scale =0.3] (2) at (1, 0.5) {};
     \node [circle, draw=black, fill=red, scale =0.3] (3) at (2, 0.25) {};
     \node [circle, draw=black, fill=black, scale =0.3] (6) at (1, 1.25) {};
     \draw [thin, gray] (1,1.25) circle (0.3);
     \draw [very thick] (1)--(2)--(3);
     \draw [very thick] (2)--(6);   
        \end{scope}
        \end{scope}

        \end{scope}

    \begin{scope}[ yshift = -6cm]
            \begin{scope}
        \begin{scope}
            \node () at (-3.5,0.5) {$c_{1,1}(D_4)=1$};
         \node [circle, draw=black, fill=black, scale =0.3] (1) at (0,0.25) {};
     \node [circle, draw=black, fill=black, scale =0.3] (2) at (1, 0.5) {};
     \node [circle, draw=black, fill=black, scale =0.3] (3) at (2, 0.25) {};
     \node [circle, draw=black, fill=red, scale =0.3] (6) at (1, 1.25) {};
     \draw [thin, gray] (1,1.25) circle (0.3);
     \draw [very thick] (1)--(2)--(3);
     \draw [very thick] (2)--(6);   
        \end{scope}
        \end{scope}
    \end{scope}

    \begin{scope}[ xshift = 17.5cm]
             \begin{scope}
        \begin{scope}
            \node () at (-3.5,0.5) {$c_{2,1}(D_4)=3$};
         \node [circle, draw=black, fill=red, scale =0.3] (1) at (0,0.25) {};
     \node [circle, draw=black, fill=green, scale =0.3] (2) at (1, 0.5) {};
     \node [circle, draw=black, fill=red, scale =0.3] (3) at (2, 0.25) {};
     \node [circle, draw=black, fill=black, scale =0.3] (6) at (1, 1.25) {};
     \draw [thin, gray] (1,1.25) circle (0.3);
     \draw [very thick] (1)--(2)--(3);
     \draw [very thick] (2)--(6);   
        \end{scope}
        \end{scope}

        \begin{scope}[xshift = 4cm]
         \begin{scope}
         \node [circle, draw=black, fill=red, scale =0.3] (1) at (0,0.25) {};
     \node [circle, draw=black, fill=green, scale =0.3] (2) at (1, 0.5) {};
     \node [circle, draw=black, fill=black, scale =0.3] (3) at (2, 0.25) {};
     \node [circle, draw=black, fill=red, scale =0.3] (6) at (1, 1.25) {};
     \draw [thin, gray] (1,1.25) circle (0.3);
     \draw [very thick] (1)--(2)--(3);
     \draw [very thick] (2)--(6);   
        \end{scope}
        \end{scope}

        \begin{scope}[xshift = 8cm]
         \begin{scope}
         \node [circle, draw=black, fill=black, scale =0.3] (1) at (0,0.25) {};
     \node [circle, draw=black, fill=green, scale =0.3] (2) at (1, 0.5) {};
     \node [circle, draw=black, fill=red, scale =0.3] (3) at (2, 0.25) {};
     \node [circle, draw=black, fill=red, scale =0.3] (6) at (1, 1.25) {};
     \draw [thin, gray] (1,1.25) circle (0.3);
     \draw [very thick] (1)--(2)--(3);
     \draw [very thick] (2)--(6);   
        \end{scope}
        \end{scope}
        \draw [very thick, gray, dashed]  (-6.5, -0.5) -- (10.5, -0.5);
    \end{scope}

     \begin{scope}[ xshift = 17.5cm, yshift=-3cm]
             \begin{scope}
        \begin{scope}
            \node () at (-3.5,0.5) {$c_{3,2}(D_4)=1$};
         \node [circle, draw=black, fill=red, scale =0.3] (1) at (0,0.25) {};
     \node [circle, draw=black, fill=green, scale =0.3] (2) at (1, 0.5) {};
     \node [circle, draw=black, fill=red, scale =0.3] (3) at (2, 0.25) {};
     \node [circle, draw=black, fill=red, scale =0.3] (6) at (1, 1.25) {};
     \draw [thin, gray] (1,1.25) circle (0.3);
     \draw [very thick] (1)--(2)--(3);
     \draw [very thick] (2)--(6);   
        \end{scope}
        \end{scope}
       
        \draw [very thick, gray, dashed]  (-6.5, -0.5) -- (10.5, -0.5);
    \end{scope}
    \end{tikzpicture}
    \end{center}
    \caption{The subsets counted by the coefficients $c_{i,j}(D_4)$. The root vertex is circled, elements of the independent sets are colored red, and the parent vertices are colored green.}
    \label{fig:D4 homfly}
\end{figure}
\begin{exmp}
    Consider the Dynkin diagram $D_4$. The nonzero coefficients $c_{i,j}(D_4)$ are shown in Figure \ref{fig:D4 homfly}. The HOMFLY polynomial $f(D_4)$ is
    \begin{align*}
        f(D_4) &= \frac{1}{a^4}\left(z^4+3z^2+z^2(1-a^{-2})+3(1-a^{-2})+z^{-2}(1-a^{-2})^2\right)\\
        &=\frac{z^4 + 4z^2 +  3 +z^{-2} }{a^4} -\frac{z^2 +3 + 2z^{-2}}{a^6} + \frac{z^{-2}}{a^8}.
    \end{align*} 
\end{exmp}

\subsection{The Point Count $R$-polynomial}
In \cite{lam2021cohomologyclustervarietiesii}, Lam and Speyer studied the point count of acyclic cluster varieties over finite fields.

\begin{prop}(\cite{lam2021cohomologyclustervarietiesii}, Proposition 3.9)
    For an acyclic quiver $Q$ with $n$ vertices the point count over $\F_q$ of the associated cluster algebra $\mathcal{A}$ is given by
    \begin{equation}\label{eqn:point count}
        \#\mathcal{A}(\F_q) = \sum\limits_{i \geq 0} a_i(Q)q^i(q-1)^{n-2i}
    \end{equation}
where $a_i(Q)$ is the number of independent sets in the underlying graph of the quiver.
\end{prop}
Galashin and Lam established a relation between this point count polynomial, which they denoted $R(Q;q)$, and the HOMFLY polynomial. In particular, given a link $L$, they defined $P^{\textrm{top}}(L;q)$ to be obtained from the top $a$-degree term of $P(L)$ via the substitutions $a=q^{-1/2}$ and $z=q^{1/2}-q^{-1/2}$. They then proved the following result for a class of plabic graphs called leaf recurrent plabic graphs, which includes reduced plabic graphs and plabic fences. 

\begin{thm}(\cite{GLplabiclinks}, Theorem 2.9) Let $G$ be a simple, leaf recurrent plabic graph with $c(G)$ connected components. Then we have
$$R(Q_G;q) = (q-1)^{c(G)-1}P^{\textrm{top}}(L_G^{\textrm{plab}};q).$$
    
\end{thm}
The HOMFLY polynomial formula in (\ref{eqn:homflyformula}) recovers the formula of Lam and Speyer in the case where $Q_G$ is a forest and $G$ is a simple, leaf recurrent, connected plabic graph. In this case, the top $a$-degree of $P(L_G^{\textrm{plab}})$ is given by 
\begin{align*}\frac{1}{a^n}\left(\sum\limits_{i,j \geq 0} c_{i,j}(Q) z^{n-2i}\right) &= \frac{1}{a^n}\left(\sum\limits_{i \geq 0}\left( \sum\limits_{j \geq 0}c_{i,j}(Q)\right) z^{n-2i}\right).
\end{align*}
Under the substitutions $a=q^{-1/2}$ and $z=q^{1/2}-q^{-1/2}$, this becomes
\begin{align*}
   q^{n/2}\left(\sum\limits_{i \geq 0}\left( \sum\limits_{j \geq 0}c_{i,j}(Q)\right) (q^{1/2}-q^{-1/2})^{n-2i}\right)
   &=  q^{n/2}\left(\sum\limits_{i \geq 0}\left( \sum\limits_{j \geq 0}c_{i,j}(Q)\right) (q^{-1/2})^{n-2i}(q-1)^{n-2i}\right)\\
   &= q^{n/2}\left(\sum\limits_{i \geq 0}\left( \sum\limits_{j \geq 0}c_{i,j}(Q)\right) (q^{-n/2})q^{i}(q-1)^{n-2i}\right)\\
   &=\sum\limits_{i \geq 0}\left( \sum\limits_{j \geq 0}c_{i,j}(Q)\right) q^{i}(q-1)^{n-2i}\\
   &=\sum\limits_{i \geq 0}a_i(Q) q^{i}(q-1)^{n-2i}.
\end{align*}
\subsection{The Alexander Polynomial}
\begin{cor}\label{cor:alexander}
    Given a forest quiver $Q$ with $n$ vertices, the Alexander polynomial of $Q$ is given by
    \begin{equation}\label{eqn:alexanderformula}
        \Delta(Q) =t^{-n/2}\cdot\sum\limits_{i \geq 0} b_{i}(Q) t^i(t-1)^{n-2i} 
    \end{equation}
    where $b_i(Q)$ is the number of ways to choose $i$ distinct edges in $Q$ which do not share any endpoints. Alternatively, $b_i(Q)$ is the number of independent sets of size $i$ in the line graph of $Q$.
\end{cor}
\begin{proof}
    The Alexander polynomial is obtained from the HOMFLY polynomial via the substitution $a=1$ and $z=t^{1/2}-t^{-1/2}$. Under this substitution, all terms in (\ref{eqn:homflyformula}) with $j> 0$ vanish. Therefore, we have that
    \begin{align*}
        \Delta(Q) &= \sum\limits_{i \geq 0} c_{i, 0}(Q) (t^{1/2}-t^{-1/2})^{n-2i}\\
        &= \sum\limits_{i \geq 0} c_{i, 0}(Q) (t^{-1/2})^{n-2i}(t-1)^{n-2i}\\
        &=t^{-n/2}\cdot\sum\limits_{i \geq 0} c_{i, 0}(Q) t^i(t-1)^{n-2i}
    \end{align*}
    Now, we note that $c_{i,0}(Q)=b_{i}(Q)$ for all $i \geq 0$. In particular, consider an independent set of size $i$ with $i-0=i$ associated parent vertices. If one draws an edge between each vertex in the independent set and its parent vertex, the result is a set of $i$ edges in $Q$ which do not touch. This induces a bijection between the sets counted by $c_{i,0}(Q)$ and those counted by $b_{i}(Q)$.
    \end{proof}

    \begin{exmp}
    Let $Q$ be the Dynkin diagram $E_6$. Then there are 5 ways to pick a single edge in $Q$, $5$ ways to pick two distinct edges which do not share a vertex, and 1 way to pick three distinct edges, none of which share any vertices; see Figure \ref{fig:E6 formula example}. Therefore the Alexander polynomial is
    \begin{align*}
        \Delta(E_6) &= t^{-3} \left(1\cdot t^0(1-t)^6 +5 \cdot t^1(1-t)^4+5\cdot t^2(1-t)^2 + 1 \cdot t^3(1-t)^0 \right)\\
        &= t^{-3}\left(t^6 - t^5 + t^3 - t + 1 \right).
    \end{align*}
    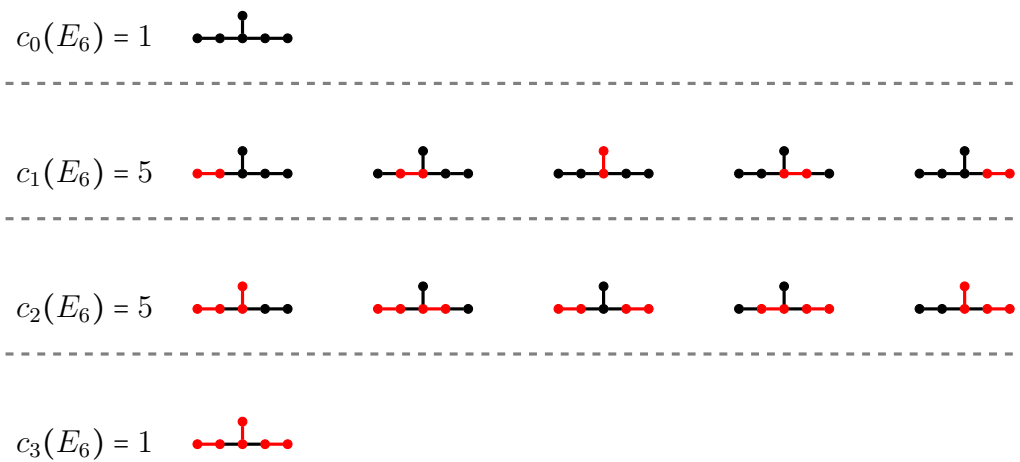
\begin{figure}
        \centering
        \begin{tikzpicture}[scale = 0.6]
        \begin{scope}
            \node () at (-2.5,0.5) {$c_0(E_6)=1$};
         \node [circle, draw=black, fill=black, scale =0.3] (1) at (0,0.5) {};
     \node [circle, draw=black, fill=black, scale =0.3] (2) at (0.5, 0.5) {};
     \node [circle, draw=black, fill=black, scale =0.3] (3) at (1, 0.5) {};
     \node [circle, draw=black, fill=black, scale =0.3] (4) at (1.5, 0.5) {};
     \node [circle, draw=black, fill=black, scale =0.3] (5) at (2, 0.5) {};
     \node [circle, draw=black, fill=black, scale =0.3] (6) at (1, 1) {};
     \draw [very thick] (1)--(2)--(3)--(4)--(5);
     \draw [very thick] (3)--(6);   
        \end{scope}
        \draw [very thick, gray, dashed] (-4.25, -0.5) -- (18.5, -0.5);
        \begin{scope}[yshift = -3cm]
        \begin{scope}
        \node () at (-2.5,0.5) {$c_1(E_6)=5$};
         \node [circle, draw=red, fill=red, scale =0.3] (1) at (0,0.5) {};
     \node [circle, draw=red, fill=red, scale =0.3] (2) at (0.5, 0.5) {};
     \node [circle, draw=black, fill=black, scale =0.3] (3) at (1, 0.5) {};
     \node [circle, draw=black, fill=black, scale =0.3] (4) at (1.5, 0.5) {};
     \node [circle, draw=black, fill=black, scale =0.3] (5) at (2, 0.5) {};
     \node [circle, draw=black, fill=black, scale =0.3] (6) at (1, 1) {};
     \draw [very thick, red] (1) -- (2);
     \draw [very thick] (2)--(3)--(4)--(5);
     \draw [very thick] (3)--(6);   
        \end{scope}

        \begin{scope}[xshift = 4cm]
         \node [circle, draw=black, fill=black, scale =0.3] (1) at (0,0.5) {};
     \node [circle, draw=red, fill=red, scale =0.3] (2) at (0.5, 0.5) {};
     \node [circle, draw=red, fill=red, scale =0.3] (3) at (1, 0.5) {};
     \node [circle, draw=black, fill=black, scale =0.3] (4) at (1.5, 0.5) {};
     \node [circle, draw=black, fill=black, scale =0.3] (5) at (2, 0.5) {};
     \node [circle, draw=black, fill=black, scale =0.3] (6) at (1, 1) {};
     \draw [very thick] (1)--(2);
     \draw [red, very thick] (2) -- (3);
     \draw [very thick] (3)--(4)--(5);
     \draw [very thick] (3)--(6);   
        \end{scope}

        \begin{scope}[xshift = 8cm]
         \node [circle, draw=black, fill=black, scale =0.3] (1) at (0,0.5) {};
     \node [circle, draw=black, fill=black, scale =0.3] (2) at (0.5, 0.5) {};
     \node [circle, draw=red, fill=red, scale =0.3] (3) at (1, 0.5) {};
     \node [circle, draw=black, fill=black, scale =0.3] (4) at (1.5, 0.5) {};
     \node [circle, draw=black, fill=black, scale =0.3] (5) at (2, 0.5) {};
     \node [circle, draw=red, fill=red, scale =0.3] (6) at (1, 1) {};
     \draw [very thick] (1)--(2)--(3)--(4)--(5);
     \draw [very thick, red] (3)--(6);   
        \end{scope}

        \begin{scope}[xshift = 12cm]
         \node [circle, draw=black, fill=black, scale =0.3] (1) at (0,0.5) {};
     \node [circle, draw=black, fill=black, scale =0.3] (2) at (0.5, 0.5) {};
     \node [circle, draw=red, fill=red, scale =0.3] (3) at (1, 0.5) {};
     \node [circle, draw=red, fill=red, scale =0.3] (4) at (1.5, 0.5) {};
     \node [circle, draw=black, fill=black, scale =0.3] (5) at (2, 0.5) {};
     \node [circle, draw=black, fill=black, scale =0.3] (6) at (1, 1) {};
     \draw [very thick] (1)--(2)--(3);
     \draw [red, very thick] (3) -- (4);
     \draw [very thick] (4)--(5);
     \draw [very thick] (3)--(6);   
        \end{scope}

        \begin{scope}[xshift = 16cm]
         \node [circle, draw=black, fill=black, scale =0.3] (1) at (0,0.5) {};
     \node [circle, draw=black, fill=black, scale =0.3] (2) at (0.5, 0.5) {};
     \node [circle, draw=black, fill=black, scale =0.3] (3) at (1, 0.5) {};
     \node [circle, draw=red, fill=red, scale =0.3] (4) at (1.5, 0.5) {};
     \node [circle, draw=red, fill=red, scale =0.3] (5) at (2, 0.5) {};
     \node [circle, draw=black, fill=black, scale =0.3] (6) at (1, 1) {};
     \draw [very thick] (1)--(2)--(3)--(4);
     \draw [red, very thick] (4) -- (5);
     \draw [very thick] (3)--(6);   
        \end{scope}
        \draw [very thick, gray, dashed] (-4.25, -0.5) -- (18.5, -0.5);
        \end{scope}

    \begin{scope}[ yshift = -6cm]
            \begin{scope}
            \node () at (-2.5,0.5) {$c_2(E_6)=5$};
         \node [circle, draw=red, fill=red, scale =0.3] (1) at (0,0.5) {};
     \node [circle, draw=red, fill=red, scale =0.3] (2) at (0.5, 0.5) {};
     \node [circle, draw=red, fill=red, scale =0.3] (3) at (1, 0.5) {};
     \node [circle, draw=black, fill=black, scale =0.3] (4) at (1.5, 0.5) {};
     \node [circle, draw=black, fill=black, scale =0.3] (5) at (2, 0.5) {};
     \node [circle, draw=red, fill=red, scale =0.3] (6) at (1, 1) {};
     \draw [red, very thick] (1)--(2);
     \draw [very thick] (2)--(3)--(4)--(5);
     \draw [red, very thick] (3)--(6);   
        \end{scope}

        \begin{scope}[xshift = 4cm]
         \node [circle, draw=red, fill=red, scale =0.3] (1) at (0,0.5) {};
     \node [circle, draw=red, fill=red, scale =0.3] (2) at (0.5, 0.5) {};
     \node [circle, draw=red, fill=red, scale =0.3] (3) at (1, 0.5) {};
     \node [circle,draw=red, fill=red, scale =0.3] (4) at (1.5, 0.5) {};
     \node [circle, draw=black, fill=black, scale =0.3] (5) at (2, 0.5) {};
     \node [circle, draw=black, fill=black, scale =0.3] (6) at (1, 1) {};
     \draw [red, very thick] (1)--(2);
     \draw [very thick] (2)--(3);
     \draw [red, very thick] (3)--(4);
     \draw [very thick] (4)--(5);
     \draw [very thick] (3)--(6);   
        \end{scope}

        \begin{scope}[xshift = 8cm]
         \node [circle, draw=red, fill=red, scale =0.3] (1) at (0,0.5) {};
     \node [circle, draw=red, fill=red, scale =0.3] (2) at (0.5, 0.5) {};
     \node [circle, draw=black, fill=black, scale =0.3] (3) at (1, 0.5) {};
     \node [circle, draw=red, fill=red, scale =0.3] (4) at (1.5, 0.5) {};
     \node [circle, draw=red, fill=red, scale =0.3] (5) at (2, 0.5) {};
     \node [circle, draw=black, fill=black, scale =0.3] (6) at (1, 1) {};
     \draw [red, very thick] (1) -- (2);
     \draw [very thick] (2)--(3)--(4);
     \draw [red, very thick] (4)--(5);
     \draw [very thick] (3)--(6);   
        \end{scope}

        \begin{scope}[xshift = 12cm]
         \node [circle, draw=black, fill=black, scale =0.3] (1) at (0,0.5) {};
     \node [circle, draw=red, fill=red, scale =0.3] (2) at (0.5, 0.5) {};
     \node [circle, draw=red, fill=red, scale =0.3] (3) at (1, 0.5) {};
     \node [circle, draw=red, fill=red, scale =0.3] (4) at (1.5, 0.5) {};
     \node [circle, draw=red, fill=red, scale =0.3] (5) at (2, 0.5) {};
     \node [circle, draw=black, fill=black, scale =0.3] (6) at (1, 1) {};
     \draw [very thick] (1)--(2);
     \draw [red, very thick] (2)--(3);
     \draw [very thick] (3)--(4);
     \draw [red, very thick] (4)--(5);
     \draw [very thick] (3)--(6);   
        \end{scope}

        \begin{scope}[xshift = 16cm]
         \node [circle, draw=black, fill=black, scale =0.3] (1) at (0,0.5) {};
     \node [circle, draw=black, fill=black, scale =0.3] (2) at (0.5, 0.5) {};
     \node [circle,draw=red, fill=red, scale =0.3] (3) at (1, 0.5) {};
     \node [circle, draw=red, fill=red, scale =0.3] (4) at (1.5, 0.5) {};
     \node [circle, draw=red, fill=red, scale =0.3] (5) at (2, 0.5) {};
     \node [circle, draw=red, fill=red, scale =0.3] (6) at (1, 1) {};
     \draw [very thick] (1)--(2)--(3)--(4);
     \draw [red, very thick] (4)--(5);
     \draw [red, very thick] (3)--(6);   
        \end{scope}
        \draw [very thick, gray, dashed] (-4.25, -0.5) -- (18.5, -0.5);
    \end{scope}

    \begin{scope}[ yshift = -9cm]
            \begin{scope}
            \node () at (-2.5,0.5) {$c_3(E_6)=1$};
         \node [circle, draw=red, fill=red, scale =0.3] (1) at (0,0.5) {};
     \node [circle, draw=red, fill=red, scale =0.3] (2) at (0.5, 0.5) {};
     \node [circle, draw=red, fill=red, scale =0.3] (3) at (1, 0.5) {};
     \node [circle, draw=red, fill=red, scale =0.3] (4) at (1.5, 0.5) {};
     \node [circle, draw=red, fill=red, scale =0.3] (5) at (2, 0.5) {};
     \node [circle, draw=red, fill=red, scale =0.3] (6) at (1, 1) {};
     \draw [red, very thick] (1)--(2);
     \draw [very thick] (2)--(3)--(4);
     \draw [red, very thick] (4)--(5);
     \draw [red, very thick] (3)--(6);   
        \end{scope}
    \end{scope}
 \end{tikzpicture}
        \caption{The sets counted by the coefficients in (\ref{eqn:alexanderformula}) when $Q$ is $E_6$.}
        \label{fig:E6 formula example}
    \end{figure}
\end{exmp}

\begin{exmp}
    Fix $n \geq 4$. Recall from Example \ref{ex: star HOMFLY} that $S_n$ is the star graph on $n$ vertices which has one vertex of degree $n-1$ connected to $n-1$ leaves. Any edge in $S_n$ must have the degree $n-1$ vertex as one of its endpoints, so it is impossible to pick multiple edges in $S_n$ which do not share an endpoint. It follows that 
    \begin{align*}
        \Delta(S_n) &= \frac{(-1)^n}{t^{n/2}}\left((1-t)^n + (n-1)\cdot t(1-t)^{n-2} \right)\\
        &=  \frac{(-1)^n}{t^{n/2}}\left(\left((1-t)^2+(n-1)t\right)(1-t)^{n-2} \right)\\
        &=  \frac{(-1)^n}{t^{n/2}}\left(\left(1-(n-3)t+t^2\right)(1-t)^{n-2} \right).
    \end{align*}
\end{exmp}

    We now make some observations about the related Alexander-Conway polynomial $\nabla(Q)$, following the work of Stoimenow in \cite{STOIMENOW2021105487} where he studied the Alexander-Conway polynomial for positive tree plumbing links. Given such a link $L$ with a plumbing tree $T$ whose line graph is $\Lambda$, Stoimenow made the following observation, which he credited to S. Baader:
    \begin{align}\label{stoi formula}
        \nabla(L) = \sum\limits_{i \geq 0} a_{i}(\Lambda) z^{n-2i}
    \end{align}
    where $n$ is the number of vertices in $T$. 
    Stoimenow then proves the log-concavity of $\nabla(L)(\sqrt{z})$ using the fact that it is a positive polynomial and has real roots.
    \begin{thm}
        (\cite{STOIMENOW2021105487} Theorem 4.1) If $L$ is a positive tree plumbing link, then all roots of $\nabla(L)(\sqrt{z})$ are real (and so $\nabla(L)(\sqrt{z})$ is log-concave).
        \label{thm: stoi log concave}
    \end{thm}
    
    We note that the formula (\ref{stoi formula}) agrees with the formula for $\nabla(T)$ for any tree quiver $T$ given by Corollary \ref{cor:alexander}. In particular, we see using the proof of the corollary that
    \begin{align*}
        \Delta(T) &= \sum\limits_{i \geq 0} c_{i, 0}(T) (t^{1/2}-t^{-1/2})^{n-2i}\\
        &= \sum\limits_{i \geq 0} b_i(T) z^{n-2i}\\
        &=\sum\limits_{i \geq 0} a_i(\Lambda) z^{n-2i}.
    \end{align*}
    Stoimenow's result generalizes to the Alexander-Conway polynomials of forest quivers since these polynomials are also positive and, as products of the Alexander-Conway polynomials of each of the quiver's connected components, also have real roots.
    \begin{cor}
        If $Q$ is a forest quiver, the Alexander-Conway polynomial $\nabla(Q)(\sqrt{z})$ is log-concave.
    \end{cor}
\printbibliography
\end{document}